\documentclass[10pt,reqno]{article}

\usepackage{amsmath,amsthm,amsfonts,amssymb,bm}

\numberwithin{equation}{section}

\newcommand{\als}{\al^*}
\newcommand{\orho}{\overline{\rho}}
\newcommand{\dbk}{\del_{\beta_k}}
\newcommand{\cZ}{\mathcal{Z}}
\newcommand{\vp}{\varphi}

\newcommand{\N}{{\mathbb N}}

\newcommand{\eps}{{\varepsilon}}

\newcommand{\C}{{\mathbb C}}
\newcommand{\R}{{\mathbb R}}

\newcommand\cA{{\cal  A}}
\newcommand\cB{{\cal  B}}
\newcommand\cD{{\cal  D}}
\newcommand\cU{{\cal  U}}

\newcommand\cK{{\cal  K}}

\newcommand\cN{{\cal  N}}

\newcommand\cO{{\cal O}}

\newcommand\cY{{\cal Y}}
\newcommand\cM{{\mathcal M}}

\newcommand\cS{{\mathcal S}}

\newcommand{\RN}[1]{%
  \textup{\uppercase\expandafter{\romannumeral#1}}
}

\newcommand{\al}{\alpha}
\newcommand{\be}{\beta}
\newcommand{\ga}{\gamma}
\newcommand{\Ga}{\Gamma}
\newcommand{\ka}{\kappa}

\newcommand{\dal}{\partial_{\alpha_l}}

\newcommand{\del}{\partial}

\def\eps{\epsilon }

\newcommand\adots{\mathinner{\mkern2mu\raise1pt\hbox{.}
\mkern3mu\raise4pt\hbox{.}\mkern1mu\raise7pt\hbox{.}}}

\renewcommand{\div}{{\rm div}}
\newcommand{\curl}{{\rm curl}}
\newcommand{\na}{{\nabla}}

\newcommand{\la}{\langle }
\newcommand{\ra}{\rangle }

\newcommand\br{\begin{remark}}
\newcommand\er{\end{remark}}
\newcommand\bp{\begin{pmatrix}}
\newcommand\ep{\end{pmatrix}}
\newcommand\beq{\begin{equation}}
\newcommand\ee{\end{equation}}
\newcommand\ba{\begin{equation}\begin{aligned}}
\newcommand\ea{\end{aligned}\end{equation}}

\newcommand\nn{\nonumber}

\newcommand{\bap}{\begin{app}}
\newcommand{\eap}{\end{app}}
\newcommand{\begs}{\begin{exams}}
\newcommand{\eegs}{\end{exams}}
\newcommand{\beg}{\begin{example}}
\newcommand{\eeg}{\end{exaplem}}
\newcommand{\bpr}{\begin{proposition}}
\newcommand{\epr}{\end{proposition}}
\newcommand{\bt}{\begin{theorem}}
\newcommand{\et}{\end{theorem}}
\newcommand{\bc}{\begin{corollary}}
\newcommand{\ec}{\end{corollary}}
\newcommand{\bl}{\begin{lem}}
\newcommand{\el}{\end{lem}}
\newcommand{\bd}{\begin{definition}}
\newcommand{\ed}{\end{definition}}
\newcommand{\brs}{\begin{remarks}}
\newcommand{\ers}{\end{remarks}}

\newtheorem{theo}{Theorem}[section]
\newtheorem{prop}[theo]{Proposition}
\newtheorem{cor}[theo]{Corollary}
\newtheorem{lem}[theo]{Lemma}
\newtheorem{defn}[theo]{Definition}

\newtheorem{rem}[theo]{Remark}

\newtheorem{exams}[theo]{Examples}

\newtheorem{assumption}[theo]{Assumption}
\newtheorem{definition}[theo]{Definition}

\newtheorem*{propnn}{Proposition}
\newtheorem*{theonn}{Theorem}

\author{Matthew Hernandez}

\title{Mechanisms of Lagrangian Analyticity in Fluids}

\begin{document}

\maketitle

\begin{abstract}
Certain systems of inviscid fluid dynamics have the property that for solutions that are only slightly better than differentiable in Eulerian variables, the corresponding Lagrangian trajectories are analytic in time. We elucidate the mechanisms in fluid dynamics systems that give rise to this automatic Lagrangian analyticity, as well as mechanisms in some particular fluids systems which prevent it from occurring.

We give a conceptual argument for a general fluids model which shows that the fulfillment of a basic set of criteria results in the analyticity of the trajectory maps in time. We then apply this to the incompressible Euler equations to prove analyticity of trajectories for vortex patch solutions. We also use the method to prove the Lagrangian trajectories are analytic for solutions to the pressureless Euler-Poisson equations, for initial data with moderate regularity.

We then examine the compressible Euler equations, and find that the finite speed of propagation in the system is incompatible with the Lagrangian analyticity property. By taking advantage of this finite speed we are able to construct smooth initial data with the property that some corresponding Lagrangian trajectory is not analytic in time. We also study the Vlasov-Poisson system on $\R^2\times \R^2$, uncovering another mechanism that deters the analyticity of trajectories. In this instance, we find that a key nonlocal operator does not preserve analytic dependence in time. For this system we can also construct smooth initial data for which the corresponding solution has some non-analytic Lagrangian trajectory. This provides a counterexample to Lagrangian analyticity for a system in which there is an infinite speed of propagation, unlike the compressible Euler equations.

\end{abstract}

\tableofcontents








\begin{section}{Introduction}

In previous studies of several systems of PDE in fluid mechanics, it has been observed that when one imposes only a small amount of regularity in the initial data, the trajectories of fluid particles for the corresponding system are in fact analytic in the time variable while the solution exists. The goal of this paper is to illuminate the underlying mechanisms in physical systems that lead to this automatic analyticity of Lagrangian trajectories, as well as those that prevent it from occurring.

In general, one can examine this property for certain PDE involving a velocity field, a time-dependent vector function $u(x,t)$ on $\R^d\times\R$, as an unknown in the system. The corresponding trajectory map $X(\al,t)$ then satisfies the ODE system
\begin{eqnarray}\label{intro01}
\frac{dX}{dt}(\al,t)&=&u(X(\al,t),t),\\
X(\al,0)&=&\al,\nn
\end{eqnarray}
where $u(x,t)$ satisfies the original PDE one started with.

Writing the original PDE in terms of $u(X(\al,t),t)$, rather than $u(x,t)$, (and making the analogous replacements for other potential unknowns in the original system) yields what is often referred to as the PDE written in Lagrangian variables. The intuitive idea is that instead of thinking of the unknown quantities in the system as functions of fixed points in space $x\in\R^d$ (and time), one can think of them as functions with various particles (and time) as their inputs, in a sense, where the particles are labelled by their starting positions $\al\in\R^d$. More concretely, the composition of a function $f$ with $X(\al,t)$ transforms the function $f$ into one which takes an initial position $\al\in\R^d$ as the input, and returns the value of $f$ obtained as one follows the associated particle and samples $f$ at the particle's location.

For certain types of equations, by using the original PDE, the evolution of the tangent of a trajectory, given by $u(X(\al,t),t)$, can be described in terms of a nonlocal operator acting on the initial velocity field $u_0=u|_{t=0}$ and the trajectory map $X$, and the equation \eqref{intro01} for $X$ can be expressed as
\ba\label{intro02}
\frac{dX}{dt}(\al,t)=F[X,u_0](\al,t),
\ea
where the operator $F$ may also depend on the initial data for other potential unknowns in the original PDE. The fact that there is no dependence on $u(x,t)$ at times $t>0$ in the right hand side of \eqref{intro02} may allow one to solve for the trajectory map $X$ without knowledge of the solution $u$ to the original PDE at later times. In many such PDE one observes that $u$ and the other unknowns in the original equation are determined by the solution $X$ to the equation \eqref{intro02} with initial data $X|_{t=0}=Id$.

Here, in particular, we study systems of PDE that have such a Lagrangian formulation. For example, we explicitly discuss examples such as vortex patch solutions for the $2$D incompressible Euler equations, the Euler-Poisson equation, and the Vlasov-Poisson equation. We also study the compressible Euler equations, which are more naturally expressed in Lagrangian variables as a PDE. Our objectives are to develop and apply a conceptual method for testing that the Lagrangian trajectories of a given fluid mechanics model are automatically analytic, even when the velocity field is somewhat nonsmooth, and to provide counterexamples, that is, to discuss fluid models which have solutions whose velocity fields are smooth, but whose Lagrangian trajectories need not be analytic. To the best of the author's knowledge, the counterexample-type results in this paper to the Lagrangian analyticity property are the first in the literature.

To verify that a fluid mechanics system has analytic Lagrangian trajectories, we use a natural approach of complexifying time in \eqref{intro02}, and considering the system as an ODE in a function space. The idea of proving analyticity of Lagrangian paths for the incompressible Euler equations by viewing them as an ODE with complex time appears in the articles \cite{serfatiT,serfati} of Serfati and in his unpublished works \cite{serfati3,serfati4,serfati5} (see also the careful development of the ideas in \cite{serfati} in the master's thesis of T. Hertel \cite{hertel}). In particular, \cite{serfatiT} and \cite{serfati3} discuss analyticity in the context of vortex patches. Results showing analyticity of trajectories for solutions to fluids equations amidst low spatial regularity have also been discussed in \cite{bf, cvw, fz,shnirelman}. These studies treat cases in which the initial data is slightly better than $C^1$. 
In \cite{fz}, a recursive formula for the time derivatives of the trajectory map is used to show analyticity for the trajectories of solutions to the $3$D incompressible Euler equations on the torus.  In \cite{cvw}, the $2$D SQG equations are considered; here the authors also use a recursive formula and then apply combinatorial identities to calculate and bound the $n$th order time derivatives of the trajectory map. In \cite{shnirelman}, a more abstract setting is used to prove the result for the incompressible Euler equations in arbitrary dimensions, where the flows of the solutions are handled as geodesics in the space of volume-preserving diffeomorphisms.

In this paper, in Section 2, we develop a systematic technique based on a complex ODE description which boils proofs of the Lagrangian analyticity property down to checking a few simple properties, and which can be applied to various fluid models. We prove that a general system has analytic Lagrangian trajectories if one is able to write it in the form \eqref{intro02} and then verify that two basic criteria hold. The first criterion is that for the particular choice of initial data $u_0$, the operator $F[\cdot,u_0]$ is locally bounded, in some sense, when acting on functions in a suitable function space. The second criterion is that the operator preserves analytic dependence in time when one inputs an analytic function of time $X(\al,t)$ in the corresponding argument.

In Section 3, as our first application of the material in Section 2, we prove Lagrangian analyticity for vortex patch solutions to the $2$D incompressible Euler equations, a case in which the gradient of the velocity field has discontinuities, unlike the results in the studies \cite{bf, cvw, fz, serfati,shnirelman}, in which the initial data is assumed to be slightly better than differentiable everywhere.

In the section that follows, we use the method to prove the analyticity of trajectories for the pressureless Euler-Poisson system for initial data with moderate spatial regularity. This system models stellar dynamics, in which particles experience a long-range attractive force, as well as the dynamics of charged fluids, in which the particles experience repulsion, when the effects of pressure are negligible.
One new obstacle over the system handled in the previous section is that the flows for the Euler-Poisson system are compressible. For this result, we provide formulas for taking spatial derivatives of the quantity $F[X,u_0]$ to arbitrary order and estimates bounding the $L^2$ norms of the results in a uniform way. While the result for vortex patch solutions to the $2$D incompressible Euler equations in Section 3 uses techniques more specific to the setting there, and is likely sharper in terms of the regularity of the initial data, the style of proof in Section 4 used for the Euler-Poisson system, in terms of Sobolev spaces, seems more robust and likely applies formulaically with almost no changes to most fluid equations with Lagrangian form \eqref{intro02} when $F$ satisfies the basic criteria discussed in Section 2.

In Section 5, we consider the compressible Euler equations, which turn out to be quite different from the incompressible Euler equations with regard to the analyticity of trajectories. We construct $C^\infty$ initial data which has trajectories that are not real analytic in time at $t=0$. This gives a counterexample to the Lagrangian analyticity property for a fluid mechanics system in the sense that no matter how smooth one requires the initial data to be, we find that the corresponding Lagrangian trajectories need not be analytic. We show that the finite speed of propagation for the compressible Euler equations is a key mechanism which is incompatible with the analyticity of the trajectories. This idea results in a relatively transparent counterexample of a fluid equation in which smooth data can easily lead to non-analytic trajectories.

Finally, in Section 6, we consider the Lagrangian formulation for the Vlasov-Poisson equation, in which the unknown $f(x,v,t):\R^2\times\R^2\times\R\to \R$ is a probability distribution governing the velocities and spatial positions of a large number of particles under the effects of either a force of mutual attraction or a repulsive force.
The natural Lagrangian formulation associated with this problem is in terms of a phase space trajectory $Z(t)=(X(t),V(t))$ in $\R^2\times\R^2$, say with $Z(0)=(x_0,v_0)$, where  $Z(t)$ satisfies an ODE of the form \eqref{intro02}. There are several interesting features the Vlasov-Poisson system shares with the incompressible Euler equations. For one, $f$ is conserved along trajectories, just as vorticity is in the $2$D incompressible Euler equations. Secondly the phase space flow for Vlasov-Poisson is also volume-preserving. Third, both systems have an infinite speed of propagation. In this case we found that, despite these similarities, the situation is very different with regard to the analyticity of trajectories. For the Vlasov-Poisson system, we are also able to construct examples of $C^\infty$ initial data with the property that some resulting Lagrangian trajectory is not analytic in time. A key difference for the Vlasov-Poisson system from incompressible Euler and Euler-Poisson is that the corresponding operator to $F[\cdot,u_0]$ in \eqref{intro02} does not preserve the analyticity of time dependence in this case.

Overall, we hope that the work here lucidly explains the essential reasons some fluid mechanics systems do have the Lagrangian analyticity property, and certain fluids mechanics systems do not.

\end{section}
\begin{section}{A general ODE method for proving Lagrangian analyticity}
In this section, we provide a strategy for proving Lagrangian analyticity results for fluid mechanics systems. One can often exploit the structure of such fluids equations, such as the incompressible Euler equations, to represent their corresponding trajectory maps as solutions of a relatively tractable ODE in an abstract function space. The main result in this section is a proof that if one has an ODE in an abstract function space, say for the trajectory maps, which is of a particular form, then the solutions are analytic in time. This results in a conceptual method for proving Lagrangian analyticity. Indeed, criteria leading to existence and uniqueness, as well as analyticity, of solutions to ODE are not overly technical, so the core parts of our strategy are clear-cut and easily summarized.

The first step is to reformulate the PDE in Lagrangian variables in the manner described in the introduction. Let us take the $\R^d$-valued function $u(x,t)$ on $\R^d\times\R$ to be the velocity field associated to the original PDE. Then the trajectory map $X(\al,t)$ is the solution to the equation
\begin{eqnarray}
\frac{dX}{dt}(\al,t)&=&u(X(\al,t),t),\\
X(\al,0)&=&\al.\nn
\end{eqnarray}
For this section, we assume that one is able to rewrite this system in the form
\begin{eqnarray}
\frac{dX}{dt}&=&\tilde F(X),\label{eq00001}\\
X|_{t=0}&=&Id,\nn
\end{eqnarray}
for an operator $\tilde F$ defined on a set of maps from $\R^d$ to $\C^d$, depending on $u_0=u|_{t=0}$ and the other initial data prescribed for the original PDE. We give explicit derivations for the corresponding operators $\tilde F$ in the case of the incompressible Euler equations in Section 3 and in the case of the Euler-Poisson equation in Section 4.

We have yet to discuss some simple matters such as the domain and range of the operator $\tilde F$ in \eqref{eq00001}. We now address this, as well as give a complete discussion on the structure required for our strategy.
\begin{subsection}{The basic structure}
Let us consider an ODE of the form
\ba\label{thesis001}
\frac{dX}{dt}&=&\tilde F (X),\quad
X =Id\quad\mbox{at }t=0,
\ea
for unknown $X(t)$, with $(X(t)-Id)\in Y$ for each $t$, where $Y$ is a Banach space of maps from a fixed open subset $\Omega\subset\R^d$ to $\C^d$, for some positive integer $d$. Moreover, we require that $\lVert X(t)-Id\rVert_Y \leq \delta$, for some $\delta>0$, and that the operator $\tilde F$ is defined on the set $\{X:\Omega\to\C^d\,:\lVert X-Id\rVert_Y \leq \delta\}$, taking values in $Y$. 
We work under the assumption that $Y$ is a Banach space of continuous maps with a norm that dominates the $C^{0,1}(\Omega)$ norm, i.e. with the property that
\ba
\lVert
f
\rVert_{C^{0,1}(\Omega)}
=\sup_{x\in\Omega}|f(x)|+\sup_{x,y\in\Omega}\frac{|f(x)-f(y)|}{|x-y|}\leq C \lVert f \rVert_Y.
\ea
In this section we give criteria for the operator $\tilde F$ which are relatively simple and which guarantee that solutions $X(t)$ are analytic in time in a small disc in $\C$ as maps into $Y$, with Theorem \ref{mainth}. We now give an overview of these criteria.

\textbf{A summary of the key properties of $\tilde F$ that lead to analytic trajectories}
\begin{itemize}
\item
\textbf{Local boundedness}
\end{itemize}

The first criterion is that $\tilde F$ satisfies a local boundedness property near the identity map. To say this is to assert that for a fixed $\delta>0$, $\tilde F$ sends maps $X$ in the closed $\delta$-ball about the identity in $Y$ into $Y$, satisfying a bound of the form
\ba\label{0000'}
\sup\{\lVert \tilde F(X) \rVert_{Y}:\lVert X-Id\rVert_{Y}\leq\delta\} \leq C,
\ea
for some $C$, possibly dependent on the initial data under consideration.

We present Proposition \ref{bddF}, proved in Section 3, which gives us this property for the corresponding operator for vortex patch solutions to the $2$D incompressible Euler equations. In this case, we take $Y$ to be the set of $C^{1,\mu}$ maps on $\Omega$ for some $\mu\in(0,1)$, where $\Omega$ is the initial patch.

\begin{propnn}
Fix $\mu \in(0,1)$. Assume we have a finitely-connected, bounded domain $\Omega$ which has $C^2$ non-self-intersecting boundary, and let $F$ be as in Definition \ref{F2}, with 
\ba
\omega(\al):=\begin{cases}
1 & \mbox{if }\al\in\Omega, \\
0 & \mbox{otherwise.}
\end{cases}
\ea
Then there exists a $\delta>0$ such that
\ba
\sup\{\lVert F(X) \rVert_{C^{1,\mu}(\Omega)}:\lVert X-Id\rVert_{C^{1,\mu}(\Omega)}\leq\delta\} \leq C,
\ea
where $C$ depends only on $\Omega$.
\end{propnn}

For the pressureless Euler-Poisson system, in Section 4, we give a proof of the corresponding bound for the analogous operator for $H^s$ maps from $\R^3$ to $\C^3$, $s\geq6$, with Proposition \ref{pcor02}.

The analyticity of trajectories also relies critically on the analytic properties of the operator $\tilde F$, which are related to the remaining major criterion for analytic trajectories.

\begin{itemize}
\item
\textbf{Preservation of analytic time dependence}
\end{itemize}

Now we consider the composition of $\tilde F$ with maps dependent on an additional, complex parameter $z$, which plays the role of the time variable. Precisely, we compose it with a function of $z$ in a complex disc, $X_z:\Omega\to\C^d$ with $(X_z-Id)\in Y$ for each $z$. Let us denote $d_r:=\{z\in\C:|z|<r\}$ for $r>0$. The desired property of $\tilde F$ is that the following implication holds for any such $X_z$ mapping into the domain of $\tilde F$ and any $r>0$:
\begin{eqnarray}\label{thesis002}
& &\mbox{If for each }\al\in\Omega\mbox{ we have }X_z(\al):d_r\to\C^d\mbox{ is analytic in }z,\\
& &\mbox{then for each }\al\in\Omega\mbox{ we have }F(X_z)(\al):d_r\to\C^d\mbox{ is analytic in }z.\nn
\end{eqnarray}

In this case we say the operator $\tilde F$ preserves analyticity ((ii) of Definition \ref{pres3}).
Such a property is quite natural if we expect to have analytic trajectories. Note that if $X(\al,t)$ solves \eqref{thesis001} and is analytic in $t$, $t$ playing the role of $z$, and we take the composition $\tilde F(X(\cdot,t))(\al)$, this returns the tangent vector $\frac{dX}{dt}(\al,t)$, which must be analytic in $t$. It is verified with Lemma \ref{vPreservesA'} that for vortex patch solutions to $2$D Euler, the corresponding operator $F$ preserves analyticity in this sense. In Section 4, we show the analogous result holds for the Euler-Poisson system.

With Theorem \ref{mainth} we prove that if $\tilde F$ is locally bounded near the identity map in $Y$ and preserves analyticity, then there is a unique analytic $Y$-valued function on a disc in $\C$ solving the equation \eqref{thesis001}.

We verify that both of the above criteria are satisfied for vortex patch solutions to the $2$D incompressible Euler equations in Section 3. We also prove the criteria are satisfied for the $3$D Euler-Poisson system in Section 4. By using Theorem \ref{mainth}, we thus get analyticity of trajectories for both systems.

Additionally, we examine some examples of fluid mechanics systems to which the above method of analytic trajectories does not apply. In particular, they are both examples which are related to the incompressible Euler equations, but we prove that they do not have the Lagrangian analyticity property, unlike incompressible Euler.
\begin{rem}
(i) For the compressible Euler equations, it appears that the Lagrangian formulation of the system is most naturally expressed as a second order PDE, rather than an ODE of the form \eqref{thesis001}. In this case, there exists $C^\infty$ initial data for which the solution has some trajectory which is not analytic in time. This is proved in Section 5.

(ii) The Vlasov-Poisson system has a natural Lagrangian formulation of the form \eqref{thesis001}, but the corresponding operator \textit{does not} satisfy the preservation of analyticity property, \eqref{thesis002}. For this system, we are also able to construct $C^\infty$ initial data such that the solution has some trajectory which is not analytic in time. This is done in Section 6.
\end{rem}
\end{subsection}
\begin{subsection}{Analytic solutions in time for the ODE in an abstract function space}
Now let us put the criteria discussed above on a firm footing.
We write $X_z$ to indicate a function dependent on a parameter $z\in\C$, with $(X_z-Id)$ taking values in a Banach space $Y$ of continuous maps from an open subset $\Omega$ of $\R^d$ to $\C^d$. As mentioned above, we make the assumption that the $\Vert\cdot\rVert_Y$ norm dominates the $C^{0,1}$ norm. 
\begin{assumption}\label{as1}
$Y\subset C(\Omega)$ is a Banach space of maps from $\Omega\subset\R^d$ to $\C^d$ with the property that
\ba\label{bprop}
\lVert f\rVert_{C^{0,1}(\Omega)}\leq C \lVert f \rVert_Y\textrm{ for all }f\in Y.
\ea
\end{assumption}
Moreover, we specifically consider maps within a distance $\delta$ of the identity map in the $Y$ norm, for some $\delta>0$.
\begin{definition}\label{banach'}
Given a Banach space $Y$ satisfying Assumption \ref{as1} we define (i)
\ba
B_\delta:=\{X:\lVert X-Id\rVert_Y\leq \delta\},
\ea
(ii)
\ba\label{24}
\cY_r:=\{Y\textrm{-valued maps }X_z\textrm{ analytic in }z\textrm{ on }d_r\},
\ea
where $d_r=\{|z|<r\}$, and (iii)
we consider the set of maps which differ from the identity by an element of $\cY_r$, and which map into the ball $B_\delta$ as functions of $z$:
\ba\label{26}
\cB_{\delta,r}:=\{\mbox{maps }X_{(\cdot)}:(X_{(\cdot)}-Id)\in\cY_r, \;X_z\in B_\delta \textrm{ for all }z\in d_r\}.
\ea
\end{definition}

With this, we now have a setting in which we can verify whether or not for each fixed $\al$, $\tilde F(X_z)(\al)$ inherits analyticity in $z$ from the functions $X_z$ analytic in $z$, that is, whether the property of preservation of analyticity holds.
\begin{defn}\label{pres3}
(i) For any function $\tilde F:B_\delta\to\{\C^d\textrm{-valued maps on }\Omega\}$, we define $\tilde F$ with domain $\cB_{\delta,r}$ in the compatible way, mapping $X_{(\cdot)}\in\cB_{\delta,r}$ to the function $\tilde F(X_{(\cdot)}):d_r\to\{\C^d\textrm{-valued maps on }\Omega\}$ defined by\footnote{Occasionally we use $\tilde F(X_z)$ to refer to the map $\tilde F(X_{(\cdot)})$.}
\ba
\tilde F(X_{(\cdot)})(z):=\tilde F(X_z)\quad\textrm{ for }z\in d_r.
\ea

(ii) We say that $\tilde F$ preserves analyticity if (for any $r>0$ for the disc of analyticity $d_r$) for any $X_{(\cdot)}\in\cB_{\delta,r}$, at each fixed $\al\in\Omega$, $\tilde F(X_z)(\al)$ is analytic in $d_r$.
\end{defn}

Note that so far we have only discussed how analyticity of $\tilde F (X_z)(\al)$ in $z$ in a disc $d_r$ for each fixed $\al\in\Omega$ is inherited from analyticity of $X_z(\al)$ in $z$ in $d_r$ for each $\al\in\Omega$. Now we discuss how to improve this to get the seemingly stronger conclusion that $\tilde F(X_z)$ is analytic in $z$ in $d_r$ as a map into the Banach space $Y$.
\begin{lem}\label{omegalem'}
Let $B_\delta$, $\cB_{\delta,r}$, $Y$, and $\cY_r$ be as in Definition \ref{banach'}. Consider a function $\tilde F:B_\delta\to Y$. If we have a bound of the form
\ba\label{65}
\sup_{X\in B_\delta} \lVert \tilde F(X) \rVert_Y\leq C<\infty,
\ea
and if $\tilde F$ preserves analyticity, then in fact for any $r>0$ as the radius of analyticity in the definition of $\cB_{\delta,r}$, for $X_z\in\cB_{\delta,r}$ we have that $\tilde F(X_z)$ is an analytic function from $d_r$ into $Y$. That is,
\ba\label{27}
\tilde F:\cB_{\delta,r}\to \cY_r.
\ea
\end{lem}
\begin{proof}
By the Cauchy integral formula, for each $\al\in\Omega$ and $z\in d_r$ we have
\ba\label{28}
\frac{d}{dz} \left( \tilde F(X_z)(\al)\right)=\frac{1}{2\pi i} \int_{|\zeta-z|=r'} \frac{ \tilde F(X_\zeta)(\al)}{(\zeta-z)^2}d\zeta,
\ea
for $r'<\textrm{dist}(z,\del d_r)$. Thus by applying $\lVert\cdot\rVert_Y$ to both sides we obtain
\ba\label{29}
\bigg\lVert \frac{d}{dz} \tilde F(X_z)  \bigg\rVert_Y\leq\frac{C'}{\textrm{dist}(z,\del d_{r})}
\sup_{|\zeta|<r}\lVert \tilde F(X_\zeta) \rVert_Y\leq \frac{C''}{\textrm{dist}(z,\del d_{r})},
\ea
by \eqref{65}, where in the left hand side we are considering the norm of the function mapping a given $\al$ to $\frac{d}{dz}\left( \tilde F(X_z)(\al)\right)$. Now it just remains to prove that as $h$ tends to zero, the difference quotient $(\tilde F(X_{z+h})-\tilde F(X_z))/h$ converges in the $Y$ norm to the function $\al\mapsto\frac{d}{dz} \left( \tilde F(X_z)(\al)\right)$. Consider $h$ of size much smaller than the distance from $z$ to $\del d_r$, and fix $r'>|h|$, with $r'$ smaller than the distances of both $z$ and $z+h$ to the boundary of $d_r$. Then
\begin{eqnarray}
\frac{1}{h}(\tilde F(X_{z+h})(\al)-\tilde F(X_z)(\al))
&=&\frac{1}{2\pi i h}
	\left(	\int_{|\zeta-z|=r'}
				\frac{\tilde F(X_\zeta)(\al)}{\zeta-(z+h)}d\zeta
				-
			\int_{|\zeta-z|=r'}
				\frac{\tilde F(X_\zeta)(\al)}{\zeta-z}d\zeta
	\right),\\
&=&
\frac{1}{2\pi i}
\int_{|\zeta-z|=r'}\frac{\tilde F(X_\zeta)(\al)}{(\zeta-(z+h))(\zeta-z)}d\zeta,\nn
\end{eqnarray}
and so we have
\begin{eqnarray}\\
\frac{1}{h}(\tilde F(X_{z+h})(\al)-\tilde F(X_z)(\al))-\frac{d}{dz}(\tilde F(X_z)(\al))
&=&
\frac{1}{2\pi i}\bigg(
		\int_{|\zeta-z|=r'}
			\frac{\tilde F(X_\zeta)(\al)}{(\zeta-(z+h))(\zeta-z)}d\zeta\nn\\
& &
\hspace{.7in}-\int_{|\zeta-z|=r'}
			\frac{\tilde F(X_\zeta)(\al)}{(\zeta-z)^2}d\zeta
		\bigg),\nn\\
&=&
\frac{h}{2\pi i}
\int_{|\zeta-z|=r'}
	\frac{\tilde F(X_\zeta)(\al)}{(\zeta-(z+h))(\zeta-z)^2}d\zeta.\nn
\end{eqnarray}
We apply the $Y$ norm to both sides to get the bound
\begin{eqnarray}
\bigg\lVert
\frac{1}{h}(\tilde F(X_{z+h})(\al)-\tilde F(X_z)(\al))-\frac{d}{dz}(\tilde F(X_z)(\al))
\bigg\rVert_Y
&\leq&
\frac{C' h}{(\textrm{dist}(z,\del d_{r}))^2}\sup_{|\zeta|<r}\lVert \tilde F(X_\zeta)\rVert_Y\\
&\leq& \frac{C'' h}{(\textrm{dist}(z,\del d_{r}))^2}.\nn
\end{eqnarray}
Since the right hand side is $O(|h|)$, the conclusion of the lemma follows.
\end{proof}

In the coming ODE argument, we will also need that the operator of concern is Lipschitz on some ball about the identity in $Y$. With the following lemma, we show this property also follows from the local boundedness property together with preservation of analyticity.

\begin{lem}\label{what'}
Consider a function $\tilde F:B_\delta\to Y$ and assume that we have a bound of the form
\ba\label{65prime'}
\sup_{X\in B_\delta} \lVert \tilde F(X) \rVert_Y\leq C,
\ea
and that $\tilde F$ preserves analyticity. For any $\eps\leq \delta/10$ we have that the function $\tilde F$ with domain $B_\eps$ is Lipschitz with constant $C_0$ depending only on $C$.
\end{lem}
\begin{proof}
Let us define $r_0:=\delta/4$. For distinct $X,Z\in B_\eps$, we may write
\ba\label{201}
\tilde F(X)-\tilde F(Z)=\int^{\lVert X-Z\rVert_Y}_0\frac{d}{d\tau}
\tilde F(X_\tau)d\tau,
\ea
where we define $X_\tau:=Z+\tau(X-Z)/\lVert X-Z\rVert_Y$ for any $\tau\in\{|z| < r_0\}$. Note this contains $[0,2\eps]$, and thus the interval $[0,\lVert X-Z\rVert_Y]$ integrated over in \eqref{201}. Now we check that we have a suitable bound on $\frac{d}{d\tau}(\tilde F(X_\tau))$, uniform in $\tau\in[0,\lVert X-Z\rVert_Y]$. Since $\eps$ and $r_0$ are each less than $\delta/2$, we know that for all $|\tau|<r_0$ we have $X_\tau \in B_\delta$. Since then $X_{(\cdot)}$, defined on the disc $\{|z|<r_0\}$, is in $\cB_{\delta,r_0}$, we know that $\tilde F(X_\tau)$ is analytic in $\tau$ in $\{|z|<r_0\}$ by Lemma \ref{omegalem'}, and moreover the bound \eqref{29} with $r_0$ in place of $r$ shows $\lVert\frac{d}{d\tau}\tilde F(X_\tau)\rVert_Y$ is bounded by some constant as $\tau$ varies in the smaller disc $\{|z|\leq 2\eps\}$. For this, we have used the fact that $r_0-2\eps\geq \delta/20>0$. We observe that the resulting constant is thus determined by $\delta$ and the bounding constant $C$ from \eqref{65prime'}. The argument is concluded by applying $\lVert \cdot \rVert_Y$ to both sides of \eqref{201} and bounding this above by moving the norm into the integrand.
\end{proof}
Now suppose $\tilde F$ satisfies the hypotheses of Lemma \ref{what'}, and let us pick a constant $\eps$ which satisfies the hypotheses of Lemma \ref{what'}:
\ba
\eps:=\frac{\delta}{10}.
\ea

By restricting the domain of such an $\tilde F$ to $B_\eps$, we ensure that it is Lipschitz with constant $C_0$, from Lemma \ref{what'}, which readies us for an application of a fixed point theorem.
Note that for any $r>0$ we have a natural choice of norm for $\cY_r$, with
\ba\label{40}
\lVert X_z \rVert_{\cY_r}:=\sup_{z\in d_r}\lVert X_z \rVert_Y.
\ea

Now we have established all the properties of $\tilde F$ that are required for solving the equation
\begin{eqnarray}\label{36'}
&\frac{d}{dt} X_t& = \tilde F(X_t),\\
&X_t |_{t=0}&=Id,\nonumber
\end{eqnarray}
in the desired setting. We prove that the solution to the equation exists as a fixed point of a contraction mapping on the space of functions analytic in time as maps into $Y$. For a given $\tilde F$ and $r>0$, we define the following mapping on functions $X_{(\cdot)}\in \cB_{\eps,r}$:
\ba\label{35}
\Phi_{\tilde F}\left(X_{(\cdot)}\right)(t):=Id+\int^t_0 {\tilde F}(X_\zeta)d\zeta\quad\quad\textrm{for }t\in d_r,
\ea
where the integration is taken along the straight path from $0$ to $t$. This is the map we will show has a fixed point that solves the ODE \eqref{36'}.
\begin{theo}\label{mainth}
Suppose that $\tilde F:B_\delta\to Y$ satisfies a bound of the form \eqref{65} 
and that $\tilde F$ preserves analyticity.
Then for $\eps=\delta/10$ and sufficiently small $r_1>0$, there exists a unique fixed point $X_{(\cdot)}\in\cB_{\eps,{r_1}}$ of $\Phi_{\tilde F}:\cB_{\eps,{r_1}}\to\cB_{\eps,{r_1}}$, and the function $X(\al,t)=X_t(\al)$ is analytic in $t$ on $d_{r_1}$ and the unique solution to the equation \eqref{36'}.
\end{theo}
\begin{proof}
It is easy to check that $\Phi_{\tilde F}$ maps $\cB_{\eps,{r_1}}$ to itself, given that $t$ is restricted to $d_{r_1}$ and $r_1$ is chosen small enough. The bound \eqref{65} allows us to ensure that $\Phi_{\tilde F}(X_{(\cdot)})$ is close to $Id$ for small $r_1$. 
Using Lemma \ref{what'}, we have the Lipschitz property of $\tilde F$ on $B_{\eps}$, and so we get that $\Phi_{\tilde F}$ is a contraction mapping on $\cB_{\eps,r_1}$ as long as $r_1<1/(2C_0)$. The conclusion of the theorem now follows from the contraction mapping principle.
\end{proof}

Having established the more general material of this section, for our proofs showing analyticity of Lagrangian trajectories for vortex patches for $2$D incompressible Euler, and for solutions to the Euler-Poisson system, it remains for us to verify that the local boundedness property and the preservation of analyticity property both hold for the corresponding $\tilde F$ operators. In the next section we do so for vortex patch solutions to the incompressible Euler equations.
\end{subsection}
\end{section}

\begin{section}{Lagrangian analyticity for the incompressible Euler equations}
The incompressible Euler equations on $\R^d$ are given by
\begin{eqnarray}\label{eulereqns}
\frac{\del u}{\del t}+u\cdot \nabla u +\na p = 0,\\
\na\cdot u = 0,\nonumber\\
u|_{t=0}=u_0,\nn
\end{eqnarray}
where $u(x,t)$ is a vector field defined on $\R^d\times\R$, taking values in $\R^d$, and the pressure $p(x,t)$ is an unknown scalar, with initial data $u_0$ to be prescribed, for example, as a function in $C^{1,\mu}$ for some $\mu\in(0,1)$, or as an initial velocity field corresponding to a vortex patch. For the results of this section, we are interested in the case $d=2$, the setting for vortex patches, but we comment that the core ideas of the method described in Section 2 for Lagrangian analyticity results do not depend significantly on the dimension under consideration.

In this section we show that vortex patch solutions to the $2$D Euler equations have analytic Lagrangian trajectories. Though we only give the details of Lagrangian analyticity for these kinds of solutions to incompressible Euler, the methods in this paper can also be applied to verify the analyticity of trajectories for velocity fields solving the $2$D and $3$D Euler equations as long as the vorticity is H\"{o}lder continuous and the solution is periodic or satisfies some mild decay condition.

Let us proceed by showing how one frames the $2$D incompressible Euler equations written in Lagrangian variables in the form of an ODE in an abstract function space, such as \eqref{thesis001}.
\begin{subsection}{The Lagrangian setting for the incompressible Euler equations}

The first thing to observe is that by taking the scalar curl of the first equation in \eqref{eulereqns} in the $2$D case, we arrive at the transport equation for the vorticity, $\omega=\del_1 u_2 - \del_2 u_1$.
\ba
\frac{\del \omega}{\del t} + u \cdot \na \omega = 0.
\ea
Thus the vorticity is transported by the flow, which is equivalently expressed by $\frac{d}{dt}\omega(X(\al,t),t)=0$, or
\ba\label{tr}
\omega(X(\al,t),t)=\omega_0(\al),
\ea
where the initial vorticity is denoted by $\omega_0(\al)=\omega(\al,0)$.

Our main tool is the Biot-Savart law, which allows us to recover the velocity vector $u(x,t)$ from the vorticity $\omega$, with
\ba\label{biotsavartlaw}
u(x,t)=\frac{1}{2\pi}\int\frac{(x-y)^\perp}{|x-y|^2}\omega(y,t)dy,
\ea
for the $2$D case. In this expression, we use the notation $y^\perp=(-y_2,y_1)$. Considering the Lagrangian point of view, we may use this to find the velocity vector tangent to a specified point along a path given by $X(\al,t)$. Inputting this function into the above expression and making a change of variables yields
\ba\label{000'''}
\frac{dX}{dt}(\al,t)=\frac{1}{2\pi}\int \frac{(X(\al,t)-X(\be,t))^\perp}{|X(\al,t)-X(\be,t)|^2}\omega_0(\be)d\be,
\ea
where we have taken into account \eqref{tr} along with the area preserving condition satisfied by incompressible flows, i.e. that $\det \na X(\al,t)=1$. We denote the kernel appearing here by $K$, with
\ba\label{01prime}
K(y)=\frac{1}{2\pi}\frac{(-y_2,y_1)}{y_1^2+y_2^2}.
\ea
Note that $K(y)$ analytically extends to complex pairs $(y_1,y_2)$ outside of the set $\{y_1\pm i y_2=0\}$. Considering the right hand side of \eqref{000'''}, inside of the integral, we have an expression with very nice dependence on the quantity $X(\al,t)$. The idea will be to regard the right hand side as a function $F$ of the map $X_t=X(\cdot,t):\R^2\to\R^2$, and extend the definition of $F$ to a well-behaved function defined on maps more generally from $\R^2$ to $\C^2$. Here the complex output is necessitated by our complexification of the time variable. In that formulation, the equation \eqref{000'''} becomes
\ba\label{000'}
\frac{d}{dt}X_t=F(X_t),
\ea
together with the initial data $X_t|_{t=0}=Id$.
\end{subsection}

\begin{subsection}{Lagrangian analyticity for vortex patch solutions}

Vortex patch solutions to the $2$D incompressible Euler equations are characterized by having vorticity $\omega(x,t)$ of the form
\ba
\omega(x,t)=\begin{cases}
0 & \mbox{if }x\in \Omega(t),\\
1 & \mbox{otherwise},
\end{cases}
\ea
where $\Omega(t)$ is a bounded region that is transported by the flow. If the vorticity is initially given by an identifier function on a region $\Omega$ with smooth boundary, the solution exists globally, and $\omega$ has this form at later times, due to the fact that vorticity is transported by the flow for the $2$D incompressible Euler equations. One has good existence and uniqueness properties and that the smoothness of the boundary $\del\Omega(t)$ is maintained in time (see, for example \cite{bc,chemin2,serfati2,yudovich}) despite the jump discontinuity in the gradient of the velocity across the boundary. Considering the question of the regularity in time of the particle trajectories, and thus the evolution of the set $\Omega(t)$, for these discontinuous vorticities is a natural progression from the corresponding question for H\"{o}lder continuous vorticity, studied in past studies of Lagrangian analyticity for the incompressible Euler equations such as \cite{bf, cvw, fz,serfati,shnirelman}. See Serfati \cite{serfatiT,serfati3} for the first work on Lagrangian analyticity in the context of vortex patches.

To show that the Lagrangian trajectories of vortex patches are analytic, we have the main task of proving the local boundedness result, Proposition \ref{bddF}, for the operator $F$ in this scenario, which is given in Definition \ref{F2}. In accordance with the remarks above, in this part we treat the case where the initial curl $\omega(\al):=\mbox{curl}\,u_0(\al)$ is of the form
\ba\label{final001}
\omega(\al)=\begin{cases}
1 & \mbox{if }\al\in\Omega,\\
0 & \mbox{otherwise},
\end{cases}
\ea
for a bounded open set $\Omega\subset \R^2$ satisfying some basic assumptions. We also show in this section that the operator $F$ then preserves analyticity, in the sense of (ii) of Definition \ref{pres3}. As a consequence, we are able to prove Theorem \ref{final}, which we summarize below.

\begin{theonn}
Assume that $\Omega$ is a bounded open set which is finitely-connected, and which has $C^2$ non-self-intersecting boundary. Then for the vortex patch solution to \eqref{eulereqns} with initial vorticity given by \eqref{final001}, with corresponding trajectory map $X$, we have for each $\al$ in $\overline{\Omega}$ that the trajectory $X(\al,t)$ is analytic in time.
\end{theonn}

Now we begin by rigorously defining the corresponding operator $F$ for the case of a vortex patch.
\begin{definition}\label{F2}
Consider a fixed function $\omega(\al)$ of the form \eqref{final001}, for some given bounded set $\Omega\subset\R^2$. We consider the space $C^{0,1}(\Omega)$ of $\C^2$-valued, Lipschitz continuous functions on $\Omega$. For a map $X\in C^{0,1}(\Omega)$ with $\lVert X-Id\rVert_{C^{0,1}(\Omega)}\leq\delta$, for a small fixed $\delta>0$ to be specified later, we define
\ba
F(X)(\al)=\int_{\Omega}
	K(X(\al)-X(\be)) d\be.
\ea
\end{definition}
In this definition $K$ is regarded as the complex vector valued function
\ba\label{01}
K(z)=\frac{1}{2\pi}\frac{(-z_2,z_1)}{z_1^2+z_2^2},\quad z_1\pm i z_2\neq 0.
\ea
Let us note that for $X$ sufficiently close to $Id$, say as in Definition \ref{F2}, the denominator of $K(X(\al)-X(\be))$ cannot vanish unless $\al=\be$. This follows from the fact that
\ba\label{closetoid'}
(X(\al)-X(\be))-(\al-\be)=(X-Id)(\al)-(X-Id)(\be)=O(|\al-\be|)O(\delta).
\ea
Here we comment that we use the convention that vectors indicating a spatial position or displacement, such as $X(\al)$ or $(\al-\be)$, are row vectors. Thus, as we have written it here, the $O(|\al-\be|)$ factor in the rightmost expression is also a row vector, and the $O(\delta)$ vector is a $2\times 2$ matrix acting on it.
From \eqref{closetoid'} we get
\begin{eqnarray}\label{015}
|(X_1(\al)-X_1(\be))^2+(X_2(\al)-X_2(\be))^2|
&\geq&
\left|(\al-\be)^2_1+(\al-\be)^2_2\right|
-C\delta|\al-\be|^2 ,\\
&\geq& |\al-\be|^2/2,\nonumber
\end{eqnarray}
for sufficiently small $\delta$. This computation leads us to
\begin{rem}
In the integral defining $F(X)(\al)$, we get something no more singular than $|\al-\be|^{-1}$, and so $F(X)$ gives us a finite $\C^2$-valued function on $\Omega$.
\end{rem}

Now that the operator $F$ has been defined, we consider its behavior with regard to the criterion of preservation of analytic dependence in time, discussed in Section 2.

\vspace{.1in}

\textbf{Preservation of analyticity for the operator $F$}

\vspace{.05in}

\noindent  With the goal of showing that the operator $F$ defined above preserves analyticity, we first present a bound on the derivatives of the kernel $K$ which is useful in the verification of this fact.
\begin{lem}\label{kernelbound}
There is a small positive constant $\delta'$ such that for a matrix $M\in\C^{2\times2}$ with $|M|\leq \delta'$, the function $K(z):\C^2\setminus \{z_1\pm i z_2=0\}\to \C^2$  defined by \eqref{01}, satisfies
\ba\label{Kbound'}
|(\del^\ga_z K)(\be(I+M))| &\leq& A_\ga|\be|^{-|\ga|-1}\quad\textrm{for all nonzero }\be\in\R^2\textrm{ and }|\ga|\geq 0.
\ea
\end{lem}
\begin{proof}
An easy calculation shows that the derivatives of $(\del^\ga_z K)(z)$ satisfy the bound
\ba\label{nmk}
|(\del^\ga_z K)(z)|\leq \frac{C_\ga|z|^{|\ga|+1}}{{\left|z^2_1+z^2_2\right|}^{|\ga|+1}}.
\ea
Using this and the estimate $|[\be (I+M)]^2_1+[\be (I+M)]^2_2|\geq |\be|^2/2$ for small enough $\delta'$ (checked in a similar manner to the bound \eqref{015}) one may verify \eqref{Kbound'}.
\end{proof}

Now we claim the following.
\begin{lem}\label{vPreservesA'}
Consider a fixed function $\omega(\al)$ of the form \eqref{final001} for a bounded set $\Omega$. Then the operator $F$ defined in Definition \ref{F2} preserves analyticity. That is, for any $r>0$ and $X_{(\cdot)}\in\cB_{\delta,r}$, for fixed $\al\in\Omega$, $F(X_z)(\al)$ is analytic in $d_r$, where $\cB_{\delta,r}$ is as in Definition \ref{banach'}.
\end{lem}
\begin{proof}
Fix an $r>0$, and an $X_{(\cdot)}\in\cB_{\delta,r}$. What we want to prove is that the limit definition of
$\frac{d}{d z} \left( F(X_z)(\al)\right)$ converges for each $z\in d_r$, $\al\in\Omega$. Fixing $\al\in\Omega$ and $z\in d_r$ we find
\begin{eqnarray}\label{0point56'}
& &\frac{1}{h}\int_\Omega( K(X_{z+h}(\al)-X_{z+h}(\be))-K(X_z(\al)-X_z(\be)))\omega(\be) d\be\\
& &=\int_\Omega
\frac{1}{h}
\left(
\int^{z+h}_z
\left(\frac{d}{d\zeta}X_\zeta(\al)-\frac{d}{d\zeta}X_\zeta(\be)\right)\na K(X_\zeta(\al)-X_\zeta(\be))d\zeta\right) d\be,\nn
\end{eqnarray}
where we take the straight path from $z$ to $z+h$ in the integral. Since $X_{(\cdot)}\in\cB_{\delta,r}$ implies $X_\zeta(\cdot)$ is Lipschitz with a fixed constant as $\zeta$ varies in $d_r$, by using the Cauchy integral formula one finds that we get a uniform Lipschitz bound on $\frac{d}{d\zeta}X_\zeta(\cdot)$, i.e.
\ba
\sup_{\ga,\be\in\Omega}\left(\frac{\left|\frac{d}{d\zeta}X_\zeta(\ga)-\frac{d}{d\zeta}X_\zeta(\be)\right|}{|\ga-\be|}\right)\leq C,
\ea
as $\zeta$ varies in the interval from $z$ to $z+h$, which must be contained in $d_r$ for small $h$. Moreover, since we have $\left(X_\zeta(\al)-X_\zeta(\be)\right)=(\al-\be)(I+O(\delta))$, we can apply Lemma \ref{kernelbound} to get a bound on the kernel in the right hand side of \eqref{0point56'}.  Together with the Lipschitz bound on $\frac{d}{d\zeta}X_\zeta(\cdot)$, we can thus bound the quantity inside the integral over $\beta$ uniformly by $C|\al-\be|^{-1}$, then observe that the quantity within the integral taken in the $\zeta$ variable is a continuous function of $\zeta$ as long as $\beta\neq\al$, and then apply the dominated convergence theorem to find
\ba
\frac{d}{dz}\left(F(X_z)(\al)\right)
=
\int_\Omega
\left(\frac{d}{dz}X_z(\al)-\frac{d}{dz}X_z(\be)\right)\na K(X_z(\al)-X_z(\be))d\be.
\ea
\end{proof}

\textbf{Local boundedness for the operator $F$}

\vspace{.05in}

In all that follows let us consider an arbitrary fixed $\mu\in(0,1)$, and let us use the notation $F_X$ to refer to $F(X)$. Proposition \ref{bddF} asserts the following: under the hypothesis that the initial patch $\Omega$ is bounded, finitely-connected, and has a $C^2$ non-self-intersecting boundary, there exists a $\delta>0$ and a bound of the form
\ba
\sup_{X\in B_\delta}\lVert F_X \rVert_{C^{1,\mu}(\Omega)}\leq C,
\ea
for a constant $C=C(\Omega)$, where $B_\delta=\{X\in C^{1,\mu}(\Omega):\lVert X-Id\rVert_{C^{1,\mu}(\Omega)}\leq\delta\}$. The proof of this estimate is the main difficulty in proving the Lagrangian analyticity result.

We apply the methods of Section 2, taking the Banach space of maps $Y$ to be $C^{1,\mu}(\Omega)$. When combining the result of Proposition \ref{bddF} with Lemma \ref{vPreservesA'} and Theorem \ref{mainth}, this proves that for $\al\in\Omega$, the trajectories $X(\al,t)$ of the solutions $u$ solving \eqref{eulereqns} are analytic in time, as we conclude in Theorem \ref{final}.

Observe that $F_X$ is not difficult to bound pointwise with the use of Lemma \ref{kernelbound}:
\begin{align}\label{aa1'}
|F_X(\al)|\leq\left|\int_\Omega K(X(\al)-X(\be))d\be\right|\leq\int_\Omega C|\al-\be|^{-1}d\be<\infty.
\end{align}
The main estimates involve bounds on the gradient of $F_X(\al)$. The first task is to establish pointwise bounds on $\na F_X(\al)$ and related quantities, via Lemmas \ref{2ndBOUND} and \ref{MAINBOUND}. With that, we are also able to justify a useful formula for the derivatives of $F_X$ in the directions of vector fields $T$ tangent to the boundary of the patch and $N$ normal to the boundary, which is given in Lemma \ref{diffFormula}. With these formulas, the proof of the H\"{o}lder bound on $\na F_X(\al)$ is essentially reduced to proving the H\"{o}lder continuity of $T(F_X)(\al)$, which is verified with Lemma \ref{mainHbound}. Once these lemmas have been established, the proof of Proposition \ref{bddF} follows without much additional difficulty.

\begin{subsubsection}{Pointwise gradient bounds}
We begin with Lemma \ref{2ndBOUND}, which guarantees in particular the differentiability of $F_X$.
\begin{lem}\label{2ndBOUND}
Assume we have a finitely-connected, bounded domain $\Omega$ with $C^2$ non-self-intersecting boundary. Then there is a $\delta_1>0$ such that for all $X\in B_{{\delta_1}}$, $F_X(\al)$ is differentiable in $\Omega$ with gradient satisfying the pointwise bound
\ba
|\na F_X(\al)|\leq C\quad \mbox{for all }\al\in\Omega,
\ea
for a constant $C$ depending on $\Omega$.
\end{lem}
For the most part the proof of Lemma \ref{2ndBOUND} is routine once Lemma \ref{MAINBOUND} is proven. Note that if we were to simply place the gradient inside the integral in the expression for $F_X$ and use a principal value integral, we would obtain the expression
\ba
p.v.\int_\Omega\na X(\al)\na K(X(\al)-X(\be))d\be. 
\ea
Here we use the convention that $(\na X )_{ij}=\del X_j/\del \al_i$. The next lemma allows us to bound expressions like this, and provides the main estimate used to prove Lemma \ref{2ndBOUND}. It is also used crucially in Lemma \ref{mainHbound}, which in turn provides the key estimate for establishing the H\"{o}lder continuity of $\na F_X$.
\begin{lem}\label{MAINBOUND}
Under the same assumptions as in Lemma \ref{2ndBOUND}, for sufficiently small $ \delta_2$ there exists a constant $C$ such that that for all $X\in B_{\delta_2}$, all $\alpha\in\Omega$, and $r>0$,
\ba\label{a01}
\left|\int_{\Omega\setminus B_r(\al)} \na K(X(\al)-X(\be))d\be\right|\leq C.
\ea
\end{lem}
\begin{proof}
To begin, we claim that the quantity in the left hand side of \eqref{a01} that we want to bound is within a bounded constant of
\ba\label{vpn005}
I:=\int_{\Omega\setminus B_r(\al)} \na K((\al-\be)\na X(\al))d\be.
\ea
To show this, first we establish that for $\al$ and $\be$ in $\Omega$, we have
\ba\label{p007}
|X(\al)-X(\be)-(\al-\be)\na X(\al)|
\leq
C|\al-\be|^{1+\mu}.
\ea
To verify the above inequality, let us define a distance function $d_\Omega(\al,\be)$ for points $\al,\be\in\Omega$ by 
\ba
d_\Omega(\al,\be):=\mbox{inf}\{\mbox{length}(\ga):\mbox{curves }\ga\subset\Omega\mbox{ joining }\al\mbox{ and }\be\}.
\ea
Since the boundary $\del \Omega$ is given by a $C^2$, non-self-intersecting curve, for some constant $C$ depending on $\Omega$ we have the following for all $\al$ and~$\be$ in $\Omega$.
\ba\label{vpn003}
d_\Omega(\al,\be)\leq C|\al-\be|.
\ea
Consider a fixed $\al$ and $\be$ in $\Omega$. For small $\eps>0$, we may choose a curve $\ga^\eps_{\al,\be}(s):[0,1]\to\Omega$ with $\ga^\eps_{\al,\be}(0)=\be$, $\ga^\eps_{\al,\be}(1)=\al$, and
\ba\label{vpn004}
|\mbox{length}(\ga^\eps_{\al,\be})-d_\Omega(\al,\be)|\leq\eps.
\ea
Now we observe
\begin{eqnarray}\label{vpn001}
X(\al)-X(\be)-(\al-\be)\na X(\al)
&=& \int^1_0 \frac{d}{d\tau}(X(\ga^\eps_{\al,\be}(\tau))d\tau-(\al-\be)\na X(\al),\\
&=& \int^1_0 \frac{d\ga^\eps_{\al,\be}}{d\tau}(\tau)\big(\na X(\ga^\eps_{\al,\be}(\tau))-\na X(\al)\big)d\tau,\nn
\end{eqnarray}
For a function $c(\eps)$ which tends to zero as $\eps$ tends to zero, we have for all $\tau\in[0,1]$
\ba\label{vpn003002}
\left|d_\Omega(\al,\be)-[d_\Omega(\al,\ga^\eps_{\al,\be}(\tau))+d_\Omega(\ga^\eps_{\al,\be}(\tau),\be)]\right|\leq c(\eps).
\ea
By using \eqref{vpn003002} and \eqref{vpn003} we find
\begin{eqnarray}\label{vpn003003}
|\al-\ga^\eps_{\al,\be}(\tau)|\leq d_\Omega(\al,\ga^\eps_{\al,\be}(\tau)) \leq  
d_\Omega(\al,\ga^\eps_{\al,\be}(\tau))+d_\Omega(\ga^\eps_{\al,\be}(\tau),\be)
&\leq& d_\Omega(\al,\be)+c(\eps)\nn\\
&\leq& C|\al-\be|+c(\eps).
\end{eqnarray}
In \eqref{vpn001}, we use the H\"{o}lder continuity of $\na X$ along with \eqref{vpn004}, \eqref{vpn003003}, and \eqref{vpn003} to get
\begin{eqnarray}\label{vpn002}
|X(\al)-X(\be)-(\al-\be)\na X(\al)|
&\leq&
\int^1_0
\left|\frac{d \ga^\eps_{\al,\be}}{d\tau}(\tau)d\tau
\right|
\sup_{\tau\in[0,1]}|\na X(\ga^\eps_{\al,\be}(\tau))-\na X(\al)|\\
&\leq&
C'\;\mbox{length}(\ga^\eps_{\al,\be})|\al-\ga^\eps_{\al,\be}(\tau)|^\mu\nn\\
&\leq&
C'(d_\Omega(\al,\be)+\eps)(C|\al-\be|+c(\eps))^\mu\nn\\
&\leq&
C'(C|\al-\be|+\eps)(C|\al-\be|+c(\eps))^\mu.
\end{eqnarray}
Taking the limit as $\eps$ tends to zero, we obtain \eqref{p007}.
Now we observe
\begin{eqnarray}\label{p006}
& &\na K(X(\al)-X(\be))-\na K((\al-\be)\na X(\al))=\\
& &\sum^2_{i=1}[(X(\al)-X(\be))-(\al-\be)\na X(\al)]_i\int^1_0(\del_i\na K)(\tau(X(\al)-X(\be))+(1-\tau)(\al-\be)\na X(\al))d\tau\nonumber.
\end{eqnarray}

For the integral over $\tau$ in the right hand side of \eqref{p006}, we note $X(\al)-X(\be)=(\al-\be)(I+O(\delta_2))$ and $\na X(\al)=I+O( \delta_2)$, so that for any $\tau\in(0,1)$,
\ba
\tau(X(\al)-X(\be))+(1-\tau)(\al-\be)\na X(\al)=(\al-\be)(I+O( \delta_2)).
\ea
Thus, by using Lemma \ref{kernelbound}, for small enough $\delta_2$ we get
\ba\label{p008}
|\na^2 K(\tau(X(\al)-X(\be))+(1-\tau)(\al-\be)\na X(\al))|\leq c'|\al-\be|^{-3}.
\ea
Using \eqref{p007} and \eqref{p008} in \eqref{p006} we now find
\begin{eqnarray}
|\na K(X(\al)-X(\be))-\na K((\al-\be)\na X(\al))|
\leq c''|\al-\be|^{\mu-2}.
\end{eqnarray}
Now we can bound the difference:
\ba\label{p010}
\left|
\int_{\Omega\setminus B_r(\al)}(\na K(X(\al)-X(\be))-\na K((\al-\be)\na X(\al))d\be
\right|
\leq c''\int_{\Omega}|\al-\be|^{\mu-2}d\be\leq \tilde{c}.
\ea

Recall \eqref{vpn005} and consider that we want \eqref{a01}. Now we have reduced the proof of the lemma to bounding $I$. We fix a positive constant $D>0$ to be specified later, determined by $\Omega$, and uniformly bound the quantity
\ba\label{p011}
I':=\int_{\Omega\cap\{|\al-\be|>D\}}|\na K((\al-\be)\na X(\al))|d\be.
\ea
Indeed, we have
\ba\label{p012}
I'
\leq C\int_{\Omega\cap\{|\al-\be|>D\}}|\al-\be|^{-2}d\be
\leq C\int_{A_D(\al)}|\al-\be|^{-2}d\be\leq \underline{C},
\ea
where $A_D(\al)$ is taken to be an annulus centered at $\al$ with inner radius $D$ and with the same area as $\Omega$.

Now note that if $r\geq D$, we have
\ba\label{p013}
|I|\leq I'\leq\underline{C},
\ea
and so in that case we are done.

On the other hand, suppose that we have $r<D$. In that case we split the integral over $\Omega\setminus B_r$ in the expression for $I$, getting
\begin{eqnarray}\label{p014}
|I|&=&\left|\left(
\int_{\Omega\cap\{|\al-\be|\geq D\}}+\int_{\Omega\cap\{D>|\al-\be|>r\}}
\right)
\na K((\al-\be)\na X(\al))d\be\right|,\\
&\leq& I'+\left|\int_{\Omega\cap\{D>|\al-\be|>r\}}\na K((\al-\be)\na X(\al))d\be\right|,\label{p014b}
\end{eqnarray}
so that it remains to bound the integral on the right in line \eqref{p014b}. Let us denote $d_\al:=\textrm{dist}(\al,\del\Omega)$. Consider now the case that $d_\al > r$. First, suppose $d_\al$ is so large that $d_\al\geq D$. In that case the remaining integral is over the entire annulus, which is then contained in $\Omega$, namely
\ba
\int_{D>|\al-\be|>r}\na K((\al-\be)\na X(\al))d\be.
\ea
However, we have that for any annulus $A(\al)$ centered at $\al$,
\ba
\int_{A(\al)}\na K((\al-\be)\na X(\al))d\be=0.
\ea
To see this, consider the following calculation, for which we take $R$ and $R'$ to be the inner and outer radii, respectively, of $A(\al)$.
\begin{eqnarray}\\
\int_{A(\al)}\na K((\al-\be)\na X(\al))d\be
&=&-(\na X(\al))^{-1}\int_{A(\al)}\na_\be( K((\al-\be)\na X(\al)))d\be,\nn\\
&=&-(\na X(\al))^{-1}\int_{\del A(\al)}\nu_{out}(\be) K((\al-\be)\na X(\al))d\sigma(\be),\label{p015b}\nn\\
&=&-(\na X(\al))^{-1}\left(
\int_{|\al-\be|=R'}
-\int_{|\al-\be|=R}\right)\label{p015'}\\
& &\hspace{1.5in}
n^+(\al-\be) K((\al-\be)\na X(\al))d\sigma(\be),\nn\\
&=&0,\nn
\end{eqnarray}
where $\nu_{out}=(\nu^1_{out},\nu^2_{out})^t$ indicates the outward pointing normal for $\del A(\al)$, and $n^+(\al-\be)=\frac{(\al-\be)^t}{|\al-\be|}$. One finds the expression on the right in \eqref{p015'} is zero by observing that the expression
\ba
\int_{|\al-\be|=\rho}n^+(\al-\be) K((\al-\be)\na X(\al))d\sigma(\be)=\int_{|\be|=\rho}n^+(\be) K(\be\na X(\al))d\sigma(\be)
\ea
is independent of $\rho$, as the expression in the integral on the right is homogenous of degree $0$ in $\be$. Thus the case $d_\al\geq D$ is handled. Now we consider the case that $d_\al<D$. Then if $d_\al>r$ we can split the integral on the right in \eqref{p014b} into the following two:
\begin{eqnarray}\label{p016}
\int_{\Omega\cap\{D>|\al-\be|>r\}}\na K((\al-\be)\na X(\al))d\be
&=&
\int_{\Omega\cap\{D>|\al-\be|>d_\al\}}
\na K((\al-\be)\na X(\al))d\be\\
& &+\int_{\{d_\al\geq |\al-\be|>r\}}
\na K((\al-\be)\na X(\al))d\be.\nn
\end{eqnarray}
The second integral is over an annulus centered at $\al$, so it is zero, and we are left with just the first integral in the right hand side in the case that $d_\al>r$. Thus, we will have handled both cases, $d_\al>r$ and $d_\al\leq r$, if we can bound the quantity
\ba
\tilde{I}:=\int_{\Omega\cap\{D>|\al-\be|>\rho_0\}}\na K((\al-\be)\na X(\al))d\be,
\ea
where $\rho_0=\max(d_\al,r)$. First, we choose $\hat{\al}\in\del \Omega$ such that $|\al-\hat{\al}|=d_\al$, and we define
\begin{eqnarray}
\cS(\al)     &:=&\{\hat{u}\in S^1:\hat{u}\cdot\nu_{in}(\hat{\al})> 0\},\\
\cU(\al)     &:=&\{\al+\rho \hat{u}:\rho\in(\rho_0,D),\hat{u}\in\cS(\al)\}.\nn
\end{eqnarray}
The goal will be to show that, by the choice of $D$ we make, the regions of integration for the integral in $\tilde{I}$ and the corresponding integral over the semi-annulus $\cU(\al)$ will mostly overlap, since the boundary of $\Omega$ is somewhat smooth. This will be helpful since we will be able to show that the corresponding integral over $\cU(\al)$ is zero.

In the proof of Lemma \ref{vfpair}, a $C^2$ function $\varphi:\R^2\to\R$ is produced with the properties that the interior of the patch is given by $\Omega=\{\be\in \R^2:\varphi(\be)>0\}$, where the vector field $T$ in the statement of the lemma is given by $\na^\perp \varphi\big|_{\overline\Omega}$. Following the strategy of \cite{bc}, we use this to define $D$, taking
\ba\label{p004b}
D:=
\frac{\inf_{\del\Omega}|\na\varphi|}{\sup_\Omega|\na^2\varphi|}
.
\ea
Clearly the denominator is nonzero. Otherwise $T_\al$ would be constant, and $\Omega$ unbounded. We denote $A_{\rho_0,D}:=\{D>|\al-\be|>\rho_0\}$ and check
\begin{eqnarray}
\tilde{I}	&=&	\left(	\int_{\Omega\cap A_{\rho_0,D}}
						-\int_{\cU(\al)}+\int_{\cU(\al)}
				\right)
				\na K((\al-\be)\na X(\al))d\be,\\
			&=&	\left(	\int_{(\Omega\cap A_{\rho_0,D})\setminus \cU(\al)}
						-\int_{\cU(\al)\setminus (\Omega\cap A_{\rho_0,D})}
						+\int_{\cU(\al)}
				\right)
				\na K((\al-\be)\na X(\al))d\be.\label{p017}
\end{eqnarray}
We argue that
\ba\label{p018}
\int_{\cU(\al)}\na K((\al-\be)\na X(\al))d\be=0.
\ea
Because $K(z)$ is odd, $\na K(z)$ is even, and so \eqref{p018} follows from the corresponding fact for the full annulus. Using \eqref{p018} in \eqref{p017}, it follows that
\ba
|\tilde{I}|	\leq	\int_{(\Omega\cap A_{\rho_0,D})\Delta\cU(\al)}|\na K((\al-\be)\na X(\al))|d\be.
\ea
We define for $\rho\in(d_\al,D)$
\begin{eqnarray}
\cA_\rho(\al)&:=&\{ \hat{u}\in S^1:\al+\rho \hat{u}\in\Omega\},\\
R_\rho(\al)  &:=&\cA_\rho(\al)\Delta\cS(\al),\nn
\end{eqnarray}
and note that we have $(\Omega\cap A_{\rho_0,D})\Delta \cU(\al)=\{\al+\rho \hat{u}:\rho\in(\rho_0,D),\hat{u}\in R_\rho(\al)\}$. By using the Geometric Lemma from \cite{bc}, one finds that the following bound on the measure of $R_\rho(\al)$ holds\footnote{In \cite{bc} the proof is given for patches with $C^{1,\nu}$ boundary for $\nu\in(0,1)$ and an analogous constant replacing $D$ in the definition of the set corresponding to $R_\rho(\al)$. However, the Geometric Lemma also holds in the case $\nu=1$, which is what we use.}:
\ba\label{p019}
|R_\rho(\al)|\leq 2\pi\left(3\frac{d_\al}{\rho}+2\frac{\rho}{D}\right),
\ea
as long as $\rho\in(d_\al,D)$ and $d_\al<D$. Thus we get the bound
\begin{eqnarray}\label{p020}
|\tilde{I}|	&\leq&	C\int_{(\Omega\cap A_{\rho_0,D})\Delta\cU(\al)}	|\al-\be|^{-2}d\be,\\
			&=&		C\int^D_{\rho_0}	\int_{R_\rho(\al)}d\theta\rho^{-1}d\rho,\nn\\
			&=&		C\int^D_{\rho_0}	|R_\rho(\al)|\rho^{-1}d\rho,\nn\\
			&\leq&	C'\left(
								\int^D_{\rho_0}	d_\al \rho^{-2}d\rho
								+\int^D_{\rho_0}	D^{-1}d\rho
					\right),\nn\\
			&\leq&	C''((d_\al/\rho_0-d_\al/D)+(1-\rho_0/D)).\nn
\end{eqnarray}
Now using \eqref{p012} and \eqref{p014}-\eqref{p014b}, our bound on $\tilde I$, and \eqref{p016} (needed for the case $d_\al>r$ only), we have bounded $I$. Using \eqref{p010}, this bounds the integral in the left hand side of \eqref{a01}.
\end{proof}
Now, we provide the proof of Lemma \ref{2ndBOUND}.
\begin{proof}{(Proof of Lemma \ref{2ndBOUND})}
Consider $\al$ varying in a small open neighborhood $N\subset \Omega$ and a fixed $r>0$ such that $B_r(\al)\subset\Omega$ for all $\al$ in that neighborhood. We write
\ba
F_X(\al)=\int_{\Omega\setminus B_r(\al)} K(X(\al)-X(\be))d\be+\int_{B_r(\al)}K(X(\al)-X(\be))d\be.
\ea
We denote by $D^h_{\al_i}$ the difference quotient operator with increment $h$ approximating the partial derivative with respect to $\al_i$. For a fixed $i$ we have
\begin{eqnarray}
D^h_{\al_i}(F_X)(\al)=
D^h_{\al_i}\left(
			\int_{\Omega\setminus B_r(\al)} K(X(\al)-X(\be))d\be
	\right)
	+
D^h_{\al_i}\left(
			\int_{B_r(\al)} K(X(\al)-X(\be))d\be
	\right),\\
	=I^\flat_h+I^\sharp_h.\nonumber
\end{eqnarray}
To handle $I^\flat_h$, we note
\begin{eqnarray}\label{f01}
\del_i\left(\int_{\Omega\setminus B_r(\al)}K(X(\al)-X(\be))d\be\right)&=&\del_i X(\al)\int_{\Omega\setminus B_r(\al)}\na K(X(\al)-X(\be))d\be\\
& &+\int_{|\be-\al|=r}\nu^i_{in}(\be)K(X(\al)-X(\be))d\sigma(\be),\nn
\end{eqnarray}
where the second integral in the right hand side of \eqref{f01} is easily seen to be bounded uniformly in $r$, since $K(X(\al)-X(\be))=O(|\al-\be|^{-1})=O(r^{-1})$ and the circumference of the ball is $2\pi r$. Furthermore, the first integral on the right hand side of \eqref{f01} is bounded independently of $\al$, $N$, and $r$ by Lemma \ref{MAINBOUND}. Thus, the left hand side of \eqref{f01} is bounded uniformly, and so we can conclude that $\lim_{h\to 0}I^\flat_h$ exists for any $\al\in N$ and is uniformly bounded. In addition, note that by using the dominated convergence theorem in the right hand side of \eqref{f01}, we find that the left hand side defines a continuous function of $\al$ in $N$.
Observe
\begin{eqnarray}\label{f02}
I^\sharp_h&=&
\int_{B_r(\al+h e_i)}D^{h}_{\al_i}(K(X(\al)-X(\be)))d\be
+\int K(X(\al)-X(\be))D^{h}_{\al_i}(1_{B_r(\al)}(\be))d\be,
\end{eqnarray}
where
\ba
\lim_{h\to 0}
\int K(X(\al)-X(\be))D^{h}_{\al_i}(1_{B_r(\al)}(\be))d\be
=\int_{|\be-\al|=r}\nu^i_{out}(\be) K(X(\al)-X(\be))d\sigma(\be).
\ea
We handle this in the same way that we handled the boundary integral in the right hand side of \eqref{f01}, finding the result is bounded independently of $r$, $\al$, and $N$, and is continuous with respect to $\al$ in $N$. For the integral on the left in \eqref{f02}, we have
\begin{eqnarray}
\int_{B_r(\al+h e_i)}D^{h}_{\al_i}(K(X(\al)-X(\be)))d\be&=&
\int_{B_r(\al+h e_i)}(D^{h}_{\al_i}+D^{h}_{\be_i})(K(X(\al)-X(\be)))d\be\\
& &-\int_{B_r(\al+h \hat{u})}D^{h}_{\be_i}(K(X(\al)-X(\be)))d\be,\nn\\
&=&J^1_h+J^2_h.\nn
\end{eqnarray}
Regarding $J^1_h$, we note that
\ba\label{fin2031}
\int_{B_r(\al)}(\del_{\al_i}+\del_{\be_i})(K(X(\al)-X(\be)))d\be=
\int_{B_r(\al)}(\del_i X(\al)-\del_i X(\be))\na K(X(\al)-X(\be))d\be
\ea
is uniformly bounded in $r$, since $|\na X(\al)-\na X(\be)|\leq C|\al-\be|^\mu$ and $|\na K(X(\al)-X(\be))|\leq C'|\al-\be|^{-2}$. From this we can conclude that $\lim_{h\to 0}J^1_h$ exists and is uniformly bounded. In addition, with a change of variables we find the right hand side of \eqref{fin2031} is
\ba
\int_{B_r(0)}(\del_i X(\al)-\del_i X(\al-\be))\na K(X(\al)-X(\al-\be))d\be,
\ea
which we see is continuous in $\al$ by using the dominated convergence theorem.
Now we consider $J^2_h$. Note
\begin{eqnarray}\label{fin203}
\int_{B_r(\al+h e_i)}D^{h}_{\be_i}(K(X(\al)-X(\be)))d\be
&=&
	\frac{1}{h}\bigg(
					\int_{B_r(\al+h e_i)}
						K(X(\al)-X(\be+he_i)) 
					d\be\\
& &\hspace{.5in}
					-
					\int_{B_r(\al+h e_i)}
					K(X(\al)-X(\be))
					d\be
				\bigg),\nonumber\\
&=&					
	\frac{1}{h}\bigg(
					\int_{B_r(\al+2he_i)}
						K(X(\al)-X(\be)) 
					d\be\nn\\
& &\hspace{.5in}
					-
					\int_{B_r(\al+h e_i)}
					K(X(\al)-X(\be))
					d\be
				\bigg),\nn
\end{eqnarray}
and
\ba\label{n001}
\lim_{h\to0}\frac{1}{h}\left(
					\int_{B_r(\al+h e_i)}-\int_{B_r(\al)}
					\right)
					K(X(\al)-X(\be))d\be
				=
\int_{|\al-\be|=r} \nu^i_{out}(\be)K(X(\al)-X(\be))d\sigma(\be),
\ea
which as we saw before is uniformly bounded. Also, the existence of the limit \eqref{n001} implies that the limit as $h$ tends to zero of the expression in \eqref{fin203} exists and is the same. Note, too, that the right hand side of \eqref{n001} is a continuous function of $\al$. With that, we get that $\lim_{h\to 0}J^2_h$ exists and is uniformly bounded, and we can conclude that $\lim_{h\to 0}I^\sharp_h$ in addition to $\lim_{h\to 0}I^\flat_h$ both exists and is uniformly bounded, giving us the existence of the limit defining $\del_i(F_X(\al))$, which we find is a continuous function of $\al$, for both $i=1$ and $2$. Since we get the existence and continuity of each of the partial derivatives in the neighborhood $N$, we have that $F_X$ is differentiable there. From the above discussion we also see that this gives a bound on $|\na F_X(\al)|$ that is independent of the chosen $r$, $\al$, and $N$.
\end{proof}
\end{subsubsection}
\begin{subsubsection}{Tangential vector fields and H\"{o}lder bounds}
Now that we have justified that one can make sense of the gradient of $F_X$ and we have some related bounds, we can turn our attention to finding a convenient way to express derivatives along certain vector fields. With the following lemma, we get a formula for differentiating along a vector field tangent to the boundary of $\Omega$ and along a vector field normal to that one. The remainder of the key steps for proving Proposition \ref{bddF} are framed in terms of the resulting expressions.
\begin{lem}\label{diffFormula}
Assume $\Omega$ is as in Lemma \ref{2ndBOUND}. Then there is a $\delta>0$ so that we have the following. Let $T=(T^1,T^2):\overline{\Omega}\to\R^2$ be any continuously differentiable, divergence-free vector field tangent to $\del \Omega$. For $X\in B_\delta$ we have (i)
\ba
T(F_X)(\al)=\int_\Omega (T_\al \na X(\al)-T_\be \na X(\be))\na K(X(\al)-X(\be))d\be\quad\textrm{for all }\al\in\Omega.
\ea
Here, we have denoted the value of the vector field at the point $\al$ by $T_\al=(T^1(\al),T^2(\al))$, and when the matrix valued function $\na X(\cdot)$ appears to the right of $T_{(\cdot)}$, we mean that it acts on $T_{(\cdot)}$ via multiplication.

(ii) Taking $N=(N^1,N^2):\overline{\Omega}\to\R^2$ to be the orthogonal vector field given by $N_\al=T_\al^\perp$, we have
\begin{align}\label{p0001}
N(F_X)(\al)=\left(
T(F_X)(\al)P_X(\al)
-\det \cO(\al)e_1
\right)
(Q_X(\al))^{-1},
\end{align}
for all $\al\in\Omega$ such that $T_\al\neq 0$, where $e_1=(1, 0)$,
$\cO(\al):=
\left(
\begin{array}{ccc}
T_\al \\
N_\al     \end{array} \right)$, and
\begin{align}
P_X=
\left(
\begin{array}{ccc}
N(X_1)        & -N(X_2) \\
N(X_2)        & N(X_1) \end{array} \right),
\quad\quad
Q_X=
\left(
\begin{array}{ccc}
T(X_1)        & -T(X_2) \\
T(X_2)        & T(X_1) \end{array} \right).
\end{align}
\end{lem}
\begin{proof}
One verifies the statement of (i) with a computation using the fact that $T$ is divergence free and tangent to $\del \Omega$ in particular. Now we show that (ii) holds. At first, let us prove the formula \eqref{p0001} holds for $X\in B_{\delta'}$ for small $\delta'$ with $X$ taking values in $\R^2$, rather than $\C^2$. We then define
\ba\label{p031}
G_X(x):=\int_\Omega K(x-X(\be))d\be.
\ea
Then we claim that
\begin{eqnarray}\label{p032a}
(\div G_X)(X(\al))  &=& 0,\\
(\curl G_X)(X(\al)) &=& \det (\na X(\al))^{-1},\quad\textrm{for }\al\in\Omega.\label{p032b}
\end{eqnarray}
This follows if we observe that
\ba
G_X(x)=\int_{X(\Omega)}K(x-y)|\det (\na X(X^{-1}(y)))^{-1}|dy,
\ea
so $\div G_X=0$ and
\ba
(\curl G_X)(X(\al))=\int_{X(\Omega)}\delta(X(\al)-y)|\det \left(\na X (X^{-1}(y))\right)^{-1}|dy=\det \left(\na X(\al)\right)^{-1},
\ea
noting that we can ensure the determinant is positive for real $X\in B_{\delta'}$ if $\delta'$ is small enough. From \eqref{p032a} and \eqref{p032b} we find
\ba\label{p033}
(\del_{x_2}G_X)(X(\al))
=
	(\del_{x_1}G)(X(\al))
		\left(
			\begin{array}{ccc}	0	&	-1	\\
								1	&	0
			\end{array}
		\right)
	-\det(\na X(\al))^{-1} e_1.
\ea
Meanwhile,
\ba\label{p034}
\left(\cO(\al)\na X(\al)\right)^{-1}
\left(
	\begin{array}{ccc}	T(F_X)(\al) \\
						N(F_X)(\al)
	\end{array}
\right)
=(\na G_X)(X(\al)).
\ea
This gives us two ways of calculating $\del_{x_2}G_X(X(\al))$. For the first, we simply look at the second row in the above expression. For the second, we compute the first row, $\del_{x_1}G_X(X(\al))$, and then use \eqref{p033} to get another expression for $\del_{x_2}G_X(X(\al))$ in terms of this. Equating these gives us
\begin{eqnarray}\label{p034'}
& &
e_2 (\cO(\al)\na X(\al))^{-1}
\left(
	\begin{array}{ccc}	T(F_X)(\al) \\
						N(F_X)(\al)
	\end{array}
\right)
\\
& &\hspace{.75in}=
	e_1
	\left((\cO(\al)\na X(\al))^{-1}
	\left(
		\begin{array}{ccc}	T(F_X)(\al) \\
							N(F_X)(\al)
		\end{array}
	\right)
	\left(
	\begin{array}{ccc}	0	&	-1 \\
						1	&	0
	\end{array}
	\right)
	-\det (\na X(\al))^{-1}I
\right).\nn
\end{eqnarray}
After a little computation and rearrangement this gives
\ba\label{p035}
N(F_X)(\al)Q_X(\al)
=
T(F_X)(\al)P_X(\al)-\det\cO(\al) e_1,
\ea
and then multiplying on the right by $(Q_X(\al))^{-1}$ gives the desired formula. Thus, \eqref{p0001} holds for $X\in B_{\delta'}$ taking values in $\R^2$. To get the same result for  $\delta$ sufficiently small and any $X\in B_\delta$, not necessarily real, we fix such an $X$ and define the map $X_\lambda:=\textrm{Re}(X)+\lambda \textrm{Im}(X)$ for $\lambda$ in a disc in $\C$ of radius 2 about $0$. Then observe that $X\in B_\delta$ implies $X_\lambda$ satisfies $\lVert X_\lambda-Id\rVert_Y\leq C\delta$. Thus, as long as $\delta<\frac{\delta'}{C}$, $X_\lambda\in B_{C\delta}$ implies that the formula holds for $X_\lambda$ for all $\lambda$ in the interval $[-1,1]$. Now we claim that if we insert $X_\lambda$ in the right hand side of \eqref{p0001}, the result is an analytic function of $\lambda$ in $\{|\lambda|< 2\}$, given $X\in B_\delta$ for small enough $\delta$. 
To verify this, we first fix an arbitrary in $\al\in\Omega$. Now we consider the spatial difference quotients approximating $T(F_{X_\lambda})(\al)$. For sufficiently small $h\neq 0$ and a differentiable function $g:\Omega\to\R^2$, we define
\ba
T^h(g)(\al):=\frac{1}{h}(g(\al+h T_\al )-g(\al)).
\ea 
Note that as long as $\delta$ is small enough, $X_\lambda\in B_{C\delta}$ satisfies the hypothesis of Lemma \ref{2ndBOUND}. We find that as $h$ varies in some small interval, say $h\neq 0$ with $|h|<h_{max}(\al)$ for some $h_{max}>0$, the $T^h(F_{X_\lambda})(\al)$ are all analytic functions of $\lambda$, and moreover they are convergent and uniformly bounded as $h$ tends to zero, by Lemma \ref{2ndBOUND}. This implies that $T(F_{X_\lambda})(\al)$ is analytic in $\{|\lambda|< 2\}$. In fact, if we consider \eqref{p0001} with $X_\lambda$ in place of $X$, which we shall refer to as equation \eqref{p0001}${}_\lambda$, we find that all the terms in the right hand side are analytic in $\lambda$ as long as $T_\al\neq 0$. Similarly, the spatial difference quotients approximating $N(F_{X_\lambda})(\al)$ are all analytic functions of $\lambda$ in $\{|\lambda|<2\}$, also uniformly bounded, so that $N(F_{X_\lambda})(\al)$ is also analytic in $\lambda$. Recall that we already have the equality \eqref{p0001}${}_\lambda$ for $\lambda$ in $[-1,1]$, and so we must have that \eqref{p0001}${}_\lambda$ holds for all $|\lambda|<2$. In particular it holds for $\lambda=i$, which proves the formula for $X=\textrm{Re}(X)+i\textrm{Im}(X)$.
\end{proof}

Note that we can try to recover $\na F_X$ from the expressions for $T(F_X)$ and $N(F_X)$ via the relation
\ba\label{p034f}
\cO(\al)\na F_X(\al)=
\left(
\begin{array}{ccc}
T(F_X)(\al) \\
N(F_X)(\al)     \end{array} \right).
\ea
Of course, this does not work if $\cO(\al)$ is not invertible, which happens exactly at points where $T_\al$ is zero. The next lemma provides a pair of vector fields $T$ and $T'$ which are not zero at the same points, so that when it is used in combination with Lemma \ref{diffFormula} we get a complete description of the gradient of $F_X(\al)$ at every point $\al\in\Omega$.
\begin{lem}\label{vfpair}
Assume $\Omega$ is as in Lemma \ref{2ndBOUND}. There exists a pair of $C^1$, divergence-free vector fields $T,T':\overline{\Omega}\to\R^2$ tangent to $\del\Omega$ and such that each of $T$ and $T'$ has finitely many zeros belonging to the interior, with the zeros of $T$ all distinct from the zeros of $T'$.
\end{lem}
\begin{proof}
The assumptions on $\Omega$ imply that it can be described as a set of the form $\{\al\in\R^2:\varphi(\al)>0\}$ for a $C^2$ function $\varphi:\R^2\to\R$ with a finite number of critical points. Moreover, by making slight modifications to the function $\varphi$, one can also obtain a function $\varphi':\R^2\to\R$ with these properties as well, but all of whose critical points are distinct from those of $\varphi$. Taking $T:=\na^\perp\varphi\big|_{\overline \Omega}$ and $T':=\na^\perp\varphi'\big|_{\overline \Omega}$ then gives the vector fields with the desired properties.

\end{proof}

The main remaining task is to prove the H\"{o}lder seminorm bound for $T (F_X)$ and $T' (F_X)$. Once we have this, in combination with using \eqref{p034f} to express the gradient $\na F_X$ in terms of the derivatives along $T$ and $T'$, it will not be hard to prove the H\"{o}lder continuity of $\na F_X$.
\begin{lem}\label{mainHbound}
Let $T$ and $T'$ be the vector fields asserted by Lemma \ref{vfpair}. There exists a constant $C$ depending only on $\Omega$, $T$, and $T'$ such that for all $X\in B_\delta$
\ba\label{p001'}
|T(F_X)(\al)-T(F_X)(\ga)|\leq C|\al-\ga|^\mu\quad\textrm{ for $\al,\ga\in\Omega$}.
\ea
The same bound holds with $T'$ in place of $T$.
\end{lem}
\begin{proof}
Here we prove the bound for the vector field $T$, though $T$ is straightforwardly replaced by $T'$ for the other bound.

Fix $\al$, $\ga\in\Omega$. First we split $\Omega$ into the regions
\ba
\Omega_{near}:=\Omega\cap \{\be:|\al-\be|<10|\al-\ga|\},\quad\quad\Omega_{far}:=\Omega\cap\{\be:|\al-\be|>10|\al-\ga|\}.
\ea
Thus, using Lemma \ref{diffFormula}, we can write
\begin{eqnarray}\label{p001pt1}
T(F_X)(\al)-T(F_X)(\ga)
&=&\int_{\Omega_{near}\cup\Omega_{far}}(T_\al\na X(\al)-T_\be \na X(\be))\na K(X(\al)-X(\be))d\be\\
& &-\int_{\Omega_{near}\cup\Omega_{far}}(T_\ga\na X(\ga)-T_\be \na X(\be))\na K(X(\ga)-X(\be))d\be,\nonumber\\
&=&I_{near}(\al)-I_{near}(\ga)+I_{far}(\al)-I_{far}(\ga),\nonumber
\end{eqnarray}
where
\ba
I_{near}(\zeta):=\int_{\Omega_{near}}(T_\zeta\na X(\zeta)-T_\be \na X(\be))\na K(X(\zeta)-X(\be))d\be\quad\textrm{ for }\zeta\in\Omega,
\ea
and $I_{far}$ is defined analogously.
First we will show that $|I_{near}(\al)|$ and $|I_{near}(\ga)|$ are both $O(|\al-\ga|^\mu$). Observe that
\begin{eqnarray}\label{p002'}
|I_{near}(\al)|
&\leq& \int_{\Omega_{near}}|T_\al\na X(\al)-T_\be\na X(\be)||\na K(X(\al)-X(\be))|d\be,\\
&\leq& C'\int_{|\al-\be|<10|\al-\ga|}|\al-\be|^{\mu-2}d\be,\nn\\
&\leq& C'' |\al-\ga|^\mu.\nn
\end{eqnarray}
Similarly,
\begin{eqnarray}\label{p003'}
|I_{near}(\ga)|
&\leq& \int_{\Omega_{near}}|T_\ga\na X(\ga)-T_\be\na X(\be)||\na K(X(\ga)-X(\be))|d\be,\\
&\leq& C'\int_{|\al-\be|<10|\al-\ga|}|\ga-\be|^{\mu-2}d\be,\nn\\
&\leq& C'\int_{|\ga-\be|<10|\al-\ga|}|\ga-\be|^{\mu-2}d\be,\nn\\
&\leq& C'' |\al-\ga|^\mu.\nn
\end{eqnarray}
Now that these quantities have been bounded, we are left with handling
\begin{eqnarray}\label{p003pt1}\\
I_{far}(\al)-I_{far}(\ga)
 &=& \int_{\Omega_{far}}(T_\al\na X(\al)-T_\ga\na X(\ga))\na K(X(\al)-X(\be))d\be\nn\\
 & &+\int_{\Omega_{far}}(T_\ga\na X(\ga)-T_\be\na X(\be))\nn\\
 & &\hspace{.5in}\times(\na K(X(\al)-X(\be))-\na K(X(\ga)-X(\be)))d\be,\nonumber\\
 &=&I^i_{far}+I^{ii}_{far}.\nonumber
\end{eqnarray}
Note that
\begin{eqnarray}
|I^i_{far}|&\leq&
|T_\al\na X(\al)-T_\ga\na X(\ga)|
\left|\int_{\Omega_{far}}
\na K(X(\al)-X(\be))d\be\right|,\\
&\leq& C|\al-\ga|^\mu\left|\int_{\Omega_{far}}\na K(X(\al)-X(\be))d\be\right|,\nn
\end{eqnarray}
where $\Omega_{far}=\Omega\cap\{|\al-\be|>10|\al-\ga|\}$, so an application of Lemma \ref{MAINBOUND} gives the uniform bound
\ba\label{p020pt1}
|I^i_{far}|\leq \tilde C |\al-\ga|^\mu\quad\textrm{for all }\al,\ga\in\Omega.
\ea
Now we proceed with the bound for $I^{ii}_{far}$. We have
\ba\label{p021}
I^{ii}_{far}=\int_{\Omega_{far}}(T_\ga\na X(\ga)-T_\be\na X(\be))(\na K(X(\al)-X(\be))-\na K(X(\ga)-X(\be)))d\be.
\ea
We rewrite
\begin{eqnarray}\label{p022}
& &
\na K(X(\al)-X(\be))-\na K(X(\ga)-X(\be))\\
& &\hspace{.8in}
=\sum^2_{i=1}(X_i(\al)-X_i(\ga))\int^1_0 (\del_i\na K)(\tau X(\al)+(1-\tau)X(\ga)-X(\be))d\tau.\nn
\end{eqnarray}
Now we claim that for small enough $\delta$ there is a universal constant $a>0$ such that for any $\tau\in[0,1]$, and $\al,\be,\ga\in\Omega$,
\ba\label{p023}
\frac{\left|\sum^2_{i=1}(\tau X_i(\al)+(1-\tau)X_i(\ga)-X_i(\be))^2\right|}{|\ga-\be|^2}\geq a\quad\mbox{ if }\quad|\al-\be|>10|\al-\ga|.
\ea
To verify this, first we notice that we have for any $\tau\in[0,1]$,
\begin{eqnarray}\label{n12B}
\sum^2_{i=1}
(\tau X_i(\al)+(1-\tau)X_i(\ga)-X_i(\be))^2
&=&
\sum^2_{i=1}(X_i(\ga)-X_i(\be))^2\\
& &+
2\tau \sum^2_{i=1}(X_i(\ga)-X_i(\be))(X_i(\al)-X_i(\ga))
\nonumber\\& &+\tau^2
\sum^2_{i=1}(X_i(\al)-X_i(\ga))^2.\nonumber
\end{eqnarray}
Using \eqref{n12B}, we bound the left hand side of \eqref{p023} below by
\ba\label{n1703B}
\frac{
\left|
\sum^2_{i=1}
(X_i(\ga)-X_i(\be))^2
\right|
}
{|\ga-\be|^2}
-
2\frac{
\left|
\sum^2_{i=1}
(X_i(\ga)-X_i(\be))
(X_i(\al)-X_i(\ga))
\right|
}
{|\ga-\be|^2}
-\frac{
\left|
\sum^2_{i=1}
(X_i(\al)-X_i(\ga))^2
\right|
}
{|\ga-\be|^2}\\
=A-B_1-B_2.
\ea
Now we note
\begin{eqnarray}\label{neweq0}
& &\quad|X_i(\xi)-X_i(\eta)|\leq(1+\delta)|\xi-\eta|\\
& &\mbox{ and }\quad
\left|\sum^2_{i=1} (X_i(\xi)-X_i(\eta))^2\right|\geq (1-c\delta)|\xi-\eta|^2\quad\mbox{ for any }\xi,\eta\in\Omega,\nn
\end{eqnarray}
and that
\ba\label{neweq2}
\frac{|\al-\ga|}{|\be-\ga|}< \frac{1}{9}\quad\mbox{ if }\quad|\al-\be|> 10|\al-\ga|.
\ea
From \eqref{neweq0}, we find that
\ba\label{neweq21}
A\geq 1-c\delta.
\ea
From \eqref{neweq0} and \eqref{neweq2}, we find that for $\be$ in $\Omega_{far}$, where $|\al-\be|> 10|\al-\ga|$,
\ba\label{n1705B}
B_1 &\leq& 2(1+\delta)\frac{|X(\al)-X(\ga)|}{|\ga-\be|}
=
2(1+\delta)\frac{|\al-\ga|}{|\be-\ga|}\frac{|X(\al)-X(\ga)|}{|\al-\ga|}
\leq 
2(1+\delta)^2\frac{1}{9}.
\ea
From \eqref{neweq0} and \eqref{neweq2}, assuming $\be$ is in $\Omega_{far}$, we also obtain that
\ba\label{n1706B}
B_2=
\frac{
\left|
\sum^2_{i=1}
(X_i(\al)-X_i(\ga))^2
\right|
}
{|\al-\ga|^2}
\frac{|\al-\ga|^2}
{|\be-\ga|^2}
\leq(1+\delta)^2\frac{1}{9^2}.
\ea
Using \eqref{neweq21}-\eqref{n1706B}, assuming $\be\in\Omega_{far}$, provided $\delta>0$ is small enough, for some constant $a>0$ we find
\ba
A-B_1-B_2 \geq a.
\ea
Recalling that the left hand side of \eqref{p023} is bounded below by the left hand side of \eqref{n1703B}, we conclude that we have the bound \eqref{p023} if $\delta$ is small enough.

Using \eqref{p023} along with \eqref{p022} in \eqref{p021} yields
\begin{eqnarray}
|I^{ii}_{far}|
&\leq&	\int_{\Omega_{far}}|T_\ga\na X(\ga)-T_\be\na X(\be)||\na K(X(\al)-X(\be))-\na K(X(\ga)-X(\be))|d\be\nn\\
&\leq&	C'\int_{\Omega_{far}}|\ga-\be|^\mu 	|\ga-\be|^{-3}|X(\al)-X(\ga)|d\be,\\
&\leq&	C''|\al-\ga|\int_{\Omega_{far}}|\ga-\be|^{\mu-3}d\be,\nn\\
&\leq&	C''|\al-\ga|\int_{|\al-\be|>10|\al-\ga|}|\ga-\be|^{\mu-3}d\be,\nn\\
&\leq&	C''|\al-\ga|\int_{|\ga-\be|>|\al-\ga|}|\ga-\be|^{\mu-3}d\be,\nn\\
&=& C'''|\al-\ga|^\mu.\nn
\end{eqnarray}
Using our bound on $I^{ii}_{far}$ along with \eqref{p020pt1}, \eqref{p003pt1}, \eqref{p003'}, and \eqref{p002'}, in \eqref{p001pt1}, we conclude with \eqref{p001'}.
\end{proof}
\end{subsubsection}
\begin{subsubsection}{Proofs of the main bound and of the analyticity of trajectories}
Having established the lemmas of the previous sections, we are ready to give the proof of the main proposition.
\begin{prop}\label{bddF}
Fix $\mu \in(0,1)$. Assume we have a finitely-connected, bounded domain $\Omega$ with $C^2$ non-self-intersecting boundary, and let $F$ be as in Definition \ref{F2}, with 
\ba\label{vpatchdata}
\omega(\al):=\begin{cases}
1 & \mbox{if }\al\in\Omega, \\
0 & \mbox{otherwise.}
\end{cases}
\ea
Then there exists a $\delta>0$ such that
\ba
\sup_{X\in B_\delta}\lVert F(X) \rVert_{C^{1,\mu}(\Omega)} \leq C,
\ea
where $C$ depends only on $\Omega$.
\end{prop}
\begin{proof}

Let $X\in B_\delta$. To bound $\lVert F_X\rVert_{L^\infty(\Omega)}$, we simply observe that \eqref{aa1'} gives a uniform bound. For the $C^1(\Omega)$ norm, we use the uniform bound on $|\na F_X(\al)|$ following from Lemma \ref{2ndBOUND}.
%
Thus it only remains to give the $[\na F_X]_\mu$ estimate. First we establish the pair of estimates
\begin{eqnarray}
|T(F_X)(\al)-T(F_X)(\ga)|&\leq& C|\al-\ga|^\mu,\label{fin201}\\
|N(F_X)(\al)-N(F_X)(\ga)|&\leq& C|\al-\ga|^\mu,\label{fin202}
\end{eqnarray}
for $\al,\ga\in\Omega\setminus\cN$, where $T$ is the vector field given by Lemma \ref{vfpair}, $N=T^\perp$, and $\cN$ is a small neighborhood of the set of zeros of $T$, which we recall is finite.

Lemma \ref{mainHbound} establishes the H\"older bound for $T(F_X)$ with \eqref{p001}. Recall the formula for $N(F_X)$ given by Lemma \ref{diffFormula}. From that, and from the observation that outside of $\cN$ we have the desired H\"{o}lder continuity of $\cO$, $P_X$, and $Q^{-1}_X$ (since $T$ is $C^1$ and $\na X$ is H\"{o}lder continuous), we get that \eqref{fin201} also implies the estimate \eqref{fin202} for $\al,\ga\in\Omega\setminus\cN$.

Now note $\cO(\al)$ is invertible at any $\al\in\Omega\setminus\{\al:T(\al)=0\}$, and
\ba\label{p0002}
\na F_X(\al)=
(\cO(\al))^{-1}
\left(
\begin{array}{ccc}
T(F_X)(\al) \\
N(F_X)(\al)     \end{array} \right).
\ea
With \eqref{p0002}, and the H\"{o}lder continuity of $\cO^{-1}$ outside of $\cN$, we then find that this implies for some new constant $C$
\ba\label{p0003}
|\na F_X(\al)-\na F_X(\ga)|\leq C|\al-\ga|^\mu
\ea
for all $\al,\ga\in\Omega\setminus\cN$. Applying Lemma \ref{mainHbound} and Lemma \ref{diffFormula} with the vector field $T'$ from Lemma \ref{vfpair} and $N'=T'^\perp$ gives us the same estimate for all $\al,\ga\in\Omega\setminus\cN'$, for a neighborhood $\cN'$ of the set of zeros of $T'$. Then, since we are free to select the neighborhoods $\cN$ and $\cN'$ in such a way that they are disjoint, we have the H\"{o}lder estimate \eqref{p0003} in all of $\Omega$. We conclude
\ba
\lVert F_X\rVert_{C^{1,\mu}(\Omega)}\leq C.
\ea
\end{proof}


Now applying Theorem \ref{mainth} in the case $Y=C^{1,\mu}(\Omega)$, we obtain the Lagrangian analyticity result for vortex patches.
\begin{theo}\label{final}
Assume that $\Omega$ is a bounded open set which is finitely-connected, and which has $C^2$ non-self-intersecting boundary. We define $Y=C^{1,\mu}(\Omega)$. Then there is an $\eps>0$ and an $r>0$ with a unique solution to \eqref{000'} $X(\al,t)$ in $\cB_{\eps,r}$, as defined in Definition \ref{banach'}. The function $X(\al,t)$ is the trajectory map for the solution to \eqref{eulereqns} with initial vorticity given by \eqref{vpatchdata}, and for each $\al\in\overline{\Omega}$, the trajectory $X(\al,t)$ is analytic in $t$.
\end{theo}
\begin{proof}
The result follows from Proposition \ref{bddF}, Lemma \ref{vPreservesA'}, and Theorem \ref{mainth}. In particular, we find directly that the points inside the patch $\Omega(t)=X(\Omega,t)$ have analytic trajectories. To see the analyticity of the particle trajectories at the boundary, consider an $\al\in \del \Omega$. Taking a sequence of $\al_k$ tending to $\al$, the analytic functions $X(\al_k,t)$ are uniformly bounded and converge to $X(\al,t)$, which then must also be analytic.
\end{proof}

\end{subsubsection}
\end{subsection}
\end{section}
\nocite{serfatiT, cvk,bf,cvw,fz,fz2,shnirelman,kv,kv2,bb,isett,chemin,wollman,mb,gst,katoA}



\begin{section}{Lagrangian analyticity for the Euler-Poisson system}
The pressureless Euler-Poisson equation gives a model for the dynamics of a fluid in which the fluid particles experience long-range interaction. In particular, it models stellar dynamics, in which case the force between particles is attractive, or the dynamics of a charged fluid, in which case it is repulsive. Specifically, we consider the three dimensional case. The system is given by
\begin{eqnarray}\label{epfinal01a}
\del_t\rho + u \cdot \na \rho &=&-\rho \div u,\\
\del_t u+u\cdot \na u &=& q\na\Delta^{-1}\rho,\label{epfinal01b}\\
(\rho,u)&=&(\rho_0,u_0)\mbox{ at }t=0.\nn
\end{eqnarray}
where $q=\pm 1$, with unknowns density $\rho:\R^3\times\R\to\R^+$ and velocity $u:\R^3\times\R\to\R^3$, and with initial data $\rho_0\in C^s_c$ and $u_0\in H^s$ for $s\geq 6$. In the case $q=1$, the force is repulsive, and in the case $q=-1$, it is attractive. Here we explicitly handle the case of a repulsive force, i.e. $q=1$, but the results carry over to the $q=-1$ case with essentially no changes to the proofs.

In this section, using a strategy based on that in Section 2, with minor modifications, we will demonstrate that we also have analyticity of the Lagrangian trajectories in $t$ for this fluid mechanics model, for initial data of the regularity prescribed above.

We remark that in the study \cite{fz}, a Lagrangian analyticity-type property is verified for a different Euler-Poisson system, modeling the dynamics of a self-gravitating fluid in a cosmological setting. There the authors use recursive relations for Taylor coefficients to prove the trajectories are analytic as functions of a kind of time parameter, as long as the initial data for the system is slightly better than differentiable. For several reasons the work of \cite{fz} does not apply to the Lagrangian analyticity problem treated here. To highlight one difference, we remark that in \cite{fz}, analyticity is proven for their Euler-Poisson type model not with respect to the typical physical time variable, say $t$, but instead with respect to a parameter $\tau$ proportional to $t^{2/3}$, and analyticity in the typical time $t$ is not obtained, as it is in our situation.

\begin{subsection}{Lagrangian formulation}
Consider a solution $(\rho,u)$ to the Euler-Poisson equations. For each $\al\in\R^3$ the Lagrangian trajectory $X(\al,t)$ is then defined as the solution of the ODE
\begin{eqnarray}\label{epfinal02b}
\frac{dX}{dt}(\al,t)&=&u(X(\al,t),t),\\
X(\al,0)&=&\al.\nn
\end{eqnarray}
One can show that a consequence of \eqref{epfinal01a} together with \eqref{epfinal02b} is that
\ba\label{lasteqn}
\rho(X(\al,t),t))=(\det\na X(\al,t))^{-1}\rho_0(\al).
\ea
For the kernel
\ba\label{epfinal03}
K(y)=\frac{1}{4\pi}\frac{y}{|y|^3},
\ea
from \eqref{epfinal01b}, we find
\begin{eqnarray}\label{NewEp02}
\frac{d^2 X}{dt^2}(\al,t)&=&
\int K(X(\al,t)-y)\rho(y,t)dy,\\
&=&
\int K(X(\al,t)-X(\be,t))\rho(X(\be,t),t)\det \na X(\be,t)d\be,\nn\\
&=&
\int K(X(\al,t)-X(\be,t))\rho_0(\be)d\be.\nn
\end{eqnarray}
Thus, we can write the Lagrangian formulation as a first order ODE in the following way:
\begin{eqnarray}
\frac{dX}{dt}(\al,t)&=&V(\al,t),\\
\frac{dV}{dt}(\al,t)&=&\int K(X(\al,t)-X(\be,t))\rho_0(\be)d\be,\nn
\end{eqnarray}
with initial data $X|_{t=0}=Id$, $V|_{t=0}=u_0$. In the next section, we will define the corresponding operator $F$ on the appropriate Banach space so that the above system can be written as
\begin{eqnarray}\label{NewEp01}
\frac{dX}{dt}(\al,t)&=&V(\al,t),\\
\frac{dV}{dt}(\al,t)&=&F(X)(\al,t),\nn\\
(X,V)|_{t=0}&=&(Id,u_0),\nn
\end{eqnarray}
given initial data $\rho_0\in C^s_c$ and $u_0\in H^s$ for $s\geq 6$ for \eqref{epfinal01a}-\eqref{epfinal01b}.

Here we record the main result of this section, Theorem \ref{maintheoEP}.
\begin{theonn}
Consider a solution $(\rho,u)$ to \eqref{epfinal01a}, with initial data $\rho_0\in C^s_c$, $u_0\in H^s$ for $s\geq 6$. The corresponding trajectory map $X(\al,t)$ solving \eqref{epfinal02b} is analytic in $t$ at $t=0$ for each $\al\in\R^3$.
\end{theonn}

\end{subsection}

\begin{subsection}{Definitions of operators}

\begin{subsubsection}{Defining the operator $F(X)$}
In the following we have a fixed $\rho\in C^s_c(\R^3)$, with some fixed $s\geq 6$, where $\rho$ plays the role of the initial density for the system \eqref{epfinal01a}.

Now we define the $F$ for the case of the $3$D Euler-Poisson system. We consider the vector valued functions on $\R^3$ which take values in $\C^3$, in particular such maps sufficiently close to the identify map $Id$ in the Sobolev norm $H^s(\R^3)$, $s\geq 6$.

\ba
B_\delta=\{\C^3 \textrm{-valued } X \textrm { defined on }\R^3: \lVert X-Id \rVert_{H^s(\R^3)}\leq \delta \}.
\ea

We have the kernel
\ba\label{66}
K(z):=\frac{1}{4\pi}\frac{(z_1,z_2,z_3)}{(z^2_1+z^2_2+z^2_3)^{3/2}}\quad \mbox{for }z\in\C^3\setminus\{z^2_1+z^2_2+z^2_3\leq 0\},
\ea
extended analytically, where we choose to take the branch cut for the square root along the negative reals.
Now we record the fact
\ba\label{closetoid}
X(\al)-X(\be)=(\al-\be)\int^1_0 \na X(\tau \al + (1-\tau) \be) d\tau =(\al-\be)(I+O(\delta)).
\ea
Note that we use the convention $(\na X )_{ij}=\del X_j/\del \al_i$.

Now observe
\begin{eqnarray}\label{71}
\sum^3_{i=1}(X_i(\al)-X_i(\be))^2&=&\sum^3_{i=1} [(\al-\be)\int^1_0 \na X(\tau \al+(1-\tau)\be) d\tau]^2_i,\\
&=&\sum^3_{i=1}[(\al-\be)(I+O(\delta))]^2_i=|\al-\be|^2(1+O(\delta))\nonumber.
\end{eqnarray}
Taking the real part of both sides of the equation gives
\ba\label{72''}
\textrm{Re}\left(\sum^3_{i=1}(X_i(\al)-X_i(\be))^2\right)=|\al-\be|^2\textrm{Re}(1+O(\delta))\geq |\al-\be|^2/2,
\ea
for small enough $\delta$. Thus we may plug $z=X(\al)-X(\be)$ into $K(z)$ for any $X\in B_\delta$, and $\al,\be\in\R^3$, and we have no problems with the proposed branch cut for the square root in the definition of $K(z)$.

One verifies from \eqref{71} that the singularity along the diagonal in $K(X(\al)-X(\be))$ is no worse than that of $|\al-\be|^{-2}$, i.e.
\ba\label{72'''}
|K(X(\al)-X(\be))|\leq C |\al-\be|^{-2}
\ea
for some constant $C$, which is an integrable singularity in three dimensions.

\begin{definition}\label{Fd}
Fix a scalar, real valued function $\rho\in H^s(\R^3)\cap C_c(\R^3)$. For $X\in B_\delta$, provided $\delta$ is sufficiently small, we define $F(X)(\al)$ for each $\al\in\R^3$ by
\ba\label{00}
F(X)(\al)=\int
	K(X(\al)-X(\be)) \rho(\beta) d\be.
\ea
\end{definition}

\end{subsubsection}

\begin{subsubsection}{Relevant integral kernel operators}
Taking spatial derivatives of $F(X)(\al)$ for $X\in B_\delta$ naturally leads us to some singular integral operators of the form 
\ba\label{defT}
(T_{\cK}f)(\al)=p.v. \int f(\be)\cK(\al,\be)  d\be,
\ea
for a scalar $f:\R^3\to\C$ in $C^1(\R^3)$, and kernel $\cK :(\R^3\times\R^3\setminus\{\al=\be\})\to \C$, where $|\cK(\al,\be)|\leq C |\al-\be|^{-3}$. The $T_{\cK}f$ will then be defined pointwise, given basic assumptions are satisfied for a given $\cK(\al,\be)$. We extend this definition naturally for vector valued kernels and matrix valued kernels.

We will also handle related operators from $L^2(\R^3)$ to $L^2(\R^3)$. It will assist us to have the following facts recorded, which give bounds for these kinds of integral kernels. The proofs of the boundedness of these operators are delegated to the final subsection on the Euler-Poisson model.
\begin{lem}\label{kernelbound'}
For $M\in(\C^3)^2$ with $|M|\leq \delta'$, where $\delta'$ is a small constant, the function $K(z):\C^3\setminus 0\to \C^3$  defined by  
\eqref{66}, satisfies
\ba\label{Kbound}
|(\del^\ga_z K)(\be(I+M))| &\leq& A_\ga|\be|^{-2-|\ga|}\quad\textrm{for all }\be\in\R^3\textrm{ and }|\ga|\geq 0.
\ea
\end{lem}
\begin{proof}
The proof of this is is a straightforward calculation.
\end{proof}
\begin{definition}\label{def001}
Consider a map $X\in B_\delta$ for $\delta>0$. Provided $\delta$ is sufficiently small, we make the following definitions.

(i) For each $k=1,2,3$, for $\al,\be\in\R^3$, $\al\neq \be$, we define
\ba\label{m001}
\tilde K^k_X(\al,\be):=\sum^3_{i=1} \del_k X_i(\be)(\del_i K)(X(\al)-X(\be)),
\ea
and $K_X(\al,\be):=(K^1_X(\al,\be),K^2_X(\al,\be),K^3_X(\al,\be))$.

(ii)
We define for $f\in L^2\cap C^1$
\ba\label{m002}
(T_{X,k} f)(\al):=p.v.\int \tilde K^k_X(\al,\be)f(\be)d\be\quad\mbox{ for }\al\in\R^3,
\ea
and $(T_{X} f)(\al):=((T_{X,1} f)(\al),(T_{X,2} f)(\al),(T_{X,3} f)(\al))$.

(iii)
We define for $f\in L^2$ and $\eps>0$
\ba\label{m003}
(T^\eps_{X,k}f)(\al):=\int_{|\al-\be|>\eps} \tilde K^k_X(\al,\be)f(\be)d\be,
\ea
and $(T^\eps_{X}f)(\al):=((T^\eps_{X,1}f)(\al),(T^\eps_{X,2}f)(\al),(T^\eps_{X,3}f)(\al))$. It is proved in Lemma \ref{lemmaA01} that for each $k$, $T^\eps_{X,k}f$ converges in $L^2$ as $\eps$ tends to zero. We define
\ba
T^0_{X,k}f:=\lim_{\eps\to 0}T^\eps_{X,k}f\quad\mbox{ in }L^2,
\ea
and $T^0_{X}f:=(T^0_{X,1}f,T^0_{X,2}f,T^0_{X,3}f)$.
\end{definition}
\begin{rem}\label{rem001}
It is proved in the subsection on integral kernel operator bounds at the end of our section on the Euler-Poisson model that as long as $\delta$ is sufficiently small, for each $k$, the p.v. limit defining $T_{X,k}$ converges uniformly in $\al$, and that $T^\eps_{X,k}$ and $T^0_{X,k}$ as defined are bounded operators from $L^2$ to $L^2$, with Lemmas \ref{kernellem1'} and \ref{lemmaA01}.
Also as a consequence of these lemmas we have the bounds
\ba\label{m005}
\lVert T^\eps_{X} f \rVert_{L^\infty}\leq C( \lVert f \rVert_{L^2}+\lVert f \rVert_{C^1})\quad\mbox{for all }f\in L^2\cap C^1,
\ea
and
\ba
\lVert T_{X} f \rVert_{L^\infty}\leq C( \lVert f \rVert_{L^2}+\lVert f \rVert_{C^1})\quad\mbox{for all }f\in L^2\cap C^1,
\ea
and the $L^2$ bounds
\ba\label{m006}
\lVert T^\eps_{X}f\rVert_{L^2}
\leq
C\lVert f \rVert_{L^2}\quad\mbox{for all }f\in L^2,
\ea
for $\eps>0$, and
\ba
\lVert T^0_{X}f\rVert_{L^2}
\leq
C\lVert f \rVert_{L^2}\quad\mbox{for all }f\in L^2.
\ea
\end{rem}

\end{subsubsection}

\end{subsection}

\begin{subsection}{Differentiation formulae}
The goal of this section is to give a formula for taking any number of derivatives of the $F(X)$ of Definition $\ref{Fd}$ which results in a compact expression well-suited for making $L^2$ estimates. Using this formula, we readily bound the Sobolev norms of $F(X)$, which then gives us the local boundedness estimate for the operator $F$. The proof of analyticity of trajectories is easily shown once this is verified.

Fix $k=1,2$ or $3$. Observe that for an $X\in B_\delta$ we may regard the quantity
\ba\label{ep001}
\cK_X(\al,\be):=
\del_{\beta_k}\left(
	K(X(\al)-X(\al-\be))
\right)
\ea
as a kernel, with
\ba
|\cK_X(\al,\be)|\leq C|\be|^{-3}.
\ea
Objects like this occur in integrals like in \eqref{m002} and their derivatives, which is apparent if one first makes a change of variable in the integral replacing $\be$ with $(\al-\be)$, then differentiates with respect to some $\al$ component, and manipulates the resulting expression. We elaborate on this technique further in what follows; it provides a useful technique for calculating high order derivatives of the Biot-Savart type integrals composed with trajectory maps, such as \eqref{m002}, in particular, and we remark that it can be adapted to the analogous operators for the incompressible Euler equations. The goal of this part of the section is to give a formula for taking any number of derivatives of $F(X)$, as in Definition $\ref{Fd}$, resulting in a compact expression well-suited for making $L^2$ estimates. Using this, together with the singular integral operator estimates given at the end of the section, we readily bound the Sobolev norms of $F(X)$ as a corollary, which then gives us the local boundedness estimate for the operator $F$. The proof of analyticity of trajectories is easily shown once that is verified.

One can take derivatives of integral convolution type expressions with kernel $\cK_X(\al,\be)$ as written in \eqref{ep001}, potentially repeatedly, in a way that yields a sum of expressions of a similar form, with nearly the same kernel as that in \eqref{ep001}.
We give a rigorous derivation of the relevant differentiation formula with Proposition \ref{prop001}. In the following discussion leading up to the statement of the proposition, we sketch out the main idea of the way the manipulations work, to motivate Proposition \ref{prop001} and lemmas leading up to the differentiation formula.

We begin by considering functions $f\in C^2_c$, $g\in H^1$. As our candidate expression to be differentiated, let us take the following, without being too careful about the precise meaning of the principal value integral in this instance.
\ba
T(\al):=
\mbox{p.v.}
\int
(g(\al)-g(\al-\be))f(\al-\be)
\del_{\beta_k}\left(
	K(X(\al)-X(\al-\be))
\right)d\be.
\ea
Then if we differentiate with respect to $\al_l$, for some $l=1,2$, or $3$, and, without justification at the moment, move the differentiation under the integral sign, we find
\begin{eqnarray}
\del_l T(\al)
&=&
\mbox{p.v.}
\int
\dal
\left[(g(\al)-g(\al-\be))f(\al-\be)
\right]
\del_{\beta_k}\left(
	K(X(\al)-X(\al-\be))
\right)d\be\\
& &+
\mbox{p.v.}
\int
(g(\al)-g(\al-\be))f(\al-\be)
\del_{\beta_k}\left(
	\dal
	[K(X(\al)-X(\al-\be))]
\right)d\be.\nn
\end{eqnarray}
Now we integrate by parts in the second integral on the right. This in fact yields no boundary term; we justify this for these types of calculations in the rigorous derivations coming later in the section.
\begin{eqnarray}
\del_l T(\al)
&=&
\mbox{p.v.}
\int
\dal
\left[(g(\al)-g(\al-\be))f(\al-\be)
\right]
\del_{\beta_k}\left(
	K(X(\al)-X(\al-\be))
\right)d\be\\
& &-
\mbox{p.v.}
\int
\del_{\beta_k}
\left[(g(\al)-g(\al-\be))f(\al-\be)\right]
	\dal
	[K(X(\al)-X(\al-\be))]
d\be.\nn
\end{eqnarray}
Now, we use the next identity, following just from the chain rule
\begin{eqnarray}
\na_{\al}
(K(X(\al)-X(\al-\be)))
&=&
(\na X(\al)-\na X(\al-\be)) \na K(X(\al)-X(\al-\be)),\nn\\
&=&
(\na X(\al)-\na X(\al-\be)) (\na X(\al-\be))^{-1}\na_\be( K(X(\al)-X(\al-\be)))\nn.
\end{eqnarray}
Using this we obtain
\begin{eqnarray}
\del_l T(\al)
&=&
\mbox{p.v.}
\int
\dal
\left[(g(\al)-g(\al-\be))f(\al-\be)
\right]
\del_{\beta_k}\left(
	K(X(\al)-X(\al-\be))
\right)d\be\\
& &-
\sum^3_{i,j=1}
\mbox{p.v.}
\int
\del_{\beta_k}
\left[(g(\al)-g(\al-\be))f(\al-\be)\right]\nn\\
& &\hspace{.8in}\times
(\del_l X_i(\al)-\del_l X_i(\al-\be))
(\na X(\al-\be))^{-1}_{ij}
\del_{\be_j}(K(X(\al)-X(\al-\be))
d\be.\nn
\end{eqnarray}
Thus we have an expression for the derivative in terms of integrals with kernels just of the form \eqref{ep001}. Given higher differentiability in $f$ and $g$, one can repeat the process over and over to calculate higher order derivatives of $T(\al)$, in terms of the corresponding derivatives of $f$, $g$, and $X$, integrated against the one type of integral kernel. The main gain of using this strategy is just that we essentially only need to verify $L^2$ estimates for the one kernel that appears.

This outlines the idea behind the proposition that follows. The proposition is used to verify that computations that arise for us in our estimates of $\lVert F(X)\rVert_{H^s}$ can basically be handled in the same way, with the resulting principal value integrals converging in the $L^2$ sense.

\begin{prop}\label{prop001}
Fix $X\in B_\delta$. We define the spaces
\begin{eqnarray}
\begin{array}{ccc}
\cZ_1=C^1\times C^1\times H^1, &\cZ_2=C^1\times H^1\times C^2, &\cZ_3=H^1\times C^2\times C^2,\\
\cZ_1'=L^\infty\times L^\infty\times L^2, &\cZ_2'=L^\infty\times L^2\times C^1, &\cZ_3'=L^2\times C^1\times C^1,
\end{array}
\end{eqnarray}
where by $C^k$ we mean the set of $k$-times continuously differentiable functions on $\R^3$ with bounded partial derivatives up to order $k$, and we define for $i=1,2,$ and $3$,
\ba
\lVert
(h,g,f)
\rVert_{\cZ_i'}:=
\begin{cases}
\lVert h\rVert_{L^\infty}\cdot\lVert g\rVert_{L^\infty}\cdot\lVert f\rVert_{L^2} & \mbox{if }i=1,\\
\lVert h\rVert_{L^\infty}\cdot\lVert g\rVert_{L^2}\cdot\lVert f\rVert_{C^1} & \mbox{if }i=2,\\
\lVert h\rVert_{L^2}\cdot\lVert g\rVert_{C^1}\cdot\lVert f\rVert_{C^1} & \mbox{if }i=3,
\end{cases}
\quad\quad\mbox{for }(h,g,f)\in \cZ_i'.\quad\quad
\ea
Fix a triple $(h,g,f)\in \cZ_1'$, $\cZ_2'$, or $\cZ_3'$ for which $\mbox{supp}(f)$ is bounded, and where $h,g$, and $f$ are functions defined almost everywhere on $\R^3$.
For any $\eps>0$ we write $B^c_\eps(\al)$ to indicate the complement in $\R^3$ of the ball $B_\eps(\al)$ about $\al\in\R^3$.

(i)
For $\eps>0$ we have the a.e. pointwise-defined function
\ba\label{m007}
(T^\eps_{X}[h,g,f])(\al)
:=
h(\al)\int_{B^c_\eps(\al)}(g(\al)-g(\be))f(\be)\tilde K_X(\al,\be)d\be,
\ea
and we have
\ba\label{m008}
\lim_{\eps\to 0}(T^\eps_{X}[h,g,f])\quad\mbox{ converges in }L^2.
\ea
Defining the limit
\ba
(T^0_{X}[h,g,f]):=\lim_{\eps\to 0}(T^\eps_{X}[h,g,f])\quad\mbox{ in }L^2,
\ea
we have
\ba
\lVert
T^0_{X}[h,g,f]
\rVert_{L^2}
\leq C(\emph{supp}(f))\lVert(h,g,f)\rVert_{\cZ_i'},
\ea
for any $(h,g,f)\in \cZ_i'$, for $i=1,2,$ or $3$. In the above bound the factor in front depends only on the support of $f$.

(ii)
If in addition the triple $(h,g,f)$ is in $\cZ_i$ for some $i=1,2,$ or $3$, we have that $(T^0_{X}[h,g,f])$ is in $H^1$, and for $k=1,2$, or $3$, and $l=1,2$, or $3$, we have the formula
\begin{eqnarray}\label{m009}\\
\del_{l}((T^0_{X,k}[h,g,f])(\al))
&=&
\lim_{\eps\to 0}
\bigg(
\int_{B^c_\eps(0)}
\dal\big(h(\al)(g(\al)-g(\al-\be))f(\al-\be)\big)\tilde K^k_X(\al,\al-\be)
d\be\nn\\
& &
\hspace{0in}-
h(\al)\sum^3_{i,j=1}
\int_{B^c_\eps}
\del_{\be_k}
\big((g(\al)-g(\al-\be))f(\al-\be)\big)(\del_l X_i(\al)-\del_l X_i(\al-\be))\nn\\
& &\hspace{1.4in}\times (\na X(\al-\be))^{-1}_{ij}\tilde K^k_X(\al,\al-\be)
d\be
\bigg)\quad\mbox{ in }L^2.\nn
\end{eqnarray}

\end{prop}
\begin{proof}

Fix a triple $(h,g,f)\in \cZ_1'$, $\cZ_2'$, or $\cZ_3'$, with each of $f$, $g$, and $h$ defined almost everywhere, and with $\mbox{supp}(f)$ bounded. Fix an $l=1,2$, or $3$. The first thing we note is that for each $\eps>0$, for each $\al$ for which both $h(\al)$ and $g(\al)$ are defined (guaranteed to hold for almost all $\al$), the integrand in the definition of $(T^\eps_{X,k}[h,g,f])(\al)$ is clearly integrable over the region of integration. Thus this gives a pointwise defined function almost everywhere. 

For \eqref{m008}, notice that for each $\eps>0$, a change of variables gives the following, where we simply write $B^c_\eps$ to indicate the complement of the ball $B_\eps(0)$ about the origin.
\begin{eqnarray}\label{m009001}
(T^\eps_{X,k}[h,g,f])(\al)&=&
h(\al)g(\al)\int_{B^c_\eps} f(\al-\be)\del_{\be_k}(K(\theta(\al,\be)))d\be\\
& &-h(\al)\int_{B^c_\eps}g(\al-\be)f(\al-\be)\del_{\be_k}(K(\theta(\al,\be)))d\be,\nn\\
&=&
h(\al)g(\al)\int_{(B_\eps(\al))^c} f(\be)\tilde K^k_X(\al,\be) d\be-h(\al)\int_{(B_\eps(\al))^c}g(\be)f(\be)\tilde K^k_X(\al,\be)d\be,\nn\\
&=&
h(\al)g(\al)( T^\eps_{X,k} f)(\al)-h(\al)(T^\eps_{X,k}(gf))(\al).\nn
\end{eqnarray}
Suppose $(h,g,f)\in \cZ_1'$. Then we find $f\in L^2$ and $gf\in L^2$, so $T^\eps_{X,k} f$ and $T^\eps_{X,k}(gf)$ both converge in $L^2$ as $\eps$ tends to zero, in view of Definition \ref{def001} and Remark \ref{rem001}. At the same time $h$ and $g$ are in $L^\infty$, so the entire right hand side of the last equality of \eqref{m009001} converges in $L^2$ as $\eps$ tends to zero. One easily verifies then
\ba
\lVert
T^0_{X,k}[h,g,f]
\rVert_{L^2}
\leq
C\lVert
(h,g,f)
\rVert_{\cZ_1'}
\ea

Suppose $(h,g,f)\in \cZ_2'$. Then $f\in C^1_c$ and $hg\in L^2$. So by Definition \ref{def001} and Remark \ref{rem001}, $T^\eps_{X,k} f$ converges in $L^\infty$ and thus $(hg\cdot ( T^\eps_{X,k} f))$ converges in $L^2$ norm as $\eps$ tends to zero. Also, we have $h\in L^\infty$ and $gf\in L^2$, so $(h\cdot (T^\eps_{X,k}(gf))$ also converges in $L^2$ in the limit, and again we find the entire right hand side converges in $L^2$. In this case we find
\begin{eqnarray}
\lVert
T^0_{X,k}[h,g,f]
\rVert_{L^2}
&\leq&
\lVert h \rVert_{L^\infty}
\lVert g \rVert_{L^2}(\lVert f \rVert_{C^1}+\lVert f \rVert_{L^2})
+
\lVert h \rVert_{L^\infty}
\lVert g \rVert_{L^2}\lVert f \rVert_{L^\infty},\\
&\leq&
C(\mbox{supp}(f))\lVert(h,g,f)
\rVert_{\cZ_2'}.\nn
\end{eqnarray}

Now suppose $(h,g,f)\in \cZ_3'$. Then we have $hg\in L^2$ and $f\in C^1_c$, and we have $h\in L^2$ and $gf\in C^1_c$, so that we find the right hand side converges in $L^2$ as $\eps$ tends to zero similarly. Finally, we have
\begin{eqnarray}
\lVert
T^0_{X,k}[h,g,f]
\rVert_{L^2}
&\leq&
\lVert h \rVert_{L^2}
\lVert g \rVert_{L^\infty}(\lVert f \rVert_{C^1}+\lVert f \rVert_{L^2})
+
\lVert h \rVert_{L^2}
(
\lVert g \rVert_{C^1}\lVert f \rVert_{C^1}
+
\lVert g \rVert_{L^\infty}\lVert f \rVert_{L^2})
,\nn\\
&\leq&
C(\mbox{supp}(f))\lVert(h,g,f)
\rVert_{\cZ_3'}.
\end{eqnarray}

This completes the proof of part (i). Now we begin with part (ii).

Suppose the triple $(h,g,f)$ is in one of $\cZ_1$, $\cZ_2$, or $\cZ_3$. First we consider the $\del_l$ derivative for a fixed $\eps>0$ of $(T^\eps_{X,k}[h,g,f])$. We claim that for any test function $\varphi\in C^\infty_c$ the following holds, where $\la \cdot,\cdot\ra$ indicates the $L^2$ inner product.
\begin{eqnarray}\label{m009002}
& &-\la\del_l \varphi, (T^\eps_{X,k}[h,g,f])\rangle=\\
& &\hspace{.5in}
\int \varphi(\al)
\int_{B^c_\eps}
\dal\big(h(\al)(g(\al)-g(\al-\be))f(\al-\be)\big)\del_{\be_k}(K(\theta(\al,\be)))
d\be d\al \nn\\
& &\hspace{.4in}
-\sum^3_{i,j=1}\int \bigg(\varphi(\al)
h(\al)
\int_{B^c_\eps}
\del_{\be_k}
\big((g(\al)-g(\al-\be))f(\al-\be)\big)\nn\\
& &\hspace{1.6in}\times(\del_l X_i(\al)-\del_l X_i(\al-\be))
(\na X(\al-\be))^{-1}_{ij}\del_{\be_j}(K(\theta(\al,\be)))
d\be
\bigg)d\al\nn\\
& &\hspace{.5in}
+\la \varphi , \tilde B^\eps_{\textit{boundary}}\ra,\nn
\end{eqnarray}
where $\tilde B^\eps_{\textit{boundary}}$ is a boundary term resulting from an integration by parts in an integral over $\be$ that arises in the calculation. Compare this with \eqref{m009}. We need to verify that \eqref{m009002} holds, where the right hand side forms an $L^2$ pairing of $\varphi$ with a function in $L^2$. We also need to verify that this function that is paired with $\varphi$ converges to an element in $L^2$ as $\eps$ tends to zero, and that $\tilde B^\eps_{\textit{boundary}}$ tends to zero in $L^2$. If we can show these things, it follows that
\ba
-\la\del_l \varphi, (T^0_{X,k}[h,g,f])\rangle=-\lim_{\eps\to 0}\la\del_l \varphi, (T^\eps_{X,k}[h,g,f])\rangle=\la \varphi,\textrm{RHS of \eqref{m009}} \ra.
\ea
Since $\varphi$ was an arbitrary test function, this implies that the right hand side of \eqref{m009} is the weak derivative of $T^0_{X,k}[h,g,f]$, and that this derivative is in $L^2$. This allows us to conclude the proof of part (ii).

Now we give the calculation for \eqref{m009002}. For $\eps>0$, in the following we use $\chi_\eps$ to indicate the identifier function $1_{B^c_\eps}$ for the set $B^c_\eps=\R^3\setminus (B_\eps(0))$.
\begin{eqnarray}\label{m009003}
& &-\la
\del_l \varphi,
T^\eps_{X,k}[h,g,f])
\ra\\
& &\hspace{0.7in}=
-\int\del_l \vp(\al)h(\al)
\int_{B^c_\eps}
(g(\al)-g(\al-\be))f(\al-\be)\dbk (K(\theta(\al,\be)))d\be d\al,\nn\\
& &\hspace{0.7in}=
-\int\int
\del_l\vp(\al)h(\al)(g(\al)-g(\al-\be))f(\al-\be)\chi_\eps(\be) \dbk(K(\theta(\al,\be)))d\be d\al,\nn\\
& &\hspace{0.7in}=
-\int\int
\del_l\vp(\al)h(\al)(g(\al)-g(\al-\be))f(\al-\be)\chi_\eps(\be) \dbk(K(\theta(\al,\be)))d\al d\be,\nn\\
& &\hspace{0.7in}=
\int\int
\vp(\al)\dal\bigg(h(\al)(g(\al)-g(\al-\be))f(\al-\be)\chi_\eps(\be) \dbk(K(\theta(\al,\be)))\bigg)d\al d\be,\nn\\
& &\hspace{0.7in}=
\int\int
\vp(\al)\dal\bigg(h(\al)(g(\al)-g(\al-\be))f(\al-\be)\bigg)\chi_\eps(\be) \dbk(K(\theta(\al,\be)))d\al d\be\nn\\
& &\hspace{0.9in}+
\int\int
\vp(\al)(h(\al)(g(\al)-g(\al-\be))f(\al-\be))\chi_\eps(\be) \dbk(\dal (K(\theta(\al,\be))))d\al d\be,\nn\\
& &\hspace{0.7in}=
 A^\eps+B^\eps.\label{m009003c}
\end{eqnarray}
Regarding the third line in \eqref{m009003}, since the triple $(h,g,f)$ is in $\cZ_1$, $\cZ_2$, or $\cZ_3$, and $f$ is compactly supported, changing the order of integration is justified.

For $A^\eps$, we have
\begin{eqnarray}\label{m009004}\\
A^\eps
&=&
\int
\vp(\al)
\int_{B^c_\eps}
\bigg(
\del_l h(\al) (g(\al)-g(\al-\be))f(\al-\be) + h(\al)(\del_l g(\al)-\del_l g(\al-\be))f(\al-\be)\nn\\
& &\hspace{1in}
+h(\al)(g(\al)-g(\al-\be))\del_l f(\al-\be)
\bigg)
\dbk(K(\theta(\al,\be)))d\be d\al,\nn\\
&=&
\int\vp(\al)
\bigg(
(\del_l h(\al)g(\al)+h(\al)\del_l g(\al))\int_{B^c_\eps}f(\al-\be)\dbk(K(\theta(\al,\be)))d\be\nn\\
& &\hspace{1in}
-\del_l h(\al)
\int_{B^c_\eps}g(\al-\be)f(\al-\be)\dbk(K(\theta(\al,\be)))d\be\nn\\
& &\hspace{1in}
-h(\al)
\int_{B^c_\eps}(\del_l g(\al-\be)f(\al-\be)+g(\al-\be)\del_l f(\al-\be)) \dbk(K(\theta(\al,\be)))d\be\nn\\
& &\hspace{1in}
+h(\al)g(\al)
\int_{B^c_\eps}\del_l f(\al-\be) \dbk(K(\theta(\al,\be)))d\be
\bigg)d\al,\nn\\
&=&
\int \vp(\al)\left(A^\eps_1(\al)+A^\eps_2(\al)+A^\eps_3(\al)+A^\eps_4(\al)\right)d\al.\nn
\end{eqnarray}
After making the change of variables replacing $\be$ with $(\al-\be)$ one then finds
\begin{eqnarray}\label{m009005}
A^\eps_1
&=&
(\del_l h \cdot g +h \cdot \del_l g )\cdot T^\eps_{X,k}f,\\
A^\eps_2
&=&
-\del_l h \cdot T^\eps_{X,k}(g\cdot f),\nn\\
A^\eps_3
&=&
-h \cdot T^\eps_{X,k}(\del_l g \cdot f+ g \cdot \del_l f),\nn\\
A^\eps_4
&=&
h\cdot g \cdot T^\eps_{X,k}(\del_l f).\nn
\end{eqnarray}
One uses Remark \ref{rem001} to verify that since $(h,g,f)\in \cZ_i$ for $i=1,2$, or $3$, with $\mbox{supp}(f)$ bounded, each of the $A^\eps_j$, $j=1,\ldots,4$, is in $L^2$.

For $B^\eps$ in \eqref{m009003c}, we have
\begin{eqnarray}\label{m009006}
B^\epsilon
&=&
\int\int_{B^c_\eps}
\vp(\al)h(\al)(g(\al)-g(\al-\be))f(\al-\be)\\
& &\hspace{.8in}\times \dbk
\bigg(
(\del_l X(\al)-\del_l X(\al-\be))(\na K)(\theta(\al,\be))
\bigg)
d\be d\al,\nn\\
&=&
-\int\int_{B^c_\eps}
\vp(\al)h(\al)\dbk
\bigg((g(\al)-g(\al-\be))f(\al-\be)
\bigg)\nn\\
& &\hspace{.8in}
\times(\del_l X(\al)-\del_l X(\al-\be))(\na K)(\theta(\al,\be))
d\be d\al\nn\\
& &+
\int\int_{|\be|=\eps}
\vp(\al)h(\al)(g(\al)-g(\al-\be))f(\al-\be)
\nn\\
& &\hspace{.8in}
\times
(\del_l X(\al)-\del_l X(\al-\be))(\na K)(\theta(\al,\be))\nu^k_{\textit{out}}(\be)
d\sigma(\be) d\al\nn,\\
&=&
B^\eps_{\textit{int}}+B^\eps_{\textit{boundary}}.\nn
\end{eqnarray}
In the above expression, we denote by $\nu_{out}$ the outward pointing normal. We note that the integration by parts above works since we can guarantee that $(g(\al)-g(\al-\be))f(\al-\be)$ is at least in $H^1_\be$ or in $C^1_\be$. Now we consider $B^\eps_{\textit{int}}$. Using the chain rule and the fact that $\na_\be \theta(\al,\be)=\na X(\al-\be)$ is invertible for any $\al,\be$, we can write this as
\begin{eqnarray}\label{m009007}
B^\eps_{\textit{int}}&=&
-\int\int_{B^c_\eps}
\vp(\al)h(\al)\dbk
\bigg((g(\al)-g(\al-\be))f(\al-\be)
\bigg)
(\del_l X(\al)-\del_l X(\al-\be))\\
& &
\hspace{1.7in}\times
(\na_\be\theta(\al,\be))^{-1}
\na_\be\big( K(\theta(\al,\be))\big)
d\be d\al\nn\\
& &=
-\sum^3_{i,j=1}\int\vp(\al)
h(\al)\int_{B^c_\eps}
\dbk
\bigg((g(\al)-g(\al-\be))f(\al-\be)
\bigg)(\del_l X_i(\al)-\del_l X_i(\al-\be))
\nn\\
& &\hspace{1.7in}
\times
(\na X(\al-\be))^{-1}_{ij}
\del_{\be_j}(K(\theta(\al,\be)))d\be d\al.\nn
\end{eqnarray}
Now we compute for $i=1,2$, and $3$,
\begin{eqnarray}\label{m009008}
& &
\dbk
\big[(g(\al)-g(\al-\be))f(\al-\be)
\big]
(\del_l X_i(\al)-\del_l X_i(\al-\be))\\
& &\hspace{0in}=
\big[\del_k g(\al-\be)f(\al-\be)-g(\al)\del_k f(\al-\be)+g(\al-\be)\del_k f(\al-\be)\big]
(\del_l X_i(\al)-\del_l X_i(\al-\be)),\nn
\end{eqnarray}
and then distributing multiplication gives
\begin{eqnarray}\label{m009009}\\
\textrm{RHS of \eqref{m009008}}=\hspace{.1in}
\del_l X_i(\al)\del_k g(\al-\be)f(\al-\be)
&-&\del_k g(\al-\be)f(\al-\be)\del_l X_i(\al-\be)\nn\\
-g(\al)\del_l X_i(\al) \del_k f(\al-\be)
&+&g(\al)\del_k f(\al-\be)\del_l X_i(\al-\be)\nn\\
+\del_l X_i(\al)g(\al-\be)\del_k f(\al-\be)
&-&g(\al-\be)\del_k f(\al-\be)\del_l X_i(\al-\be).\nn
\end{eqnarray}
We arrange this as
\begin{eqnarray}\label{m009010}\\
\textrm{RHS of \eqref{m009009}}
&=&
\del_l X_i(\al)\big[\del_k g(\al-\be)f(\al-\be)+g(\al-\be)\del_k f(\al-\be)\big]\nn\\
& &-\big[\del_k g(\al-\be)f(\al-\be)+g(\al-\be)\del_k f(\al-\be)\big]\del_l X_i(\al-\be)\nn\\
& &+g(\al)\del_k f(\al-\be) \del_l X_i(\al-\be)\nn\\
& &-g(\al)\del_l X_i(\al) \del_k f(\al-\be)\nn.
\end{eqnarray}
Using \eqref{m009008}-\eqref{m009010} in \eqref{m009007}, we get
\begin{eqnarray}\label{m009011}
& &B^\eps_{\textit{int}}=-\sum^3_{i,j=1}
\int\vp(\al)\bigg[
h(\al)\del_l X_i(\al)\\
& &
\hspace{.6in}\times
\int_{B^c_\eps}(\del_k g(\al-\be)f(\al-\be)+g(\al-\be)\del_k f(\al-\be))
(\na X(\al-\be))^{-1}_{ij}
\del_{\be_j}(K(\theta(\al,\be)))d\be\nn\\
& &\hspace{1.3in}
-h(\al)\int_{B^c_\eps}
(\del_k g(\al-\be)f(\al-\be)+g(\al-\be)\del_k f(\al-\be))\del_l X_i(\al-\be)\nn\\
& &\hspace{2.2in}
\times(\na X(\al-\be))^{-1}_{ij}
\del_{\be_j}(K(\theta(\al,\be)))d\be\nn\\
& &\hspace{1.3in}
+h(\al)g(\al)\int_{B^c_\eps}
\del_k f(\al-\be) \del_l X_i(\al-\be)
(\na X(\al-\be))^{-1}_{ij}
\del_{\be_j}(K(\theta(\al,\be)))d\be\nn\\
& &\hspace{1.3in}
-h(\al)g(\al)\del_l X_i(\al)\int_{B^c_\eps}
\del_k f(\al-\be)
(\na X(\al-\be))^{-1}_{ij}
\del_{\be_j}(K(\theta(\al,\be)))d\be
\bigg]d\al,\nn\\
& &=-\sum^3_{i,j=1}\int\vp(\al)(B^\eps_{1;ij}(\al)+B^\eps_{2;ij}(\al)+B^\eps_{3;ij}(\al)+B^\eps_{4;ij}(\al))d\al.\nn
\end{eqnarray}
With the change of variables replacing $\be$ with $(\al-\be)$, we see that the $B^\eps_{m;ij}$, $i,j=1,2,3$, $m=1,\ldots,4$ are thus given by
\begin{eqnarray}\label{m009012}
B^\eps_{1;ij}&=&
h\cdot\del_l X_i\cdot 
T^\eps_{X,j}\bigg(
(\del_k g\cdot f + g \cdot \del_k f)\cdot (\na X)^{-1}_{ij}
\bigg)
\\
B^\eps_{2;ij}&=&
-h\cdot
T^\eps_{X,j}\bigg(
(\del_k g \cdot f + g \cdot \del_k f)\cdot \del_l X_i\cdot (\na X)^{-1}_{ij}
\bigg)
\nn\\
B^\eps_{3;ij}&=&
h\cdot g\cdot
T^\eps_{X,j}\bigg(
\del_k f \cdot \del_l X_i \cdot (\na X)^{-1}_{ij}
\bigg)
\nn\\
B^\eps_{4;ij}&=&
-h\cdot g\cdot \del_l X_i\cdot 
T^\eps_{X,j}\bigg(
\del_k f\cdot (\na X)^{-1}_{ij}
\bigg)
\nn
\end{eqnarray}
One can consider each of the cases $(h,g,f)\in\cZ_1$, $(h,g,f)\in\cZ_2$, and $(h,g,f)\in\cZ_3$, and use the facts that $\mbox{supp}(f)$ is bounded and that $\na X$ and $(\na X)^{-1}$ are bounded in $C^1$, to bound each $B^\eps_{m;ij}$ in $L^2$.

Now we return to \eqref{m009006} to handle $B^\eps_{\textit{boundary}}$. We have
\begin{eqnarray}\label{m009013}
& &B^\eps_{\textit{boundary}}=\int
\vp(\al)h(\al)\\
& &\hspace{.2in}\times
\int_{|\be|=\eps}
(g(\al)-g(\al-\be))f(\al-\be)
(\del_l X(\al)-\del_l X(\al-\be))(\na K)(\theta(\al,\be))
\nu^k_{\textit{out}}(\be)
d\sigma(\be) d\al,\nn\\
& &\hspace{.6in}=
\int\vp(\al)\tilde B^\eps_{\textit{boundary}}(\al)d\al,\nn
\end{eqnarray}
where
\begin{eqnarray}\label{m009014}
& &\tilde B^\eps_{\textit{boundary}}(\al)=h(\al)C^\eps(\al),\\
& &C^\eps(\al)=\int_{|\be|=\eps}(g(\al)-g(\al-\be))f(\al-\be)
(\del_l X(\al)-\del_l X(\al-\be))(\na K)(\theta(\al,\be))
\nu^k_{\textit{out}}(\be)
d\sigma(\be),\nn
\end{eqnarray}
We comment that if we are in the case that the triple $(h,g,f)$ is in $\cZ_2$, we only assert $C^\eps(\al)$ is defined for almost all $\al$. Note that a similar caveat must be made about $\tilde B^\eps_{\textit{boundary}}(\al)$.

Now we check that $\tilde B^\eps_{\textit{boundary}}(\al)$ as defined is in fact in $L^2$. We consider first the cases that $(h,g,f)\in \cZ_1$ or $\cZ_2$.
\begin{eqnarray}\label{m009015}
& &\lVert \tilde B^\eps_{\textit{boundary}}\rVert_{L^2}\leq\lVert h\lVert_{L^\infty}\lVert C^\eps\rVert_{L^2},\\
& &\lVert C^\eps\rVert_{L^2}\leq\int_{|\be|=\eps}\lVert (g(\al)-g(\al-\be))f(\al-\be)\rVert_{L^2_\al}\nn\\
& &\hspace{1in}\times
\lVert(\del_l X(\al)-\del_l X(\al-\be))(\na K)(\theta(\al,\be))\rVert_{L^\infty_\al}
d\sigma(\be),\nn
\end{eqnarray}
where for the bound on $\lVert C^\eps\rVert_{L^2}$ we have majorized the quantity by moving the $L^2_\al$ norm inside the surface integral over $\be$, using a generalization of the Minkowski inequality. Then, by using Lemma \ref{kernelbound'} together with the fact that $\theta(\al,\be)=(X(\al)-X(\al-\be))=\be(I+O(\delta))$ and $(\del_l X(\al)-\del_l X(\al-\be))=O(|\be|)$, we get for all $\be\neq 0$ that
\ba\label{m009016}
\lVert(\del_l X(\al)-\del_l X(\al-\be))(\na K)(\theta(\al,\be))\rVert_{L^\infty_\al}\leq \frac{C}{|\be|^2}.
\ea
Meanwhile, the surface measure of the ball of radius $\eps$ is proportional to $\eps^2$, and so by using \eqref{m009016} in \eqref{m009015} we find
\ba\label{m009017}
\lVert C^\eps\rVert_{L^2}
\leq
C'\sup_{|\be|=\eps}\lVert
(g(\al)-g(\al-\be))f(\al-\be)\rVert_{L^2_\al}.
\ea
Suppose we are in the case that $(h,g,f)\in\cZ_1$. Then we can bound the right hand side of \eqref{m009017} by the finite quantity
\ba\label{m009018}
c_1(\eps):=C'\sup_{|\be|=\eps}\lVert
g(\al)-g(\al-\be)\rVert_{L^\infty_\al} \lVert f\rVert_{L^2}.
\ea
Moreover, note that $c_1(\eps)$ tends to zero as $\eps$ tends to zero, where we have used that $g$ is Lipschitz in this case. Thus we find that $\lim_{\eps\to 0}C^\eps=0$ in $L^2$.

Suppose $(h,g,f)\in\cZ_2$. Then the right hand side of \eqref{m009017} is bounded by the quantity
\ba\label{m009019}
C'\sup_{|\be|=\eps}\lVert
g(\al)-g(\al-\be)\rVert_{L^2_\al} \lVert f\rVert_{L^\infty}\leq
C''\eps\lVert
\na g \rVert_{L^2} \lVert f\rVert_{L^\infty}
=:c_2(\eps).
\ea
Note $c_2(\eps)$ tends to zero as $\eps$ tends to zero. This implies $\lim_{\eps\to 0}C^\eps=0$ in $L^2$.

Now we return to \eqref{m009014} in the case that $(h,g,f)\in\cZ_3$. Then we have the bound for all $\al\in R^3$
\begin{eqnarray}\label{m009020}
& &
\lVert \tilde B^\eps_{\textit{boundary}}\rVert_{L^2}\leq\lVert h\lVert_{L^2}\lVert C^\eps\rVert_{L^\infty},\\
& &| C^\eps(\al)|\leq\int_{|\be|=\eps}| (g(\al)-g(\al-\be))f(\al-\be)
(\del_l X(\al)-\del_l X(\al-\be))(\na K)(\theta(\al,\be))|
d\sigma(\be),\nn
\end{eqnarray}
and since $|(\del_l X(\al)-\del_l X(\al-\be))(\na K)(\theta(\al,\be))|\leq C/|\be|^2$ as we saw before, we then have the following bound.
\begin{eqnarray}\label{m009021}
|C^\eps(\al)|
&\leq&
\int_{|\be|=\eps}
|g(\al)-g(\al-\be)||f(\al-\be)|\frac{C}{|\be|^2}d\sigma(\be),\\
&\leq&
C'\sup_{|\be|=\eps}\lVert
g(\al)-g(\al-\be)
\rVert_{L^\infty_\al}
\lVert f\rVert_{L^\infty},\nn\\
&=:&c_3(\eps).\nn
\end{eqnarray}
Using Lipschitz continuity of $g$, we find that both $c_3(\eps)$ and thus $|C^\eps(\al)|$ are finite, and tend to zero as $\eps$ tends to zero, with the limit for $|C^\eps(\al)|$ uniform in $\al$. So $\lim_{\eps\to 0}C^\eps=0$ in $L^\infty$.

Moreover, in each case among $(h,g,f)\in\cZ_1$, $\cZ_2$, and $\cZ_3$, the corresponding bound \eqref{m009015} or \eqref{m009020} on $C^\eps(\al)$ implies that $\tilde B^\eps_{\textit{boundary}}(\al)$ is in $L^2$ and tends to zero as $\eps$ tends to zero.

Now recall from $\eqref{m009003}$ that we have
\ba
-\la\del_l\vp , T^\eps_{X,k}[h,g,f]\ra=A^\eps+B^\eps,
\ea
where
\ba
A^\eps=\sum^4_{m=1} \la \varphi ,A^\eps_m\ra,
\ea
with the $A^\eps_m$ given by \eqref{m009005}. In addition we have $B^\eps=B^\eps_{\textit{int}}+B^\eps_{\textit{boundary}}$, where
\ba
B^\eps_{\textit{int}}=-\sum^4_{m=1}\sum^3_{i,j=1}\la \vp ,B^\eps_{m;ij}\ra,
\ea
for the $B^\eps_{m;ij}$ given by \eqref{m009012}, and where 
\ba
B^\eps_{\textit{boundary}}=\la\vp,\tilde B^\eps_{\textit{boundary}} \ra,
\ea
so
\ba\label{m009022}
-\la\del_l\vp , T^\eps_{X,k}[h,g,f]\ra
=
\sum^4_{m=1} \la \varphi ,A^\eps_m\ra
-\sum^4_{m=1}\sum^3_{i,j=1}\la \vp ,B^\eps_{m;ij}\ra
+\la\vp,\tilde B^\eps_{\textit{boundary}}\ra.
\ea
Recall that we verified that each of the functions paired with $\varphi$ in the above $L^2$ pairings is indeed in $L^2$. We also verified the $L^2$ limit $\lim_{\eps\to 0}\tilde B^\eps_{\textit{boundary}}=0$. Now we must verify that the remaining quantities paired with $\vp$, i.e. the $A^\eps_m$ and $B^\eps_{m;ij}$, converge to elements in $L^2$ as $\eps$ tends to zero. This can be checked by considering each of the cases $(h,g,f)\in\cZ_1$, $(h,g,f)\in\cZ_2$, and $(h,g,f)\in\cZ_3$, and applying Definition \ref{def001} and Remark \ref{rem001}. The last observation to make is that one can take the quantities paired with $\vp$ in \eqref{m009022}, excluding the $\tilde B^\eps_{\textit{boundary}}$ term, and add them in the corresponding way to produce the quantity in the right hand side of \eqref{m009}, within the limit as $\eps$ tends to zero that is written there, in part (ii) of the statement of the proposition. Using these observations together, one concludes that the statement of part (ii) holds, finishing the proof.

\end{proof}

Soon, we will iterate the differentiation formula given by the above proposition to get our formula for any order of derivatives acting on $F(X)(\al)$. We proceed with a definition of a sequence of polynomials defined in terms of an $X\in B_\delta$ which will be used in our final formula.

\begin{definition}\label{def002}
Fix $X\in B_\delta$. For $l=1,\ldots,d$ we define $\Theta_l(\al,\be):=[\na X(\al) (\na X(\al-\be))^{-1}-I]_l$, where $[\cdot]_l$ denotes the $l$th row of the matrix in brackets. For $\sigma=(\sigma_1,\ldots,\sigma_n)$ for an $n=1,\ldots,s$, with each $\sigma_i = 1,2$, or $3$, we define the $\C^3$-valued functions $P^{\sigma}_m(\al,\be)$ on $\R^3\times\R^3$ for $m=1,\ldots,n$ recursively by 
\begin{eqnarray}\label{m010}\
P^{\sigma}_1(\al,\be)&=&\Theta_{\sigma_1}(\al,\be) \rho(\al-\be),\nn\\
P^{\sigma}_{m+1}(\al,\be)&=&\del_{\al_{\sigma_{m+1}}}P^{\sigma}_m(\al,\be)-\Theta_{\sigma_{m+1}}(\al,\be)\div_\be P^{\sigma}_m(\al,\be)+\Theta_{\sigma_{m+1}}(\al,\be)(\del_{\sigma_m}\ldots\del_{\sigma_1}\rho)(\al-\be),\nn\\
& &\hspace{2.2in}\mbox{ for }1\leq m\leq n-1.\nn
\end{eqnarray}
\end{definition}

We will also need to verify that the polynomials above have a form which is compatible with the kinds of objects appearing in the integrals in the statement of Proposition \ref{prop001}. This is the purpose of the following lemma and corollary.

\begin{lem}\label{m001lem}
Consider fixed $X\in B_\delta$, $\sigma=(\sigma_1,\ldots,\sigma_n)$, some $n=1,\ldots,s$, and $P^\sigma_m(\al,\be)$ as defined in Definition \ref{def002} for $1\leq m \leq n$. Then we can express
\begin{eqnarray}\label{m012}\\
P^\sigma_m(\al,\be)
&=&
\sum_{|\Gamma|=m-1}
\bigg[
\textit{coeff}\;
\bigg(\del^\eta \del_{i} X_{j}(\al)-\del^\eta \del_i X_j(\al-\be)\bigg)
\left(
\prod^{N_1}_{l=1}
\del^{\mu_l}\del_{p_l} X_{q_l}(\al)
\right)\nn\\
& &
\hspace{0.7in}
\times
\left(
\prod^{N_2}_{l=1}
\del^{\nu_l}_\al(\na X(\al-\be))^{-1}_{(r_l,s_l)}
\right)
(\del^\ga \rho)(\al-\be)e_k
\bigg].\nn
\end{eqnarray}
In the above sum the $\Gamma$ are arrays of all of the following quantities which the summation is over: we have multi-indices $\eta$, $\ga$, $\mu_l$ for $l=1,\ldots,N_1$, and $\nu_l$ for $l=1,\ldots,N_2$, we have indices $i$, $j$, $k$, each of which is one of $1$, $2$, or $3$, we have indices (for the corresponding $l$) $p_l$, $q_l$, $r_l$, and $s_l$, each of which is $1$, $2$, or $3$, and we have integers $N_1$ and $N_2$, where $N_1 \geq0$ and $N_2 \geq 1$.
%
We define
\ba
|\Gamma|=|\eta|+|\ga|+\sum^{N_1(\Gamma)}_{l=1}|\mu_l|+\sum^{N_2(\Gamma)}_{l=1}|\nu_l|.
\ea
In the summation in \eqref{m012}, by $\textit{coeff}$ we denote a coefficient which depends on $\Gamma$; its exact value is not needed.
\end{lem}
\begin{proof}
Rather than give the full proof, we comment that by using the recursive definition of the $P^\sigma_m(\al,\be)$ along with an argument by induction, one finds that the above formula holds.
\end{proof}
\begin{cor}\label{pcor01}
Fix $X\in B_\delta$ and $\sigma=(\sigma_1,\ldots,\sigma_n)$, some $n=1,\ldots, s$. For $P^\sigma_m(\al,\be)$, $m=1,\ldots,n$, as in Definition \ref{def002}, we have
\ba
P^\sigma_m(\al,\be)=\sum_{|\Gamma|=m-1}h_\Gamma(\al)(g_\Gamma(\al)-g_\Gamma(\al-\be))f_\Gamma(\al-\be)e_{k(\Gamma)},
\ea
where for each $\Gamma$ in the sum, $k(\Gamma)=1,2,$ or $3$, we have that $\mbox{supp}(f_\Gamma)$ bounded, and

(i)
if $m\leq n-1$, then $(h_\Gamma,g_\Gamma,f_\Gamma)\in\cZ_1$, $\cZ_2$, or $\cZ_3$, as defined in Proposition \ref{prop001}, and

(ii) if $m=n$, we have $(h_\Gamma,g_\Gamma,f_\Gamma)\in\cZ_1'$, $\cZ_2'$, or $\cZ_3'$, as defined in Proposition \ref{prop001}.
\end{cor}
\begin{proof}
For each $m=1,\ldots,n$ we use the expression for $P^\sigma_m(\al,\be)$ given by Lemma \ref{m001lem}, defining for each $\Gamma$ in the sum in \eqref{m012}
\begin{eqnarray}\label{m013}
h_\Gamma(\al)&:=&\textit{coeff}\prod^{N_1}_{l=1}\del^{\mu_l}\del_{p_l}X_{q_l}(\al),\\
g_\Gamma(\al)&:=&
\del^{\eta}\del_i X_j(\al),\nn\\
f_\Gamma(\al)&:=&
\left(
\prod^{N_2}_{l=1}
\del^{\nu_l}(\na X(\al))^{-1}_{(r_l,s_l)}
\right)
\del^\ga \rho(\al),\nn
\end{eqnarray}
where the quantities $\textit{coeff}$, $N_1$, $N_2$, and the various indices and multi-indices above are precisely the corresponding quantities appearing in the sum in \eqref{m012}, which depend on $\Ga$. Note that since $\rho$ is assumed to have compact support, we get $\mbox{supp}(f_\Gamma)$ is bounded for each $\Gamma$. Now we verify each triple $(h_\Gamma,g_\Gamma,f_\Gamma)$ is in at least one of $\cZ_1$, $\cZ_2$, or $\cZ_3$.

First, we consider the case that $m=1,\ldots,n-1$.

We fix a $\Gamma$ with $|\Gamma|=m-1$. If $m=1$ or $2$, the triple is in $\cZ_1$. If $m\geq 3$, but each of $|\eta|$, the $|\mu_l|$, $|\nu_l|$, and $|\ga|$, is $m-3$ or less, the triple is also in $\cZ_1$. 

Now suppose $m\geq 3$, and at least one of $|\eta|$, $|\ga|$, the $|\mu_l|$, or the $|\nu_l|$ is $m-2$ or greater. We cover this by checking the following cases.
\begin{itemize}

\item
If $|\eta|\geq m-2$, (note then $|\eta|\neq 0$) the triple is in $\cZ_2$.

\item
If $\max_l|\nu_l|\geq m-2$, (note then $\max_l|\nu_l|\neq 0$) the triple is in $\cZ_1$.

\item
If $|\ga|\geq m-2$, the triple is in $\cZ_1$.

\item
If $N_1\neq 0$, with $\max_l|\mu_l|\geq m-2$ (so that then $\max_l|\mu_l|\neq 0$), then the triple is in $\cZ_3$.
\end{itemize}
This takes care of every possibility for $m=1,\ldots, n-1$.
For the case $m=n$, one again considers each $\Gamma$ with $|\Gamma|=m-1$ and makes word for word the same set of observations as above, but with $\cZ_i'$ in place of $\cZ_i$ for $i=1,2$, and $3$, at each step.
\end{proof}

We record one more fact which will come in handy in the proof for our differentiation formula. It just asserts the pointwise and $L^2$ convergence of a key term which will occur.

\begin{lem}\label{plem01}
Fix a $\ga\in(\N_0)^3$ with $0\leq|\ga|\leq s$, where $\rho\in C^s_c(\R^3)$. For a fixed $X\in B_\delta$ let us define
\ba\label{deftheta}
\theta(\al,\be):=X(\al)-X(\al-\be).
\ea
Then the integral, defined pointwise for each $\al$ in $\R^3$,
\ba\label{p000a}
\int
K(\theta(\al,\be))(\del^\ga\rho)(\al-\be)d\be\quad\mbox{is in }L^2,
\ea
with $L^2$ norm bounded by a constant depending only on $\rho$, and
\ba\label{p000b}
\lim_{\eps\to0}
\int_{|\be|\geq \eps}
K(\theta(\al,\be))(\del^\ga\rho)(\al-\be)d\be\quad\mbox{converges in }L^2\mbox{ to }\eqref{p000a}.
\ea
\end{lem}
\begin{proof}
It is straightforward to verify that the integral in \eqref{p000a} is defined pointwise for each $\al$ in $\R^3$. Let us define for all $\al$ in $\R^3$
\ba\label{p00001}
g(\al):=\int
K(\theta(\al,\be))(\del^\ga\rho)(\al-\be)d\be.
\ea
First we claim that we have the following bound for all $\al$ in $\R^3$, uniform in $\eps>0$.
\ba\label{p00002}
\left|
\int_{|\be|\geq \eps}
K(\theta(\al,\be))(\del^\ga\rho)(\al-\be)d\be
\right|
\leq
\frac{C(\rho)}
{(1+|\al|)^2},
\ea
for a constant $C(\rho)$ depending on $\rho$. To show this, first we choose an $R>0$ with $\mbox{supp}(\rho)\subset B_R(0)$. Then we write
\begin{eqnarray}\label{p00003}
\int_{|\be|\geq \eps}
K(\theta(\al,\be))(\del^\ga\rho)(\al-\be)d\be
&=&
\int_{|\al-\be|\geq \eps}
K(X(\al)-X(\be))\del^\ga\rho(\be)d\be,\\
&=&
\int_{\eps\leq|\al-\be|\leq 100R}
K(X(\al)-X(\be))\del^\ga\rho(\be)d\be\nn\\
& &+
\int_{|\al-\be|>100R}
K(X(\al)-X(\be))\del^\ga\rho(\be)d\be,\nn\\
&=&
I_1+I_2.
\end{eqnarray}
For $I_1$ we have
\begin{eqnarray}\label{p00004}
|I_1|
&\leq&
\int_{\eps\leq|\al-\be|\leq 100R}
|K(X(\al)-X(\be))||\del^\ga \rho(\be)|d\be,\\
&\leq&
C\int_{\eps\leq|\al-\be|\leq 100R}
|\al-\be|^{-2}d\be
\left(\sup_{B_{100R}(\al)}|\del^\ga \rho|\right),\nn
\end{eqnarray}
where the right hand side is bounded by a constant depending on $\rho$ and has compact support in $\al$. This allows us to bound $|I_1|$ by something as in the right hand side of \eqref{p00002}.

For $I_2$, using that $\mbox{supp}(\rho)\subset B_R(\rho)$, we have
\ba\label{p00005}
|I_2|\leq
\int_{|\al-\be|>100R}
|K(X(\al)-X(\be))||\del^\ga\rho(\be)|d\be
\leq C'(\rho)\int_{S_\al}|\al-\be|^{-2}d\be.
\ea
where $C'(\rho)$ depends only on $\rho$, and $S_\al=\{\be:|\al-\be|>100R\mbox{ and }|\be|<R \}$. Suppose for a given $\al$ that $S_\al$ is nonempty. For any $\be\in S_\al$, we have $|\be|<R$ and $|\al-\be|>100R$. This requires $|\al|>99R$, so moreover, $\be\in S_\al$ implies $|\be|<|\al|/99$, and thus $|\al-\be|^{-1}\leq C|\al|^{-1}$. Thus from \eqref{p00005} we now find
\ba\label{p00006}
|I_2|\leq
\begin{cases}
C''(\rho)
\int_{|\be|<R}|\al|^{-2}d\be & \mbox{if }|\al|>99R,\\
0 & \mbox {if }|\al|\leq 99R,
\end{cases}
\ea
which is bounded above by an expression as in the right hand side of \eqref{p00002}. Combining this with the bound for $|I_1|$, we get the bound \eqref{p00002}.

Now let us define
\ba
g_\eps(\al):=
\int_{|\be|\geq \eps}
K(\theta(\al,\be))(\del^\ga\rho)(\al-\be)d\be
\ea
Since the right hand side of \eqref{p00002} is in $L^2$, we can bound $|g_\eps(\al)|^2$ by an integrable function of $\al$, independent of $\eps$. It then follows from the dominated convergence theorem that $g(\al)$ is in $L^2$, with the $L^2$ norm bounded by a constant depending only on $\rho$. We also find $|g_\eps(\al)-g(\al)|^2$ is bounded by an integrable function of $\al$, independent of $\eps$, and so with the dominated convergence theorem we get that $g_\eps$ tends to $g$ in $L^2$ as $\eps$ tends to zero. This concludes the proof of the lemma.
\end{proof}
Now we are ready to prove the differentiation formula, used for calculating various derivatives of the $F(X)(\al)$ defined in Definition \ref{Fd} in a manner well-suited for $L^2$ estimates. Once we have this, the local boundedness result for the operator $F$ for Sobolev spaces $H^s$ will readily follow.
\begin{prop}\label{diffformula}
For $X\in B_\delta$, the following formula holds for $F(X)(\al)$ as defined by Definition \ref{Fd}.
\begin{eqnarray}\label{08'}\\
& &\textrm{For }\ga\in(\N_0)^3\mbox{ with }0\leq |\ga|\leq s,\mbox{ and }\sigma=(\sigma_1,\ldots,\sigma_{|\ga|})\mbox{ with }\del_{\sigma_{|\ga|}}\ldots\del_{\sigma_1}=\del^\ga,\nn\\
& &\del^\ga[F(X)(\al)]=\lim_{\eps\to 0} \int_{|\be-\al|\geq\eps} P^\sigma_{|\ga|}(\al,\al-\be)\tilde K_X(\al,\be)d\be+\int K(X(\al)-X(\be))(\del^\ga\rho)(\be)d\be\quad\textrm{in }L^2.\nn
\end{eqnarray}
Here, for the case $\ga=0$ and $\sigma$ empty, we define $P^\sigma_0(\al,\be):=0$.
\end{prop}
\begin{proof}
Fix and $X\in B_\delta$. We prove that the formula in \eqref{08'} holds for $\ga\in (\N_0)^3$ by induction on $|\ga|$. Observe that it holds for $|\ga|=0$, with $\sigma$ empty, since we have the result of Lemma \ref{plem01}. Now we assume that \eqref{08'} is true for all $\ga$ with $|\ga|\leq m$, for some $0\leq m \leq s-1$. Consider a $\ga\in(\N_0)^3$ with $|\ga|=m+1$, and a $\sigma=(\sigma_1,\ldots,\sigma_{m+1})$ with $\del_{\sigma_{m+1}}\ldots\del_{\sigma_1}=\del^\ga$. Let us define $\sigma_\flat:=(\sigma_1,\ldots,\sigma_m)$, and define $\gamma_\flat$ by $\del^{\gamma_\flat}=\del_{\sigma_m}\ldots \del_{\sigma_1}$. Let us recall the definition of $\theta(\al,\be)$ from \eqref{deftheta}.
Using the formula \eqref{08'} and making a change of variables then shows
\begin{eqnarray}\label{p001}
\del^{\gamma_\flat}[F(X)(\al)]
&=&\lim_{\eps\to 0} \left(\int_{|\be-\al|\geq\eps} P^{\sigma_\flat}_m(\al,\be)\na_\be (K(\theta(\al,\be)))d\be\right)\\
& &+\int K(\theta(\al,\be))(\del^{\gamma_\flat}\rho)(\al-\be)d\be\quad\textrm{in }L^2.\nn
\end{eqnarray}
Note we have $P^{\sigma_\flat}_m(\al,\be)=P^{\sigma}_m(\al,\be)$, by the definition of these quantities. Let us make the corresponding replacement in the expression within the limit in \eqref{p001}. Next, for each $\eps>0$, we use part (i) of Corollary \ref{pcor01}, taking $n:=m+1$, to express the quantity within the limit as
\ba\label{p002}
\sum_{|\Gamma|=m-1}\left(
\int_{|\be|\geq\eps}
h_\Gamma(\al)(g_\Gamma(\al)-g_\Gamma(\al-\be))f_\Gamma(\al-\be)\del_{\be_{k(\Gamma)}}(K(\theta(\al,\be))d\be
\right),
\ea
where for each $\Gamma$, $\mbox{supp}(f_\Gamma)$ is compact, and each of the triples $(h_\Gamma,g_\Gamma,f_\Gamma)$ belongs to one of $\cZ_1$, $\cZ_2$, or $\cZ_3$. We use this in \eqref{p001} to get
\begin{eqnarray}\label{p003}\\
\del^{\gamma_\flat}[F(X)(\al)]
&=&\sum_{|\Gamma|=m-1}
\lim_{\eps\to 0} \left(\int_{|\be-\al|\geq\eps} 
h_\Gamma(\al)(g_\Gamma(\al)-g_\Gamma(\al-\be))f_\Gamma(\al-\be)\del_{\be_{k(\Gamma)}}(K(\theta(\al,\be))d\be
\right)\nn\\
& &+\int K(\theta(\al,\be))(\del^{\gamma_\flat}\rho)(\al-\be)d\be\quad\textrm{in }L^2.\nn
\end{eqnarray}
Here we have used (i) of Proposition \ref{prop001} to ensure convergence in $L^2$ of the summands in \eqref{p003}. Note by part (ii) of Proposition \ref{prop001}, the expression in the first line in the right hand side of \eqref{p003} is in $H^1$. Now we claim that
\ba\label{p004}
\int K(\theta(\al,\be))(\del^{\gamma_\flat}\rho)(\al-\be)d\be.
\ea
is also in $H^1$. To check this, let us test it against an arbitrary $\varphi\in C^\infty_c$. Fix an $l=1,2,$ or $3$. Then
\begin{eqnarray}\label{p005}
& &-\int
\del_l\vp(\al)
\int
K(\theta(\al,\be))(\del^{\gamma_\flat}\rho)(\al-\be)d\be
d\al\\
& &\hspace{.2in}=
-\lim_{\eps\to 0}
\int
\del_l\vp(\al)
\int_{|\be|\geq \eps}
K(\theta(\al,\be))(\del^{\gamma_\flat}\rho)(\al-\be)d\be
d\al,\nn
\end{eqnarray}
by Lemma \ref{plem01}. Consider the quantity in the right hand side of \eqref{p005} within the limit.
\begin{eqnarray}\label{p006'}\\
& &
\int
\del_l\vp(\al)
\int_{|\be|\geq \eps}
K(\theta(\al,\be))(\del^{\gamma_\flat}\rho)(\al-\be)d\be
d\al\nn\\
& &\hspace{.7in}=
\int_{|\be|\geq \eps}
\int
\del_l\vp(\al)
K(\theta(\al,\be))(\del^{\gamma_\flat}\rho)(\al-\be)d\al
d\be,\nn\\
& &\hspace{.7in}=
-\sum^3_{i=1}
\int_{|\be|\geq \eps}
\int
\vp(\al)
(\del_l X_i(\al)-\del_l X_i(\al-\be))(\del_i K)(\theta(\al,\be))(\del^{\gamma_\flat}\rho)(\al-\be)d\al
d\be\nn\\
& &\hspace{.85in}-
\int_{|\be|\geq \eps}\int
\vp(\al) K(\theta(\al,\be))(\del_l\del^{\ga_\flat}\rho)(\al-\be)d\al d\be,\nn\\
& &\hspace{.7in}=
-\sum^3_{i=1}
\int\vp(\al)
\bigg(
\del_l X_i(\al)\int_{|\be|\geq\eps}
(\del_i K)(\theta(\al,\be))(\del^{\gamma_\flat}\rho)(\al-\be)d\be
\bigg)d\al\nn\\
& &\hspace{.85in}+
\sum^3_{i=1}
\int\vp(\al)
\bigg(
\int_{|\be|\geq\eps}
\del_l X_i(\al-\be)(\del_i K)(\theta(\al,\be))(\del^{\gamma_\flat}\rho)(\al-\be)d\be
\bigg)d\alpha\nn\\
& &\hspace{.85in}
-\int\vp(\al)\bigg(
\int_{|\be|\geq\eps}K(\theta(\al,\be))(\del_l\del^{\gamma_\flat}\rho)(\al-\be)d\be
\bigg)d\al,\nn\\
& &\hspace{.7in}
=I^1_\eps+I^2_\eps+I^3_\eps\label{p006b}
\end{eqnarray}
It is not difficult to check that for $i=1,2,$ and $3$,
\ba\label{p007'}
(\del_i K)(\theta(\al,\be))=\sum^3_{j=1}(\na X(\al-\be))^{-1}_{ij} \tilde K^j_X(\al,\al-\be).
\ea
So for $I^1_\eps$ we have
\ba\label{p008'}
I^1_\eps=-\sum^3_{i,j=1}
\int\vp(\al)\del_l X_i(\al)
\int_{|\be|\geq\eps}
(\na X(\al-\be))^{-1}_{ij} \tilde K^j_X(\al,\al-\be)(\del^{\gamma_\flat}\rho)(\al-\be)d\be.
\ea
Meanwhile, regarding $I^2_\eps$,
\ba\label{p009}
I^2_\eps
=
\int\vp(\al)
\int_{|\be|\geq\eps}\tilde K^l_X(\al,\al-\be)(\del^{\gamma_\flat}\rho)(\al-\be)d\be,
\ea
Observe that \eqref{p008'} and \eqref{p009} can be expressed in the form $I^1_\eps=\la\vp,v_1(\eps)\ra$ and $I^2_\eps=\la\vp,v_2(\eps)\ra$, with $v_1(\eps)$ and $v_2(\eps)$ both converging in $L^2$ as $\eps$ tends to zero, in view of Definition \ref{def001} and Remark \ref{rem001}, since $\del^{\gamma_\flat}\rho$ is in $L^2$ and $(\na X)^{-1}$ is bounded.

Observe also that $I^3_\eps$ in \eqref{p006b} has the form $I^3_\eps=\la \vp,v_3(\eps)\ra$ where $v_3(\eps)$ converges in $L^2$ as $\eps$ tends to zero, by Lemma \ref{plem01}. From \eqref{p005} and \eqref{p006'} we then find
\ba\label{p010'}
-\int
\del_l\vp(\al)
\int
K(\theta(\al,\be))\del^{\gamma_\flat}\rho(\al-\be)d\be
d\al
=
-\la
\vp,
\lim_{\eps\to 0}[v_1(\eps)+v_2(\eps)+v_3(\eps)]
\ra
\ea
Let us define $h:=\lim_{\eps\to 0}[v_1(\eps)+v_2(\eps)+v_3(\eps)]$. Then it follows that $h$ is in $L^2$ and the weak $\del_l$-derivative of the expression in \eqref{p004}. This allows us to conclude that indeed, the expression in \eqref{p004} is in $H^1$. In addition, we can conclude from the above discussion, by summing $v_1(\eps)$, $v_2(\eps)$, and $v_3(\eps)$, that the $\del_l$-derivative of \eqref{p004} for $l=1,2,3$ is given by
\begin{eqnarray}\label{p011'}
& &\lim_{\eps\to 0}
\bigg(
\int_{|\be|\geq\eps}
(\del_l X(\al)-\del_l X(\al-\be))(\na X(\al-\be))^{-1}
\tilde K_X(\al,\al-\be)(\del^{\ga_\flat}\rho)(\al-\be)d\be\\
& &\hspace{1in}
+\int_{|\be|\geq\eps}
K(\theta(\al,\be))(\del_l\del^{\ga_\flat}\rho)(\al-\be)d\be\bigg)\quad\mbox{in }L^2.\nn
\end{eqnarray}
Now we return to \eqref{p003}, and apply $\del_{\sigma_{m+1}}$ to both sides. For this we use our calculation of the derivative \eqref{p011'} of \eqref{p004}, taking $l=\sigma_{m+1}$, along with (ii) of Proposition \ref{prop001}. This gives
\begin{eqnarray}\label{p012'}
& &\del^\ga[F(X)(\al)]=\\
& &\hspace{.3in}
\lim_{\eps\to 0}
\Bigg[
\sum_{|\Gamma|=m-1}\bigg(
\int_{|\be|\geq\eps}
\del_{\al_{\sigma_{m+1}}}\big[h_\Gamma(\al)(g_\Gamma(\al)-g_\Gamma(\al-\be))f_\Gamma(\al-\be)\big]\del_{\be_{k(\Gamma)}}(K(\theta(\al,\be)))
d\be\nn\\
& &
\hspace{.7in}-
\sum^3_{i,j=1}
\int_{|\be|\geq\eps}
\del_{\be_{k(\Gamma)}}
\big[h_\Gamma(\al)(g_\Gamma(\al)-g_\Gamma(\al-\be))f_\Gamma(\al-\be)\big]\nn\\
& &\hspace{1.2in}
\times
(\del_{\sigma_{m+1}} X_i(\al)-\del_{\sigma_{m+1}} X_i(\al-\be))
 (\na X(\al-\be))^{-1}_{ij}\del_{\be_j}(K(\theta(\al,\be)))
d\be
\bigg)\nn\\
& &\hspace{.7in}
+\int_{|\be|\geq\eps}
(\del_{\sigma_{m+1}} X(\al)-\del_{\sigma_{m+1}} X(\al-\be))(\na X(\al-\be))^{-1}
\tilde K_X(\al,\al-\be)(\del^{\ga_\flat}\rho)(\al-\be)d\be\nn\\
& &\hspace{1in}
+\int_{|\be|\geq\eps}
K(\theta(\al,\be))(\del^\ga\rho)(\al-\be)d\be\Bigg]\quad\mbox{in }L^2.\nn
\end{eqnarray}
The right hand side can be rearranged to give
\begin{eqnarray}\label{p013'}
& &\del^\ga[F(X)(\al)]=\\
& &\hspace{.1in}
\lim_{\eps\to 0}
\Bigg[
\int_{|\be|\geq\eps}
\del_{\al_{\sigma_{m+1}}}
\left(
\sum_{|\Gamma|=m-1}
h_\Gamma(\al)(g_\Gamma(\al)-g_\Gamma(\al-\be))f_\Gamma(\al-\be)e_{k(\Gamma)}
\right)
\na_\be(K(\theta(\al,\be)))
d\be\nn\\
& &
\hspace{.5in}-
\int_{|\be|\geq\eps}
\div_\be
\left(
\sum_{|\Gamma|=m-1}
h_\Gamma(\al)(g_\Gamma(\al)-g_\Gamma(\al-\be))f_\Gamma(\al-\be)e_{k(\Gamma)}
\right)\nn\\
& &\nn\\
& &\hspace{1.1in}\times 
(\del_{\sigma_{m+1}} X(\al)-\del_{\sigma_{m+1}} X(\al-\be))(\na X(\al-\be))^{-1}\na_\be(K(\theta(\al,\be)))
d\be\nn\\
& &\nn\\
& &\hspace{.4in}
+\int_{|\be|\geq\eps}
(\del_{\sigma_{m+1}} X(\al)-\del_{\sigma_{m+1}} X(\al-\be))(\na X(\al-\be))^{-1}
\tilde K_X(\al,\al-\be)(\del^{\ga_\flat}\rho)(\al-\be)d\be\nn\\
& &\hspace{.4in}
+\int_{|\be|\geq\eps}
K(\theta(\al,\be))(\del^\ga\rho)(\al-\be)d\be\Bigg]\quad\mbox{in }L^2.\nn
\end{eqnarray}
Recalling the $\Theta_n(\al,\be)$ notation given in Definition \ref{def002}, we then find
\begin{eqnarray}
\del^\ga[F(X)(\al)]
&=&
\lim_{\eps\to 0}
\bigg[
\int_{|\be|\geq\eps}
\del_{\al_{\sigma_{m+1}}}
P^\sigma_m(\al,\be)
\na_\be(K(\theta(\al,\be)))
d\be\\
& &
\hspace{.4in}-
\int_{|\be|\geq\eps}
\div_\be
P^\sigma_m(\al,\be)
\Theta_{\sigma_{m+1}}(\al,\be)
\na_\be(K(\theta(\al,\be)))
d\be\nn\\
& &\nn\\
& &\hspace{.4in}
+\int_{|\be|\geq\eps}
\Theta_{\sigma_{m+1}}(\al,\be)
\na_\be(K(\theta(\al,\be)))
(\del^{\ga_\flat}\rho)(\al-\be)d\be\nn\\
& &\hspace{.4in}
+\int_{|\be|\geq\eps}
K(\theta(\al,\be))(\del^\ga\rho)(\al-\be)d\be\bigg],\nn\\
&=&
\lim_{\eps\to 0}\left[
\int_{|\be|\geq\eps}
P^\sigma_{m+1}(\al,\be)
\na_\be (K(\theta(\al,\be)))d\be
+\int_{|\be|\geq\eps}
K(\theta(\al,\be))(\del^\ga\rho)(\al-\be)d\be
\right],\nn\\
&=&
\lim_{\eps\to 0}
\int_{|\be|\geq\eps}
P^\sigma_{|\ga|}(\al,\al-\be)
\tilde K_X(\al,\be)d\be
+\int
K(X(\al)-X(\be))(\del^\ga\rho)(\be)d\be
\quad\mbox{in }L^2,\nn
\end{eqnarray}
where we have applied Lemma \ref{plem01} in the last line. This concludes the inductive step.
\end{proof}
\begin{prop}\label{pcor02}
Let $X\in B_\delta$. Let the operator $F$ be as in Definition \ref{Fd}. Then
\ba\label{p014'}
\lVert
F(X)
\rVert_{H^s}
\leq C,
\ea
where $C$ only depends on $\rho$.
\end{prop}
\begin{proof}
We consider some $\ga\in(\N_0)^3$ with $0\leq|\ga|\leq s$, define $n:=|\ga|$, and choose $\sigma=(\sigma_1,\ldots,\sigma_n)$ with $\del^\ga=\del_{\sigma_n}\ldots\del_{\sigma_1}$. By Proposition \ref{diffformula},
\begin{eqnarray}\label{p015}\\
\del^\ga[F(X)(\al)]=\lim_{\eps\to 0} \int_{|\be-\al|\geq\eps} P^\sigma_n(\al,\al-\be)\tilde K_X (\al,\be)d\be+\int K(X(\al)-X(\be))(\del^\ga\rho)(\be)d\be\quad\textrm{in }L^2.\nn
\end{eqnarray}
If $n=0$ we note the first term in the right hand side of \eqref{p015} is zero, and apply Lemma \ref{plem01} to get that the remaining term is bounded in $L^2$ by a constant depending on $\rho$. Suppose $n\geq 1$. For each $\eps>0$, we consider the quantity within the limit in \eqref{p015}. We make a change of variables and apply Corollary \ref{pcor01}, producing for $P^\sigma_{n}(\al,\be)$ a collection of triples $(h_\Gamma, g_\Gamma, f_\Gamma)$, with each triple in $\cZ_1'$, $\cZ_2'$, or $\cZ_3'$, and we find
\begin{eqnarray}\label{p016'}\\
\int_{|\be-\al|\geq\eps} P^\sigma_{n}(\al,\al-\be)\tilde K_X(\al,\be)d\be
&=&
\int_{|\be-\al|\geq\eps} P^\sigma_{n}(\al,\be)\na_\be(K(\theta(\al,\be)))d\be,\nn\\
&=&
\sum_{|\Gamma|=n-1}h_\Gamma(\al)\int_{|\be-\al|\geq\eps}(g_\Gamma(\al)-g_\Gamma(\al-\be))\nn\\
& &\hspace{1.5in}\times f_\Gamma(\al-\be)\del_{\be_{k(\Gamma)}}(K(\theta(\al,\be)))d\be.\nn
\end{eqnarray}
Now we apply the result of part (i) for Proposition \ref{prop001}, 
which yields convergence in $L^2$ as $\eps$ tends to zero.
\ba\label{p017'}
\lim_{\eps \to 0}\int_{|\be-\al|\geq\eps} P^\sigma_{s}(\al,\al-\be)\tilde K_X(\al,\be)d\be
=
\sum_{|\Gamma|=s-1}
(T^0_{X,k(\Gamma)}[h_\Gamma,g_\Gamma,f_\Gamma])(\al)\quad\mbox{in }L^2.
\ea
Now we observe by examining the forms \eqref{m013} for $h_\Gamma$, $g_\Gamma$, and $f_\Gamma$ given in the proof of Corollary \ref{pcor01}, and using the fact that $X\in B_\delta$, that we must have for each of the $\Gamma$
\ba\label{ep02}
\lVert
(h_\Gamma,g_\Gamma,f_\Gamma)
\rVert_{\cZ_i'}
\leq
C(\rho).
\ea
Taking the $L^2$ norm of both sides of \eqref{p017'}, and using part (i) of Proposition \ref{prop001} along with \eqref{ep02}, we find
\ba\label{p018'}
\left\lVert\lim_{\eps \to 0}\int_{|\be-\al|\geq\eps} P^\sigma_{s}(\al,\al-\be)\tilde K_X(\al,\be)d\be
\right\rVert_{L^2}
\leq
C_1(\rho).
\ea
Meanwhile, from Lemma \ref{plem01} we get that
\ba\label{p019'}
\left\lVert
\int K(X(\al)-X(\be))(\del^\ga\rho)(\be)d\be
\right\rVert_{L^2}
\leq
C_2(\rho).
\ea
Using \eqref{p018'} and \eqref{p019'} to bound the $L^2$ norm of the right hand side of \eqref{p015}, we conclude
\ba
\lVert
\del^\ga F(X)
\rVert_{L^2}\leq C_3(\rho).
\ea
Since we have such a bound for all $\ga$ with $0\leq |\ga|\leq s$, \eqref{p014'} follows.
\end{proof}
Now that we have the desired local boundedness of the operator $F$ in this context, we must verify the other criterion of preservation of analyticity, as discussed in Section 2.

\begin{subsection}{Analyticity of trajectories}
Our proof of Lagrangian analyticity for the Euler-Poisson system largely follows the strategy outlined in Section 2 for proving Lagrangian analyticity for fluid mechanics systems. We also need to make some slight modifications to the techniques of that section when we prove the existence, uniqueness, and analyticity of solutions to the ODE system \eqref{NewEp01}, to accommodate for the fact that we have initial data of the type $(X,V)|_{t=0}=(Id,u_0)$, rather than simply the identity map, but the changes are not especially substantial. We continue by providing the analogous abstract function space setting for the Euler-Poisson problem.
\end{subsection}
\begin{definition}\label{banach}
Taking the Banach space $Y=H^s(\R^3)$, we recall the definition
\ba
B_\delta:=\{X:\lVert X-Id\rVert_Y\leq \delta\},
\ea
and for $r>0$, we define
\ba\label{24'}
\cY_r:=\{Y\textrm{-valued maps }X_z\textrm{ analytic in }z\textrm{ on }d_r\},
\ea
where $d_r=\{|z|<r\}$, and
we denote
\ba
\cB_{\delta,r}:=\{\mbox{maps }X_{(\cdot)}:(X_{(\cdot)}-Id)\in \cY_r, \;X_z\in B_\delta \textrm{ for all }z\in d_r\}.
\ea
\end{definition}
\begin{lem}\label{vPreservesA}
For $\rho\in C^s_c(\R^3)$, the operator $F$ preserves analyticity. That is, for any $r>0$ and $X_{(\cdot)}\in\cB_{\delta,r}$, for fixed $\al\in\R^3$, $F(X_z)(\al)$ is analytic in $d_r$.
\end{lem}

The following lemma follows immediately from an application of Lemma \ref{omegalem'}, since $Y$ satisfies Assumption \ref{as1}.
\begin{lem}\label{omegalem}
Consider $B_\delta$, $\cB_{\delta,r}$, $Y$, and $\cY_r$ as above. Consider a function $\tilde F:B_\delta\to Y$. If we have a bound of the form
\ba
\sup_{X\in B_\delta} \lVert \tilde F(X) \rVert_Y\leq C,
\ea
and if $\tilde F$ preserves analyticity, then in fact for any $r>0$, for $X_z\in\cB_{\delta,r}$ we have that $\tilde F(X_z)$ is an analytic function from $d_r$ into $Y$. That is,
\ba
\tilde F:\cB_{\delta,r}\to \cY_r.
\ea
\end{lem}

Now we proceed to use what has been established above to solve the following equation for $Z_t=(X_t,V_t):\R^3\times[-T,T]\to\R^6$.
\begin{eqnarray}\label{36}
&\frac{d}{dt} Z_t& = 
\left(
	\begin{array}{ccc}
	V_t \\
	F(X_t) 
	\end{array}
	\right),\\
&Z_t |_{t=0}&=
\left(
\begin{array}{ccc}
	Id \\
	u_0 
	\end{array}
	\right),\nonumber
\end{eqnarray}
in the desired setting, where $u_0$ is in $H^s$ and plays the role of our initial velocity for the Euler-Poisson system. We prove that the solution to the equation exists as a fixed point of a contraction mapping. For $\delta>0$ and $r>0$, we define
\ba
\cD_{\delta,r}:=\{Z_{(\cdot)}=(X_{(\cdot)},V_{(\cdot)}):
\; X_{(\cdot)}\in\cB_{\delta,r},\;V_{(\cdot)}\in \cY_r,\mbox{ and }\lVert V_z-u_0\rVert_Y\leq\delta \textrm{ for all }z\in d_r\},
\ea
and
\ba
\lVert (X_{(\cdot)},V_{(\cdot)})
\rVert_{\cD_{\delta,r}}
:=
\sup_{z\in d_r}
\left(
\lVert
X_z-Id
\rVert_Y+\lVert
V_z-u_0
\rVert_Y
\right).
\ea
We recall Lemma \ref{what'}, which directly implies the following lemma, which in turn applies to our operator $F$, in light of Proposition \ref{pcor02} and Lemma \ref{vPreservesA}:
\begin{lem}\label{what}
Consider $\delta>0$ and a function $\tilde F:B_\delta\to Y$, where $Y$ is a Banach space of maps from $\R^3$ to $\R^3$ with norm dominating the $C^1$ norm, and $B_\delta$ is the ball about the identity map of radius $\delta$ in $Y$. If we have a bound of the form
\ba\label{65prime}
\sup_{X\in B_\delta} \lVert \tilde F(X) \rVert_Y\leq C,
\ea
and that $\tilde F$ preserves analyticity, then for any $\eps\leq \delta/10$ we have that the function $\tilde F$ with domain $B_{\eps}$ is Lipschitz with constant $C_0$ depending only on $C$.
\end{lem}
Now we choose a $\delta$ sufficiently small to satisfy the hypotheses required for all the above results, and fix $\eps=\delta/10$. With that, for any $r>0$ we define the following mapping on elements $Z_{(\cdot)}\in \cD_{\eps,r}$:
\ba
\Phi \left(Z_{(\cdot)}\right)(t):=
\left(
\begin{array}{ccc}
Id\\
u_0
\end{array}
\right)
+\int^t_0 
\left(\begin{array}{ccc}
V_\zeta\\
 F(X_\zeta)
\end{array}\right)
d\zeta\quad\quad\textrm{for }t\in d_r,
\ea
where the integration is taken along the straight path from $0$ to $t$. This is the map we will show has a fixed point, which solves the ODE \eqref{36}.
\begin{prop}\label{epmainth}
For $\eps=\delta/10$ and sufficiently small $r_1>0$, there exists a unique fixed point $Z_{(\cdot)}\in\cD_{\eps,{r_1}}$ of $\Phi:\cD_{\eps,{r_1}}\to\cD_{\eps,{r_1}}$, and the function $Z(\al,t)=Z_t(\al)$ is analytic in $t$ on $d_{r_1}$ and the unique solution to the equation \eqref{36}.
\end{prop}
\begin{proof}
First we check that $\Phi$ maps $\cD_{\eps,{r_1}}$ to itself. Note that $t$ is restricted to $d_{r_1}$ and $r_1$ can be chosen as small as we like, while the bound
\ba\label{65'}
\sup_{X\in B_\delta} \lVert  F(X) \rVert_Y\leq C,
\ea
in particular allows us to ensure that $\Phi (Z_{(\cdot)})$ is close to $(Id,u_0)^t$ for small $r_1$. The verification that $\Phi$ maps $\cD_{\eps,{r_1}}$ to itself is completed with the aid of Lemmas \ref{vPreservesA} and \ref{omegalem}. 
Using Lemma \ref{what}, we have the Lipschitz property of $F$ on $B_{\eps}$; using this we can show that $\Phi$ is a contraction mapping on $\cD_{\eps,r_1}$ as long as $r_1$ is small enough. The conclusion of the theorem now follows from the contraction mapping principle.
\end{proof}

Thus, we conclude in particular from Proposition \ref{epmainth} the local existence and uniqueness of solutions to the system \eqref{epfinal01a}-\eqref{epfinal01b}, and that the corresponding Lagrangian trajectories $X(\al,t)$ are analytic in the time variable, as long as the initial data $(\rho_0,u_0)$ is in $C^s_c\times H^s$ for $s\geq 6$.
\begin{theo}\label{maintheoEP}
Consider a solution $(\rho,u)$ to \eqref{epfinal01a}, with initial data $\rho_0\in C^s_c$, $u_0\in H^s$ for $s\geq 6$. The corresponding trajectory map $X(\al,t)$ solving \eqref{epfinal02b} is analytic in $t$ at $t=0$ for each $\al\in\R^3$.
\end{theo}

\end{subsection}

\begin{subsection}{Integral kernel operator bounds}
In this part of the section, we direct our efforts to delineating conditions under which an integral kernel of the form $\cK(\al,\be)$, $\mathcal{K}:(\R^3\times\R^3\setminus\{\al=\be\})\to \C$, with
\ba\label{cond1}
|\mathcal{K}(\al,\be)|\leq A |\al-\be|^{-3},
\ea
has a corresponding operator that is bounded on $L^2$ in particular.
\begin{lem}\label{stein}
Suppose $\mathcal{K}:(\R^3\times\R^3\setminus\{\al=\be\})\to \C$ satisfies \eqref{cond1} and
\begin{eqnarray}\label{cond2}
|\mathcal{K}(\al,\be)-\cK(\al',\be)|&\leq& A\frac{|\al-\al'|}{|\al-\be|^{4}},\quad\textrm{if }|\al-\al'|\leq |\al-\be|/2,\\
|\cK(\al,\be)-\cK(\al,\be')|&\leq& A\frac{|\be-\be'|}{|\al-\be|^{4}},\quad\textrm{if }|\be-\be'|\leq |\al-\be|/2,\nonumber
\end{eqnarray}
and, defining the quantities
\begin{eqnarray}\label{cancelI}
I_{\eps,R}(\al)=\int_{\eps<|\al-\be|<R}\cK(\al,\be)d\be,\quad\quad I^*_{\eps,R}(\al)=\int_{\eps<|\al-\be|<R}\overline{\cK}(\be,\al)d\be,\nonumber
\end{eqnarray}
suppose we have in addition
\ba\label{cond3}
|I_{\eps,R}(\al)|\leq A \quad \textrm{for all }\eps,\;R>0\textrm{ and all }\al\in\R^3,
\ea
and an analogous uniform bound on $I^*_{\eps,R}(\al)$ holds.
Then the operators $T^\eps_\cK$ defined by
\ba
(T^\eps_\cK f)(\al)=\int_{|\be-\al|\geq \eps}\cK(\al,\be)f(\be)d\be,
\ea
are bounded from $L^2$ to $L^2$ and satisfy
\ba
\lVert T^\eps_\cK f\rVert_{L^2}\leq C \lVert f \rVert_{L^2},
\ea
independent of $\eps$.
\end{lem}
\begin{proof}The above result follows from an application of Theorem 4 in Section 7.3 of \cite{stein} together with Proposition 1 in Section 1.7 of \cite{stein}.
\end{proof}

The next lemma gives some basic conditions for $\cK$ that lead to pointwise convergence of the operators $T^\eps_\cK f$ as $\eps$ tends to zero when $f$ has a little smoothness, along with related operator bounds.
\begin{lem}\label{kernellem0'}
We take a particular choice of $\delta_0>0$ and consider a positive $\delta<\delta_0$. Fix $X\in B_\delta$. Fix $i$ and $j$ each to be one of $1,2$, or $3$. For $\al,\be\in\R^3$, $\al\neq\be$, we define
\ba\label{41p5}
K^\flat(\al,\be):=K^j(X(\al)-X(\be)),
\ea
and
\ba\label{41}
\cK(\al,\be):=\del_{\be_i} K^\flat(\al,\be).
\ea
Then we have that
\ba\label{42}
|K^\flat(\al,\be)| \leq A|\al-\be|^{-2},
\ea
and $\cK(\al,\be)$ satisfies \eqref{cond1}.

Additionally, for $f\in L^2\cap C^1$, the following limit converges uniformly in $\al$.
\ba
(T_\cK f)(\al):=p.v.\int \cK(\al,\be)f(\be)d\be,
\ea
and for $\eps,R>0$, the operators defined by
\ba
(T^{\eps,R}_\cK f)(\al):=\int_{\eps\leq|\be-\al|\leq R} \cK(\al,\be)f(\be)d\be,
\ea
and
\ba
(T^\eps_\cK f)(\al):=\int_{|\be-\al|\geq \eps} \cK(\al,\be)f(\be)d\be,
\ea
we have the bounds for all $f\in L^2\cap C^1$
\begin{eqnarray}
\lVert T^{\eps,R}_\cK f \rVert_{L^\infty}&\leq& A'( \lVert f \rVert_{L^2}+\lVert f \rVert_{C^1}),\label{432b}\\
\lVert T^\eps_\cK f \rVert_{L^\infty}&\leq& A'( \lVert f \rVert_{L^2}+\lVert f \rVert_{C^1}),\label{432}\\
\lVert T_\cK f \rVert_{L^\infty}&\leq& A'( \lVert f \rVert_{L^2}+\lVert f \rVert_{C^1}),\label{43}
\end{eqnarray}
for a universal constant $A'$. In particular $A'$ does not depend on the choice of $i,j,$ $\delta$, or $X$.
\end{lem}
\begin{proof}
Fix $\eps,R>0$, and an $f\in L^2\cap C^1$. Note for $R\leq \eps$, \eqref{432b} is trivial, so let us assume $R>\eps$. Let $r=\min(1,R)$. We split the integral for $(T^{\eps,R}_\cK f)(\al)$ into two, with
\ba\label{44}
\int_{\eps\leq|\be-\al|\leq R} \cK(\al,\be)f(\be) d\be = 
\int_{\eps\leq|\be-\al| < r}\cK(\al,\be)f(\be)d\be +\int_{1\leq |\be-\al| \leq R}\cK(\al,\be)f(\be)d\be.
\ea
The integral on the right is bounded by a constant multiple of $\lVert f \rVert_{L^2}$ since $\lVert 1_{1\leq |\be-\al|\leq R}\cK(\al,\be)\rVert_{L^2_\be}$ is bounded by a constant, by \eqref{cond1}. For the integral on the left, we have
\ba\label{45}
\int_{\eps\leq |\be-\al| < r}\cK(\al,\be)f(\be)d\be = \int_{\eps\leq |\be-\al| < r}\cK(\al,\be)(f(\be)-f(\al))d\be +f(\al)\int_{\eps\leq |\be-\al| < r}\cK(\al,\be)d\be,
\ea
where, for the first piece we have
\ba\label{46}
\left|\int_{\eps\leq|\be-\al| < r}\cK(\al,\be)(f(\be)-f(\al))d\be\right|\leq \int_{|\be-\al|<1}A|\al-\be|^{-3}\lVert f \rVert_{C^{1}}|\al-\be|d\be\leq C \lVert f \rVert_{C^1},
\ea
and the magnitude of the second piece is equal to
\begin{eqnarray}\label{47}
\left|f(\al)\int_{\eps\leq |\be-\al| < r}\del_{\be_i}K^\flat(\al,\be)d\be\right| &=& \bigg|f(\al)\bigg(-\int_{|\al-\be|=\eps}K^\flat(\al,\be)\nu_i(\be)d\sigma(\be)\\
& &\hspace{.5in}
+\int_{|\al-\be|=r}K^\flat(\al,\be)\nu_i(\be)d\sigma(\be)
\bigg)\bigg|,\nonumber\\
&\leq & |f(\al)|\bigg( \int_{|\al-\be|=\eps}A|\al-\be|^{-2}d\sigma(\be)\nn\\
& &\hspace{.5in}
+\int_{|\al-\be|=r}A|\al-\be|^{-2}d\sigma(\be)\bigg),\nn\\
&\leq & C\lVert f \rVert_{L^\infty},\nn
\end{eqnarray}
where $\nu$ denotes the unit normal. This concludes the proof of the estimate for $T^{\eps,R}_{\cK}$. By taking the limit as $R$ tends to infinity, we also get the estimate for $T^\eps_\cK$. Now, regarding the definition of $T_{\cK}f$ for $f\in C^1\cap L^2$, we use the above decomposition to verify that $T^\eps_\cK f(\al)$ converges for each $\al\in \R^3$ and in fact in the $\sup$ norm, as $\eps$ tends to zero. It suffices to prove uniform convergence for
\ba
\int_{\eps\leq |\be-\al|\leq 1} \cK(\al,\be)f(\be) d\be 
\ea
which we split as in \eqref{45}. For the piece
\ba
\int_{\eps\leq |\be-\al| < 1}\cK(\al,\be)(f(\be)-f(\al))d\be 
\ea
we note from the previous observations that the singularity in $\be$ at $\beta=\al$ in the integrand is actually integrable, and that
\ba
\int_{ |\be| < \eps}\cK(\al,\al-\be)(f(\al-\be)-f(\al))d\be 
\ea
tends to zero uniformly in $\al$ as $\eps$ tends to zero.

For the piece
\ba
f(\al)\int_{\eps\leq |\be-\al| < 1}\cK(\al,\be)d\be
\ea
we observe that
\ba
\int_{\eps\leq |\be-\al| < 1}\cK(\al,\be)d\be
=
-\int_{|\al-\be|=\eps}K^\flat(\al,\be)\nu_i(\be)d\sigma(\be)+
\int_{|\al-\be|=1}K^\flat(\al,\be)\nu_i(\be)d\sigma(\be),
\ea
and so for convergence as $\eps$ tends to zero we just have to deal with the quantity
\begin{eqnarray}\label{quantity}
\int_{|\al-\be|=\eps}K^\flat(\al,\be)\nu_i(\be)d\sigma(\be)
&=&
\int_{|\be|=\eps}K^\flat(\al,\al-\be)\nu_i(\be)d\sigma(\be),\\
&=&
\int_{|\be|=\eps}K^j(X(\al)-X(\al-\be))\nu_i(\be) d\sigma(\be).\nn
\end{eqnarray}
Now note that for any $\be$ with $|\be|=\eps$ we have $X(\al)-X(\al-\be)=\be(I+O(\delta))$ and $\be\na X(\al)=\be(I+O(\delta))$. Thus as long as $\delta$ is small enough, the straight line segment connecting $(X(\al)-X(\al-\be))$ and $\be\na X(\al)$ does not intersect zero for any such $\be$, for any $\eps>0$. We may use this to apply the fundamental theorem of calculus in the following observation.
\begin{eqnarray}
& &\int_{|\be|=\eps}K^j(X(\al)-X(\al-\be))\nu_i(\be) d\sigma(\be)
-
\int_{|\be|=\eps}K^j(\be \na X(\al))\nu_i(\be) d\sigma(\be)\\
& &=\int_{|\be|=\eps}
\big((X(\al)-X(\al-\be))-
\be \na X(\al)\big)\nn\\
& &\hspace{.6in}\times
\int^1_0
(\na K^j)(\tau (X(\al)-X(\al-\be))+
(1-\tau)\be\na X(\al))
d\tau
d\sigma(\be)
\nn\\
& &
=\int_{|\be|=\eps}
O(|\be|^{2})
\int^1_0
O(|\be|^{-3})
d\tau
d\sigma(\be),\nn\\
& &=O(\eps).
\nn
\end{eqnarray}
and so the difference between the two quantities we started with tends to zero as $\eps$ tends to zero uniformly in $\al$, but note that by the order of homogeneity of the integrand in
\ba
\int_{|\be|=\eps}K^j(\be \na X(\al))\nu_i(\be) d\sigma(\be),
\ea
we get that the above is independent of $\eps$, and so the limit of the left hand side of \eqref{quantity} as $\eps$ tends to zero is
\ba
\int_{|\be|=1}K^j(\be \na X(\al))\nu_i(\be) d\sigma(\be).
\ea
Moreover we have convergence uniform in $\al$.

The estimate for $T_{\cK}f$ is then retained in the limit as $\eps$ tends to zero from the estimate for $T^\eps_{\cK}f$, uniform in $\eps$.
\end{proof}


\begin{lem}\label{kernellem1'}
There exists a $\delta_0>0$ so that for $0<\delta< \delta_0$, for $X\in B_\delta$, and each $k=1,\ldots,d$, the function $\tilde K^k_X :(\R^3\times\R^3\setminus\{\al=\be\})\to \C^3$ defined by
\ba\label{48'2}
\tilde K^k_X (\al,\be)=\sum^3_{i=1}\del_k X_i(\be) (\del_i K)(X(\al)-X(\be)),
\ea
has corresponding operators $T^\eps_{X,k}:C^1(\R^3)\cap L^2(\R^3)\to L^\infty(\R^3)$ for $\eps>0$ and  $T_{X,k}:C^1(\R^3)\cap L^2(\R^3)\to L^\infty(\R^3)$, defined by
\ba
T^\eps_{X,k}f(\al)=\int_{|\be-\al|\geq\eps} \tilde  K^k_X (\al,\be) f(\be)d\be,
\ea
and (with the $\mbox{p.v.}$ limit converging uniformly in $\al$)
\ba
T_{X,k}f(\al)=p.v.\int \tilde  K^k_X (\al,\be) f(\be)d\be,
\ea
for $\C$-valued $f\in C^1(\R^3)\cap L^2(\R^3)$. We have the following bounds for all $f\in C^1(\R^3)\cap L^2(\R^3)$:
\begin{eqnarray}\label{502}
\lVert T^\eps_{X,k}f \rVert_{L^\infty}&\leq& C(\lVert f \rVert_{L^2}+\lVert f \rVert_{C^1} ),\\
\lVert T_{X,k}f \rVert_{L^\infty}&\leq& C(\lVert f \rVert_{L^2}+\lVert f \rVert_{C^1} ).\label{50}
\end{eqnarray}
\end{lem}
\begin{proof}
First, we fix $k=1,2,$ or $3$, and note
\ba\label{51'}
\tilde K^k_X(\al,\be)=-\del_{\be_k}(K(X(\al)-X(\be))).
\ea
In an application of Lemma \ref{kernellem0'}, a choice of $i$ and $j$ may be prescribed, for $i,j$ as in the statement of the lemma. We will take $i=k$, and a $j=1,2,$ or $3$. In that case, we find for the corresponding $\cK(\al,\be)$ of Lemma \ref{kernellem0'}
\ba
(\tilde K^k_X(\al,\be))_j=-\cK(\al,\be)
\ea
and for $f\in C^1\cap L^2$, and $\eps>0$,
\ba
\left(T^\eps_{X,k} f(\al)\right)_j=-(T^\eps_\cK f)(\al).
\ea
Thus, by Lemma \ref{kernellem0'}, we have uniform convergence of the above as $\eps$ tends to zero, as well as the bound \eqref{502}, and similarly the bound \eqref{50} follows.
\end{proof}

\begin{lem}\label{lemmaA01}
We take a particular choice of a positive constant $\delta_0$, and consider $0<\delta<\delta_0$. Consider any $X\in B_\delta$. For each $\eps>0$, and $k=1,2,$ and $3$,
(i) we have a bounded operator $T^\eps_{X,k}:L^2\to L^2$, mapping a $\C$-valued function $f\in L^2(\R^3)$ to the $\C^3$-valued function defined by
\ba\label{a000}
T^\eps_{X,k}
f(\al)
=\int_{|\be-\al|\geq \eps} \tilde K^k_X(\al,\be) f(\be)d\be\quad\mbox{ for all }f\in L^2,
\ea
where we have the $\C^3$-valued kernel $\tilde K^k_X(\al,\be)$ given by
\ba\label{a001}
\tilde K^k_X(\al,\be)
=
\sum^3_{i=1}
\del_k X_i(\be)(\del_i K)(X(\al)-X(\be)).
\ea
We have the bound
\ba\label{a002}
\lVert T^\eps_{X,k}f\rVert_{L^2}
\leq
C\lVert f \rVert_{L^2},
\ea
for a universal constant $C$.

(ii)
For $f\in L^2$, $T^\eps_{X,k}f$ converges in $L^2$ as $\eps$ tends to zero, and we define the operator $T^0_{X,k}$ on $L^2$ by
\ba\label{a003}
T^0_{X,k}f=\lim_{\eps\to 0} T^\eps_{X,k} f\quad\mbox{ in }L^2.
\ea
This gives a bounded operator from $L^2(\R^3)$ to $L^2(\R^3)$ satisfying the bound
\ba\label{a004}
\lVert T^0_{X,k}f\rVert_{L^2}
\leq
C\lVert f \rVert_{L^2},
\ea
where the constant $C$ is the same as the one in \eqref{a002}.
\end{lem}
\begin{proof}
Let us define $K':\R^3\times \R^3\setminus\{\al=\be\}\to\left(\C^3\right)^2$ by
\ba
K'(\al,\be):=(\na K)(X(\al)-X(\be)).
\ea
Now we check the conditions \eqref{cond2} in Lemma \ref{stein} hold for each of the components $(K')_{ij}$ of $K'$, $i,j=1,2,3$. Observe for $l=1,2,$ and $3$
\begin{eqnarray}
\dal(K'(\al,\be))           &=&\sum^3_{k=1}\dal X_k(\al)(\del_k\na K)(X(\al)-X(\be)),\\
\del_{\beta_l}(K'(\al,\be)) &=&-\sum^3_{k=1}\del_{\beta_l} X_k(\be)(\del_k\na K)(X(\al)-X(\be)).\nn
\end{eqnarray}
With the use of Lemma \ref{kernelbound'}, these imply
\begin{eqnarray}
|\dal(K'(\al,\be))  |         &\leq&\frac{C}{|\al-\be|^{4}},\\
|\del_{\beta_l}(K'(\al,\be))| &\leq&\frac{C}{|\al-\be|^{4}},\nn
\end{eqnarray}
which is sufficient for the conditions \eqref{cond2} to hold for the components of $K'$.

Now we check a bound such as \eqref{cond3} holds for $(I_{\eps,R})_{ij}$ and $(I^*_{\eps,R})_{ij}$ for fixed $i,j=1,2,3$, the quantities corresponding to the $I_{\eps,R}$ and $I^*_{\eps,R}$ in the statement of Lemma \ref{stein} for the entry $(K')_{ij}$.

Note
\ba\label{54}
(K'(\al,\be))_{ij}=-\sum^3_{l=1}(\na X(\be))^{-1}_{il}\del_{\be_l} (K^j(X(\al)-X(\be))).
\ea
We may write $(I_{\eps,R})_{ij}=-\sum^3_{l=1} J^{ijl}_{\eps,R}$, where for $l=1,2,$ and $3$
\begin{eqnarray}
J^{ijl}_{\eps,R}&=&\int_{\eps<|\al-\be| < R} (\na X(\be))^{-1}_{il}
\del_{\be_l}(K^j(X(\al)-X(\be)))d\be\\ &=&
\int_{\eps<|\al-\be| < R} ((\na X(\be))^{-1}-I)_{il}\del_{\be_l}(K^j(X(\al)-X(\be)))d\be\nonumber\\
& &+\int_{\eps<|\al-\be| < R} \delta_{il}\del_{\be_l}(K^j(X(\al)-X(\be)))d\be\nn\\
&=&I^l_1+I^l_2,\nn
\end{eqnarray}
where the first integral, $I^l_1$, is bounded with an application of Lemma \ref{kernellem0'}, with the estimate \eqref{432b}, since $((\na X(\al))^{-1}-I)$ is bounded in both $L^2$ and $C^1$. For $I^l_2$ we note
\begin{eqnarray}\label{53}\\
\int_{\eps<|\al-\be| < R}\del_{\be_l}(K^j(X(\al)-X(\be))d\be
 &=& \left(\int_{|\al-\be|=\eps}+\int_{|\al-\be|=R}\right)K^j(X(\al)-X(\be))\nu_l(\be)d\sigma(\be),\nonumber\\
 &=& \int_{|\al-\be|=\eps}O(\eps^{-2})d\sigma(\be)+\int_{|\al-\be|=R}O(R^{-2})d\sigma(\be),\nn\\
 &=&O(1).\nn
\end{eqnarray}
By summing over $l$, we get a uniform bound on $(I_{\eps,R})_{ij}$. Now, for $(I^*_{\eps,R})_{ij}$, we must consider $(\overline{(K')(\be,\al)})_{ij}$. We note that since $(\na K)(z)$ is even, $K'(\al,\be)=K'(\be,\al)$, and so from the previous estimate on $(I_{\eps,R})_{ij}$ and the fact that complex conjugation does not affect the bound, we also have that $(I^*_{\eps,R})_{ij}$ is uniformly bounded.

Since the $(K'(\al,\be))_{ij}$ satisfy the hypotheses of Lemma \ref{stein}, for each $i,j=1,2,3$, and all $\eps>0$, we have that the operators $T^\eps_{(K')_{ij}}$ are bounded from $L^2$ to $L^2$, and satisfy a bound of the form
\ba
\lVert
T^\eps_{(K')_{ij}}f
\rVert_{L^2}
\leq
C
\lVert
f
\rVert_{L^2}\quad\mbox{for all }f\in L^2.
\ea
With the aim of using this to get a bound for the $T^\eps_{X,k}$ for each $k=1,2,3$ we note that for each $f\in L^2(\R^3)$ and $j=1,2,3$,
\begin{eqnarray}
((T^\eps_{X,k}f)(\al))_j&=&
\sum^3_{i=1}\int_{|\be-\al|\geq \eps}
\del_k X_i(\be)(\del_i K^j)(X(\al)-X(\be))f(\be)d\be,\\
&=&
\sum^3_{i=1}\int_{|\be-\al|\geq \eps} (K'(\al,\be))_{ij}\del_k X_i(\be)f(\be)d\be,\nn\\
&=&\sum^3_{i=1}(T^\eps_{(K')_{ij}}((\del_k X_i)f))(\al).\nn
\end{eqnarray}
With this we find by using the uniform $L^\infty$ bound on $\na X(\al)$
\ba
\lVert
T^\eps_{X,k}f
\rVert_{L^2}
\leq
\sum^3_{i,j=1}
\lVert
T^\eps_{(K')_{ij}}((\del_k X_i)f)
\rVert_{L^2}
\leq
\sum^3_{i=1}
C\lVert (\del_k X_i)f\rVert_{L^2}
\leq C'\lVert f\rVert_{L^2}.
\ea
This proves part (i). For part (ii), first we show that for any $f$ in the dense subspace $C^\infty_c$ of $L^2$, we have convergence of the limit $\lim_{\eps\to 0}
T^\eps_{X,k} f$ in $L^2$. The first thing to note is that for fixed $f\in C^\infty_c(\R^3)$ we have pointwise convergence, by Lemma \ref{kernellem1'}:
\ba
\lim_{\eps \to 0}
T^\eps_{X,k} f(\al)=g(\al)\quad\mbox{ for all }\al\in\R^3,
\ea
for some function $g\in L^\infty$.
Meanwhile, since the kernel $\tilde K^k_X(\al,\be)$ in the definition of $T^\eps_{X,k}$ satisfies the bound
\ba
|\tilde K^k_X(\al,\be)|\leq \frac{C}{|\al-\be|^3},
\ea
for the fixed $f\in C^\infty_c$ we have a bound of the form
\ba\label{a005}
|T^\eps_{X,k} f(\al)|\leq\frac{C_{f}}{(1+|\al|)^3}\quad\mbox{ for all }\al\in\R^3.
\ea
To see this, fix such an $f$ and take $R>0$ large enough that $\mbox{supp} f\subset B_R(0)$. Then
\begin{eqnarray}\label{a006}
T^\eps_{X,k} f(\al)
&=&
\int_{|\be-\al|\geq \eps}
\tilde K^k_X(\al,\be)f(\be)d\be,\\
&=&
\left(\int_{\eps\leq |\be-\al|\leq 100 R}
+
\int_{|\be-\al|> 100R}
\right)
\tilde K^k_X(\al,\be)f(\be)d\be,\nn\\
&=&
I_1+I_2,\nn
\end{eqnarray}
for which we first note that we have
\begin{eqnarray}\label{a007}
I_1=
\int_{\eps \leq |\be-\al|\leq 100 R}
\tilde K^k_X(\al,\be)(f(\be)-f(\al))d\be
+f(\al)\int_{\eps \leq |\be-\al|\leq 100 R}
\tilde K^k_X(\al,\be)d\be.
\end{eqnarray}
Regarding the first term in the right hand side of \eqref{a007}, note that
\begin{eqnarray}\label{a008}\\
\left|\int_{\eps \leq |\be-\al|\leq 100 R}
\tilde K^k_X(\al,\be)(f(\be)-f(\al))d\be
\right|
&\leq&
C\int_{\eps \leq |\be-\al|\leq 100 R}
 \frac{1}{|\al-\be|^{2}}d\be
\left(\sup_{B_{100 R}(\al)}\left|\na f\right|\right)
,\nn\\
&\leq&
C'\sup_{B_{100 R}(\al)}\left|\na f\right|.\nn
\end{eqnarray}
Since $f$ has compact support, this is bounded by a quantity as in the right hand side of \eqref{a005}. For the second term composing $I_1$ in \eqref{a007}, we note
\begin{eqnarray}\label{a009}\\
f(\al)\int_{\eps \leq |\be-\al|\leq 100 R}
\tilde K^k_X(\al,\be)d\be
&=&
-f(\al)\int_{\eps \leq |\be-\al|\leq 100 R}
\del_{\be_k}(K(X(\al)-X(\be)))d\be,\nn\\
&=&
-f(\al)\left(\int_{|\be-\al|=\eps}
O(\eps^{-2})d\sigma(\be)
+\int_{|\be-\al|=100 R}
O(R^{-2})d\sigma(\be)
\right),\nn\\
&=&O(1) f(\al).\nn
\end{eqnarray}
Since $f$ is assumed to have compact support, this is also bounded by a quantity as in the right hand side of \eqref{a005}. Now we give the bound for $I_2$. Observe
\begin{eqnarray}\label{a010}
|I_2|&=&
\left|\int_{|\be-\al|>100R}
\tilde K^k_X(\al,\be) f(\be)d\be
\right|,\\
&\leq&
C\int_{\{|\be-\al|>100R \}\cap\{|\be|<R\}}
\left|\tilde K^k_X(\al,\be)\right|d\be,\nn\\
&\leq &
C'\int_{\{|\be-\al|>100R\}\cap\{|\be|<R\}}
\frac{1}{|\al-\be|^3}d\be,\nn
\end{eqnarray}
where for a given $\al$, the right hand side is $0$ unless for some $|\be|<R$, we have $|\al-\be|>100 R$. This requires $|\al|>99R$. For $|\be|<R$ we would then have $|\be|<\frac{1}{99}|\al|$. Thus for such $\al$ and $\be$ we have
\ba
\frac{1}{|\al-\be|}\leq\frac{C}{|\al|}.
\ea
Using these observations in the last inequality of \eqref{a010}, we get
\begin{eqnarray}
|I_2|
&\leq &
\begin{cases}
C''\int_{|\be|<R}
\frac{1}{|\al|^3}d\be, & \mbox{ if }|\al|>99R\\
0 & \mbox{ if }|\al|\leq 99R,
\end{cases}
\\
&\leq&
C(R) 1_{\R^3\setminus B_R(0)}(\al)\frac{1}{|\al|^3}.\nn
\end{eqnarray}
Thus, $I_2$ in addition to $I_1$ is bounded by a quantity as in the right hand side of \eqref{a005}, and so we conclude that the bound \eqref{a005} holds as written for an appropriate choice of $C_f$.

Since the right hand side of \eqref{a005} is in $L^2$, we have a bound that is uniform in $\eps$ on $|T^\eps_{X,k} f(\al)|^2$ by an integrable function of $\al$. By using the dominated convergence theorem we then see that $g(\al)$, which is the pointwise limit of the $T^\eps_{X,k} f(\al)$ as $\eps$ tends to zero, is in $L^2$. Another application of dominated convergence theorem then gives us that
\ba
\lim_{\eps\to 0}
\int
|T^\eps_{X,k} f(\al)-g(\al)|^2 d\al
=0.
\ea
Thus we have convergence of the limit $\lim_{\eps\to 0} T^\eps_{X,k} f$ in $L^2$ for each $f\in C^\infty_c$.

Now we extend this to convergence of $T^\eps_{X,k} f$ in $L^2$ for any $f\in L^2$. Fix such an $f$. Consider an arbitrary sequence $(\eps_l)_{l\geq 1}$ tending to zero. Given $\eps'>0$, we choose $h\in C^\infty_c$ with $\lVert f-h\rVert_{L^2}\leq \eps'/C$, for $C$ as in the right hand side of \eqref{a002} in part (i) of the lemma. Then for $n,m\geq N$ for an $N$ to be specified momentarily,
\begin{eqnarray}\label{0001}\\
\lVert T^{\eps_n}_{X,k} f-T^{\eps_m}_{X,k} f\rVert_{L^2}
&\leq&
 \lVert T^{\eps_n}_{X,k} f-T^{\eps_n}_{X,k} h\rVert_{L^2}
+\lVert T^{\eps_n}_{X,k} h-T^{\eps_m}_{X,k} h\rVert_{L^2}
+\lVert T^{\eps_m}_{X,k} h-T^{\eps_m}_{X,k} f\rVert_{L^2},\nn\\
&\leq&
2 C\lVert f - h \rVert_{L^2}+\lVert T^{\eps_n}_{X,k} h-T^{\eps_m}_{X,k} h\rVert_{L^2},\nn\\
&\leq &
2\eps' +\lVert T^{\eps_n}_{X,k} h-T^{\eps_m}_{X,k} h\rVert_{L^2}.\nn
\end{eqnarray}
Since $(T^{\eps_l}_{X,k} h)_{l\geq 1}\subset L^2$ is Cauchy, we can choose an $N$ such that the second term in the last line is less than $\eps'$ for all $n,m\geq N$. We conclude that for $n,m\geq N$, the left hand side of \eqref{0001} is no greater than $3\eps'$, and so we deduce that the sequence $(T^{\eps_l}_{X,k} f)_{l\geq 1}$ is Cauchy in $L^2$. Finally, this implies that $T^0_{X,k}f=\lim_{\eps\to 0}T^\eps_{X,k}f$ is a well-defined limit in $L^2$. By using the bound \eqref{a002} from part (i), we get the bound \eqref{a004} on $T^0_{X,k}$ from $L^2$ to $L^2$.
\end{proof}

\end{subsection}

\end{section}


\begin{section}{Counterexample for the compressible Euler equations}

In this section, we consider the $2$D compressible Euler equations, specifically in the isentropic case, given by
\begin{eqnarray}
\del_t\rho +\div(\rho u)=0,\label{013}\\
\del_t u +u\cdot \na u +\frac{1}{\rho}\na p = 0,\label{014}
\end{eqnarray}
where $(\rho,u)=(\rho(x,t),u(x,t)):\R^2\times \R\to(0,\infty)\times\R^2$ represent the density and velocity, respectively, of a fluid or gas, and the pressure $p$ is a known function of $\rho$, i.e. $p=p(\rho)$, which for the sake of concreteness, we will take to be $p(\rho)=A\rho^\ga$ for constants $A>0$ and $\ga>1$, a typical example.

Regarding questions of Lagrangian analyticity for compressible fluid equations, in \cite{fz} and \cite{serfati}, the question is discussed whether one has analytic trajectories for inviscid compressible fluid equations, for solutions with low spatial regularity, as one has for the incompressible Euler equations. The results in \cite{fz} do provide an example of a kind of compressible fluid model which does have this property, and we saw that the pressureless Euler-Poisson model in Section 4 also provides such an example. In this section, however, we prove that the Lagrangian trajectories for solutions to the standard compressible fluids model, the compressible Euler equations, are not in general analytic, notably distinct from previous studies, even for $C^\infty$ initial data in our case.

\nocite{fz2,sideris2,alinhac1,kv,kv2,chemin,cvk,bf,bb}

To this end we provide a counterexample with initial data which is smooth and compactly supported\footnote{For the density, we actually mean that $(\rho|_{t=0}-\orho)$ is compactly supported, for a constant $\orho>0$.
}, for which we have local existence of a unique classical solution to the problem \eqref{013}-\eqref{014}, and for which there exist Lagrangian trajectories that are not analytic in any small disc about $t=0$.

The basic strategy for constructing our counterexample is to take advantage of the finite speed of propagation property exibited by systems of conservation laws such as \eqref{013}-\eqref{014}. We take initial data with zero velocity outside of a ball, and so the Lagrangian trajectories of fluid parcels outside the support of $u$ at $t=0$ remain constant for a positive time, before the influence has propagated to these parcels. We use an elementary argument to prove that some nontrivial influence does indeed propagate away from the initial support and reach fluid parcels which were previously motionless, and at a time before the development of shocks. Thus we find that a particle moves, but was stationary for an interval of time before that moment, and so its trajectory is smooth but cannot be analytic.

We remark that for the model in \cite{fz} as well as the pressureless Euler-Poisson model considered in Section 4, the fact that the pressure term is absent is a key difference from typical fluids models, and so they are inherently different from the standard compressible Euler equations studied in this section. The presence of the pressure is an important part of the conservation law structure in the compressible Euler equations, and the conservation law structure leads to the finite speed of propagation in the system. With the method here, we show that the finite speed of propagation in the system is incompatible with automatic analyticity of Lagrangian trajectories, and one can likely use this method to show this is the case for particle trajectories in other systems of continuum mechanics with a finite speed of propagation. We make the related comment that the study \cite{frv} builds on the strategies of \cite{fz} to prove analyticity of particle trajectories in a related Newtonian cosmological model, and the authors raise the question of whether the results might extend to a relativistic cosmological model. Since having a finite speed of propagation is characteristic of relativistic models, it is suggested by the basic idea behind the results in this section that incorporating these effects could in fact break this property.

\begin{subsection}{Non-analytic trajectories as a consequence of finite speed of propagation}
The initial data we use to get non-analytic trajectories is a smooth, compactly supported perturbation of a constant state, with zero velocity outside a ball of radius $1$. We give the precise definition below.
\begin{definition}\label{020b}
Fix constants $\orho>0$ and $\eps>0$. We define $\rho_0(x):=\orho+\eps\tilde{\rho}(x)$, $u_0(x):=\eps \tilde{u}(x)$, with
\ba\label{017}
\tilde{\rho}(x):=
\begin{cases}
e^{\frac{1}{|x|^2-1}}, & \mbox{for } |x|<1,\\
0 & \mbox{for } |x|\geq 1,
\end{cases}
\ea
and
\ba\label{018}
\tilde{u}(x):=
\begin{cases}
\frac{x}{|x|}e^{q(|x|)} & \mbox{for } |x| < 1,\\
0 & \mbox{for } |x|\geq 1,
\end{cases}
\ea
where $q(r)=\frac{1}{r^2(r^2-1)}$.
\end{definition}

Despite the typical eventual formation of shocks for solutions to \eqref{013}-\eqref{014}, it is not difficult to show that as long as the initial data is sufficiently small, the classical solution to the system exists for large times. The optimal classical existence results for the system with initial data such as ours, namely smooth, compactly supported, radial perturbations of size $\eps$, are proven in \cite{alinhac2} and \cite{sideris1}, where classical existence is shown up to times of order $\frac{1}{\eps^2}$, for sufficiently small $\eps$. Having classical existence for large times allows us to ensure that $C^\infty$ smoothness is retained in an interval of time surrounding the time $t$ at which we show that some trajectory cannot be analytic.
\begin{prop}\label{AlinhacProp}
There exists $\eps_0>0$ such that for $0<\eps<\eps_0$, the following holds. Let $T_\eps$ be the maximal time of existence for the unique classical solution $(\rho,u)$ to the system \eqref{013}-\eqref{014} with initial data $(\rho_0,u_0)$. Then
\ba\label{022}
\lim_{\eps\to 0} T_\eps =\infty.
\ea
Moreover, the solution $(\rho,u)$ lies in $C^\infty([0,T_\eps)\times\R^2)$.
\end{prop}
\begin{proof}
By the results in \cite{alinhac2,sideris1}, for our initial data a classical solution exists up to a time $t=C\eps^{-2}$ for $\eps<\eps_0$, for some constants $C>0$, $\eps_0>0$. However, the proof of this result is a bit technical, and so here we provide a simple scaling argument that gives a lower bound on $T_\eps$ sufficient for \eqref{022}.

We note that the system \eqref{013}-\eqref{014} satisfies the property that for any $\lambda>0$, $(\rho,u)(x,t)$ is a classical solution with initial data $(\rho^0,u^0)(x)$ if and only if
\ba
(\rho_\lambda,u_\lambda)(x,t):=(\rho, u)(\lambda x,\lambda t)
\ea
is a classical solution with initial data $(\rho^0_\lambda,u^0_\lambda)(x)=(\rho^0,u^0)(\lambda x)$. One easily checks that for a fixed $s>0$ and large $\lambda>0$
\begin{eqnarray}
\lVert (\rho^0_\lambda-\orho)\rVert_{H^s}&\leq& C \lambda^{s-1} \lVert (\rho^0-\orho) \rVert_{H^s},\label{023}\\
\lVert u^0_\lambda \rVert_{H^s} &\leq & C\lambda^{s-1} \lVert u^0 \rVert_{H^s},\label{024}
\end{eqnarray}
for a constant $C=C(s)$. The standard local existence theory for symmetric hyperbolic systems of conservation laws shows the following: for $s>2$, if for a fixed constant $\tilde C>0$ we have $\lVert (\rho^0_\lambda-\orho,u^0_\lambda)\rVert_{H^s}\leq \tilde C$, the initial data $(\rho^0_\lambda,u^0_\lambda)$ yield classical solutions to \eqref{013}-\eqref{014} in $C([0,T],H^s(\R^2))\cap C^1([0,T],H^{s-1}(\R^2))$ for $T=T(s,\tilde C)$. Let us take $s=3$. Now we simply note that taking $(\rho^0,u^0)(x):=(\rho_0,u_0)(x)$ as in Definition \ref{020b} and $\lambda=\eps^{-1/2}$, the right hand sides of \eqref{023} and \eqref{024} are uniformly bounded for all small $\eps$, say $\eps<\eps_0$. Since for initial data $(\rho^0_\lambda,u^0_\lambda)(x)$ the classical solution $(\rho_\lambda,u_\lambda)(x,t)$ then extends up to a time $T$ independent of $\eps<\eps_0$, we find that the classical solution $(\rho,u)(x,t)$ with initial data $(\rho_0,u_0)(x)$ extends up to time $T\eps^{-1/2}$. This is enough for \eqref{022}, though we note that one can get an order $\eps^{-1}$ lower bound on $T_\eps$ by taking $s$ closer to $2$.

Given that $(\rho,u)$ is a classical solution on $[0,T_\eps)$ with smooth initial data, one can upgrade this to a solution that is smooth in both space and time by using standard energy estimates for conservation laws. We direct the reader to Theorem 2.2 from \cite{majda} for an explicit reference.
\end{proof}

With the following proposition, we give a precise formulation of the finite speed of propagation property of the system \eqref{013}-\eqref{014}.
\begin{prop}\label{FiniteSpeed}
If $(\rho,u)$ is a classical solution of \eqref{013}-\eqref{014} on the time interval $[0,T]$ with initial data
\ba\label{019}
(\rho(x,0),u(x,0))=(\rho_0(x),u_0(x)),
\ea
for $(\rho_0,u_0)$ as in Definition \ref{020b}, then for $0\leq t\leq T$
\ba\label{020}
(\rho,u)\equiv(\orho,0)\;\;\textrm{on }D(t),
\ea
where we define $\sigma:=\sqrt{p'(\orho)}$ and
\ba\label{021}
D(t):=\{x\in\R^2:|x|\geq 1+\sigma t\}.
\ea
\end{prop}
\begin{proof}
This follows from the proposition in \cite{sideris3}, which applies to a class of systems of conservation laws that includes \eqref{013}-\eqref{014}.
\end{proof}
Now we establish the main proposition, which proves that some particular Lagrangian trajectory cannot be analytic.
\begin{prop}\label{main}
Let $(\rho,u)$ be the solution to \eqref{013}-\eqref{014} as discussed in Proposition \ref{AlinhacProp}, and let $X(\al,t)$ denote the Lagrangian trajectory map tracking the position of the particle beginning at $\al\in\R^2$, defined by the equation,
\begin{eqnarray}
\frac{dX}{dt}(\al,t)&=&u(X(\al,t),t),\\
X(\al,0)&=&\al.\nonumber
\end{eqnarray}
Then there is an $\al\in \R^2$ and a time $t_0\in(0,T_\eps)$ such that the trajectory $X(\al,t)$ is not analytic in $t$ at $t_0$.
\end{prop}
\begin{proof}
Ultimately, all we need to prove is that, where $B:=\{|x|<1\}$,
\ba\label{000}
B(t):=X(B,t)\neq B\quad\textrm{for some }t\in(0,T_\eps).
\ea
Indeed, suppose we can establish that. Then by continuity of the trajectory map, we have $X(\als,t)\neq\als$ for some $\als$ near the boundary of $B$ but outside $\overline{B}$. This implies there must be a $t^*\in(0,T_\eps)$ such that
\ba\label{001}
\frac{dX}{dt}(\als,t^*)\neq 0.
\ea
Now consider
\ba\label{012}
t(\als):=\inf\{t\in[0,T_\eps):\frac{dX}{dt}(\als,t)\neq 0\}.
\ea
By Proposition \ref{FiniteSpeed}, we know that since $\als$ lies outside $\overline{B}$, $\frac{dX}{dt}(\als,t)=0$ is maintained for all $t$ in an interval $[0,\delta]$ for some $\delta>0$, and so $t(\als)>0$. In fact, $\frac{dX}{dt}(\als,t)=0$ on the time interval $[0,t(\als)]$, and yet there is a sequence $\{t_n\}_{n\geq 1}\subset(0,T_\eps)$ with $t_n\to t(\als)$ and $\frac{dX}{dt}(\als,t_n)\neq0$. Taking $t_0:=t(\als)$, we find that the trajectory $X(\als,t)$ is not analytic at $t_0$.

So now we will show that $B(t)$ cannot be $B$ for all times $t\in(0,T_\eps)$.

Assume $B(t)\equiv B$. First we claim that then $(\rho-\orho)$ and $u$ are supported in $\overline{B}$ for all $t\in[0,T_\eps)$. To see this, note that if $X(B,t)\equiv B$, then for each $t$, $u$ must vanish to infinite order along $\del B$. From the continuity equation \eqref{013}, one finds that this implies $(\rho-\orho)$ also must vanish to infinite order along $\del B$ for each $t$. Thus by defining
\ba
(\rho^\flat,u^\flat):=\begin{cases}(\rho,u) & \mbox{for }|x|<1,\\
(\orho,0) & \mbox{for }|x|\geq 1,\end{cases}
\ea
we also get a smooth solution to the system \eqref{013}-\eqref{014} with initial data $(\rho_0,u_0)$. By uniqueness, $(\rho^\flat,u^\flat)\equiv(\rho,u)$, and so $(\rho-\orho,u)$ is supported in $\overline{B}$. With this in mind, we define the quantities
\ba
\cM(t):=\int\rho(x,t) u(x,t)\cdot xdx,\quad\quad
I(t):=\int(\rho(x,t)-\orho)|x|^2dx.
\ea
Here $\cM(t)$ gives a weighted average of radial momentum, and $I(t)$ represents the moment of inertia, modulo the constant state. Observe that as a consequence of the fact that $\rho$ satisfies the continuity equation \eqref{013}, we have the identity
\ba\label{005ce}
I'(t)=\int \del_t \rho(x,t)|x|^2dx=-\int \div (\rho(x,t) u(x,t))|x|^2dx=2\int \rho(x,t)u(x,t)\cdot x dx=2\cM(t).
\ea
Now we claim $\cM'(t)\geq 0$. To verify this one can use the conservation law derived from the system \eqref{013}-\eqref{014}
\ba
\del_t(\rho u_i)+\sum^2_{j=1}\del_{x_j}(\rho u_i u_j+\delta_{ij}p)=0,\quad\quad i=1,2,
\ea
to find that
\ba\label{007}
\cM'(t)=\int \del_t (\rho(x,t)u(x,t))\cdot xdx=\int\left(\rho(x,t)|u(x,t)|^2+2(p(\rho)-p(\orho))\right)dx.
\ea
We claim the right hand side of \eqref{007} is positive. To see this, note
\begin{eqnarray}
\frac{1}{m(B)}\int_B p(\rho)dx
=
\frac{1}{m(B)}
\int_B (\rho(x,t))^\ga dx
\geq
\left(
\frac{1}{m(B)}
\int_B \rho(x,t)dx
\right)^\ga,
\end{eqnarray}
where the right hand side is conserved, and so equal to its value at $t=0$, which is
\ba
\left(
\frac{1}{m(B)}
\int_B \rho_0(x,t)dx
\right)^\ga
\geq
\left(
\frac{1}{m(B)}
\int_B \orho dx
\right)^\ga
=\orho^\ga
=p(\orho).
\ea
This means we have
\ba
\int( p(\rho)-p(\orho) )dx \geq 0,
\ea
which, when used in \eqref{007}, verifies that $\cM'(t)\geq 0$. Using this with the identity \eqref{005ce}, we find
\ba\label{008}
I(t)=I(0)+\int^t_0 2\cM(s)ds\geq I(0)+2t \cM(0)\geq 2t \cM(0).
\ea
Meanwhile, since $\int_B\rho(x,t) dx$ must be constant,
\ba\label{009}
I(t)= \int_B (\rho(x,t)-\orho)|x|^2dx\leq\int_B(\rho(x,t)-\orho)dx=\int_B(\rho_0(x)-\orho)dx\leq C' \eps.
\ea
Using the definitions of $\tilde u$ and $u_0$, noting $\orho\leq \rho_0(x)$, and combining \eqref{008} with \eqref{009}, we have
\begin{eqnarray}\label{010}
2\eps\left( \int \orho\tilde{u}(x)\cdot x dx\right) t
&\leq& 2\int\orho u_0(x)\cdot x dxt,\\
&\leq&  2\int\rho_0(x) u_0(x)\cdot x dx t,\nn\\
&=& 2 t \cM(0),\nn\\
&\leq& I(t)\leq C' \eps \quad\textrm{for all }t\in[0,T_\eps).\nn
\end{eqnarray}
Note the integral in the left hand side of \eqref{010} is just a nonzero constant independent of $\eps$ that one can calculate from the initial data. Thus this implies
\ba
t \leq C'' \quad\textrm{for all }t\in[0,T_\eps),\nn
\ea
for some constant $C''$. However $\lim_{\eps\to 0}T_\eps=\infty$, and so for small enough $\eps$ we have a contradiction. We conclude that \eqref{000} holds, and as we saw, that implies the desired result for the proposition.
\end{proof}
Typically, Lagrangian analyticity for incompressible flows is established for a brief positive time about $t=0$. The above example explicitly contrasts this if we start time at $t_0$.
\begin{theo}
There exists $C^\infty$ initial data $(\tilde\rho_0,\tilde u_0)$ for the system \eqref{013}-\eqref{014} such that there is a corresponding smooth solution $(\rho,u)$ on a time interval $[0,T]$ for some $T>0$ with a corresponding Lagrangian trajectory $X(\al,t)$ which is not analytic at $t=0$.
\end{theo}
\begin{proof}
We first begin with the solution $(\rho,u)$ given by Proposition \ref{AlinhacProp} for initial data $(\rho_0,u_0)$, as defined in Definition \ref{020b}. We then consider the corresponding $t_0$ given by Proposition \ref{main}. We take $(\rho(x,t_0),u(x,t_0))$ as our smooth initial data $(\tilde\rho_0,\tilde u_0)$ for the system \eqref{013}-\eqref{014}. It immediately follows that one does not have analyticity of some resulting Lagrangian trajectory at $t=0$.
\end{proof}
\end{subsection}

\end{section}

\begin{section}{Counterexample for the Vlasov-Poisson system}
The Vlasov-Poisson equation gives a probabilistic model of the dynamics of a plasma, or a physical system with a large number of particles that are governed by a long-range interaction, such as with Coulomb's law. The unknown is a density function $f(x,v,t)$, defined on $\R^{2d}\times \R$, describing the distribution of the positions and velocities of a large number of particles.
Here, we deal with the  Vlasov-Poisson equation in the case $d=2$, which is
\ba\label{000vp}
\del_t f(x,v,t)+v\cdot\na_x f(x,v,t)+E(x,t)\cdot\na_v f(x,v,t)=0,
\ea
for the unknown density $f(x,v,t):\R^2\times\R^2\times\R\to\R$ with initial data $f|_{t=0}=f_0\in \mathcal{S}(\R^4)$, and where we have that the field $E(x,t):\R^2\times\R\to\R^2$ is related to the unknown $f$ via
\begin{eqnarray}
E(x,t)&=&-\na_x\phi(x,t),\\
-\Delta_x\phi(x,t)&=&q\rho(x,t),\\
\rho(x,t)&=&\int f(x,w,t)dw.\label{001c}
\end{eqnarray}
Here we have real valued potential $\phi(x,t)$ and positive density function $\rho(x,t)$, defined on $\R^2\times\R$. We may take $q=\pm 1$ for the case of a repulsive or attractive force law. In what follows we just take $q=1$, though the analysis works for either case. In this case, we think of the system as governing a distribution of charged particles, which experience repulsion with each other.

We will observe that this system also has a Lagrangian formulation which can naturally be written as an ODE. The main goal of this section is to demonstrate that unlike the situation for the incompressible Euler equations, there exist solutions with non-analytic Lagrangian trajectories, even for smooth initial data.

\begin{subsection}{Lagrangian phase space formulation}

Several key features are shared between the Lagrangian dynamics of the Vlasov-Poisson equation and fluid mechanics equations such as the incompressible Euler system. For example, where the curl is transported by the flow for the $2$D incompressible Euler equations, for the Vlasov-Poisson equations we have that the unknown density $f(x,v,t)$ is constant when evaluated along phase space trajectories in the corresponding Hamiltonian system. We also have that the main governing force entering in the Vlasov-Poisson equation can be written as a well-behaved nonlocal operator, analogous to the operator that arises from the Biot-Savart law used for incompressible flows. This can be used to write this trajectory system as an ODE in an abstract function space, as can be done for the incompressible Euler equations. Now we derive the Lagrangian formulation for the system, observing these various properties as they appear in the derivation.

For the Vlasov-Poisson system, the natural Lagrangian coordinates formulation is a description in terms of trajectories in phase space. Throughout the following we use the notation $\zeta=(q,p)\in\R^2\times\R^2$ to denote a point in phase space, where each $q\in\R^2$ represents a spatial position and each $p\in\R^2$ a momentum, or velocity. Thus, a Lagrangian trajectory beginning at the point $\zeta$ in phase space is then given by a vector valued function
\ba
(X(\zeta,t),V(\zeta,t)):\R^2\times\R^2\times\R\to\R^2\times\R^2,
\ea
where
\ba
(X(\zeta,0),V(\zeta,0))=\zeta.
\ea

The system \eqref{000vp}-\eqref{001c} has a Hamiltonian system formulation, given by
\begin{eqnarray}\label{002}
\dot{X}&=&V,\\
\dot{V}&=&E(X,t),\nonumber
\end{eqnarray}
in terms of the phase space trajectories $(X(\zeta,t),V(\zeta,t))$, with initial condition $(X(\zeta,0),V(\zeta,0))=\zeta$.

Additionally, we have
\ba\label{003}
E(x,t)=\na\Delta^{-1}\rho(x,t)=\int \na_x\Delta_x^{-1}f(x,\xi,t)d\xi
=\int \int K(x-y)f(y,\xi,t)dyd\xi,
\ea
with the integral kernel for the operator $\na_x\Delta_x^{-1}$ given by
\ba\label{004}
K(x)=\frac{1}{2\pi}\frac{x}{|x|^2},\quad x\neq 0.
\ea
Making the change of variables $(y,\xi)=(X(\Xi,t),V(\Xi,t))$ in the integral on the right in \eqref{003}, so that the integration is now over $\Xi$ in $\R^2\times\R^2$, and evaluating at $x=X(\zeta,t)$, one finds that
\ba\label{005}
E(X(\zeta,t),t)=\int K (X(\zeta,t)-X(\Xi,t))f(X(\Xi,t),V(\Xi,t),t)J_{X,V}(\Xi,t)d\Xi,
\ea
where $J_{X,V}(\Xi,t)$ is the resulting Jacobian. However, one can derive from \eqref{002} that $J_{X,V}(\Xi,t)=1$. This is equivalent to the statement that volume in phase space is preserved under the natural flow for this Hamiltonian system, as dictated by Liouville's theorem.

Observe equation \eqref{000vp} together with the system \eqref{002} implies that $f$ is constant along the trajectories in phase space, so that
\ba\label{006}
f(X(\Xi,t),V(\Xi,t),t)=f_0(\Xi).
\ea
Incorporating this property in \eqref{005} we get
\ba
E(X(\zeta,t),t)=
\int K (X(\zeta,t)-X(\Xi,t))f_0(\Xi)d\Xi,
\ea
and returning to \eqref{002}, we now have the system
\begin{eqnarray}\label{008d}
\frac{dX}{dt}(\zeta,t)&=&V(\zeta,t),\\
\frac{dV}{dt}(\zeta,t)&=&\int K(X(\zeta,t)-X(\Xi,t))f_0(\Xi)d\be d\Xi,\nn\\
(X,V)(\zeta,0)&=&\zeta.\nn
\end{eqnarray}
We claim that although the system \eqref{008d} has certain similarities with the Lagrangian formulations for the incompressible Euler equations and the Euler-Poisson equation, it has very different analytic properties. We record the main result of this section, Theorem \ref{maintheoVP}, below.
\begin{theonn}
There exists $C^\infty$ initial data $f_0(x,v)$ for the system \eqref{000vp}-\eqref{001c} with a global $C^\infty$ solution $f(x,v,t)$ which has the following property. For the phase-space trajectory map $(X(\zeta,t),V(\zeta,t))$ solving the Hamiltonian formulation of the system, \eqref{002}, for $E(x,t)$ as in \eqref{003}, there is some trajectory $(X(\zeta_0,t),V(\zeta_0,t))$ which is not analytic in time at $t=0$.
\end{theonn}

We comment that the counterexample in this section is notably different from that in the section on the compressible Euler equations. The Vlasov-Poisson system has an infinite speed of propagation, so the result here provides an example of a different mechanism in a fluid equation which can cause the analyticity of the Lagrangian trajectories to fail.

\begin{rem}
One can verify that the natural operator acting on the trajectory map in the formulation \eqref{008d} in fact does not preserve analyticity in time, that is, it does not satisfy the criterion \eqref{thesis002} discussed in Section 2. The failure of this property is not directly used in an obvious way in our proof of non-analyticity, though it does appear to be a key difference in the system from the incompressible Euler equations, as well as other fluid mechanics systems which are known to have the Lagrangian analyticity property.
\end{rem}
\end{subsection}

\begin{subsection}{Preliminary estimates}
The first step for our construction of the counterexample for this system is to establish bounds for various derivatives of the trajectory map at $t=0$. To establish these preliminary bounds, in addition to using the equation \eqref{000vp} for the unknown $f(x,v,t)$, we will frequently use the Hamiltonian formulation of the system given by
\begin{eqnarray}
\frac{dX}{dt}(\zeta,t)&=&V(\zeta,t),\\
\frac{dV}{dt}(\zeta,t)&=&E(X(\zeta,t),t),\label{vpcef02}\\
(X(\zeta,0),V(\zeta,0))&=&\zeta.
\end{eqnarray}
with the equation for the electric field
\ba
E(x,t)=\int\int K(x-y)f(y,\xi,t)dyd\xi.
\ea
We now proceed with a basic bound on the initial derivatives of some key quantities in the above system.

\begin{lem}\label{vpceL01}
Let $f_0(x,v)$ be a nonnegative Schwartz function, which we take to be initial data for the Vlasov-Poisson equation. Let the corresponding solution be denoted by $f(x,v,t)$, with electric field $E(x,t)$ and Lagrangian phase space trajectory map $(X(\zeta,t),V(\zeta,t))$. Then for any multi-index $\al$ and nonnegative integer $n$ we have for all $\zeta\in\R^2\times\R^2$
\ba\label{vpceA001}
|\del^\al_\zeta\del^n_t V(\zeta,0)|\leq C_{\al,n}(1+|p|)^{n+1},
\ea
where $\zeta=(q,p)\in\R^2\times\R^2$, and where $C_{\al,n}$ depends only on $\al$, $n$, and the initial data $f_0(x,v)$. Moreover, for any multi-index $\ga$ we have for all $x\in\R^2$
\ba\label{vpceA002}
|\del^\ga_{x,t}E(x,0)|\leq C_\ga,
\ea
where $C_\ga$ depends only on $\ga$ and the initial data $f_0(x,v)$.
\end{lem}
\begin{proof}
Consider for a given $s\in\N$ the estimate \eqref{vpceA001} for arbitrary $\al$ and $n$ with $|\al|+n\leq s$ and \eqref{vpceA002} for $\ga$ with $|\ga|\leq s-1$. If we can show such estimates hold for all such $\al$, $n$, and $\ga$, for any given $s\geq 0$, then we will have proven the lemma. We will do this by induction on $s$.

Let us check the base case, $s=0$. Note that the condition involving \eqref{vpceA002} in the base case is vacuously true, since there exists no $\ga$ with $|\ga|=-1$. We find that \eqref{vpceA001} holds in the base case, i.e. when $|\al|+n=0$, as a result of the fact that $V(\zeta,0)=p$.

Now we move on to the induction step. Let us fix an $s\in\N$ and suppose that \eqref{vpceA001} holds for all $\al$ and $n$ with $|\al|+n\leq s$ and that \eqref{vpceA002} holds for all $\ga$ with $|\ga|\leq s-1$. Consider $\al$ and $n$ with $|\al|+n=s+1$. Note that if $n=0$, then \eqref{vpceA001} follows from the fact that $V(\zeta,0)=p$. Suppose instead $n\geq 1$. Then, by using \eqref{vpcef02}, we find
\begin{eqnarray}\label{vpceA003}
\del^\al_\zeta \del^n_t V(\zeta,0)
&=&
\del^\al_\zeta\del^{n-1}_t [E(X(\zeta,t),t)]_{t=0},\\
&=&
\sum_{
\substack{
0\leq|\ga|\leq s
\\
\bm{\be},\vec{j}
}
}
\textit{coeff}\;
\left(
\prod^{|\ga_x|}_{i=1}
\del^{\be_i}_{\zeta,t} X_{j_i}(\zeta,0)
\right)
\del^\ga_{x,t}E(q,0),\label{vpceA003b}
\end{eqnarray}
using the notation $\bm{\be}=(\be_1,\ldots,\be_{|\ga_x|})$, where each $\beta_i$ is a multi-index, and $\vec{j}=(j_1,\ldots,j_{|\ga_x|})$, with integer components, and we have multi-index $\ga=(\ga_x,\ga_t)$. We have written $\textit{coeff}$ in the above sum to indicate some integer coefficients whose exact values are not particularly important. Now consider any given summand in the sum in \eqref{vpceA003b}. We use the facts that $V=\frac{dX}{dt}$ and that $|\be_i|\leq s$ for each $i$ together with the induction hypothesis to bound the corresponding product:
\ba\label{vpceA004}
\left|
\prod^{|\ga_x|}_{i=1}
\del^{\be_i}_{\zeta,t}X_{j_i}(\zeta,0)
\right|
\leq C\prod^{|\ga_x|}_{i=1}(1+|p|)^{(\be_i)_t}\leq C(1+|p|)^{n-1},
\ea
where we have written $(\be_i)_t$ to denote the $t$ component of each multi-index $\be_i$. We have also used the fact that $\sum_{i}(\be_i)_t\leq n-1$, which is clear from \eqref{vpceA003}-\eqref{vpceA003b}. Also, regarding the sum in \eqref{vpceA003b}, by using the induction hypothesis we get the bound
\ba\label{vpceA005}
|\del^\ga_{x,t}E(q,0)|\leq C\quad\mbox{for }\ga\mbox{ with }|\ga|\leq s-1.
\ea
Using the \eqref{vpceA004} and \eqref{vpceA005} in \eqref{vpceA003}-\eqref{vpceA003b} we obtain
\ba\label{vpceA006}
\del^\al_\zeta\del^n_t V(\zeta,0)
=O((1+|p|)^{n-1})
+
\sum_{|\ga|=s}O((1+|p|)^{n-1})\del^\ga_{x,t}E(q,0).
\ea
Now we note
\ba\label{vpceA007}
E(x,t)=\int\int K(y)f(x-y,\xi,t)dyd\xi,
\ea
which gives us for any multi-index $\ga$
\ba\label{vpceA008}
\del^\ga_{x,t} E(x,t)
=
\int\int K(y)
\del^\ga_{x,t}f(x-y,\xi,t)dyd\xi.
\ea
Now we note that since $f$ solves the Vlasov-Poisson equation, we have
\ba\label{vpceA009}
\del_t f = -(v\cdot \na_x f+E\cdot\na_v f).
\ea
With repeated use of the above formula, we obtain a formula for higher time derivatives, for any integer $k\geq0$.
\ba\label{vpceA010}
\del^k_t f =
(-1)^k(v\cdot \na_x)^k f+
\sum_{\kappa,\bm{\eta},\mu,\nu} \textit{coeff}\;v^\kappa
\left(
\prod^r_{i=1}\del^{\eta_i}_{x,t} E
\right)
\del^\mu_x\del^\nu_v f,
\ea
where the summation is over some finite collection of $\bm \eta=(\eta_1,\ldots,\eta_{r})$, where each $\eta_i$ is a multi-index for $i=1,\ldots,r$, and we have multi-indices $\kappa,\;\mu,$ and $\nu$. Moreover, in the summation we have $|\mu|\leq k-1$, $|\mu|+|\nu|\leq k$, and for each $i$, $|\eta_i|\leq k-1$. This is easily checked by induction. Note that we include the case $k=0$, in which we have an empty sum in the right hand side.

For each multi-index $\ga=(\ga_x,\ga_t)$, by taking $k=\ga_t$ in the above equation and applying $\del^{\ga_x}_x$ to both sides, we find
\ba\label{vpceA011}
\del^\ga_{x,t}f
=
(-1)^{\ga_t}(v\cdot \na_x)^{\ga_t} \del^{\ga_x}_x f+
\sum_{\kappa,\bm{\eta},\mu,\nu} \textit{coeff}\;v^\kappa
\left(
\prod^r_{i=1}\del^{\eta_i}_{x,t} E
\right)
\del^\mu_x\del^\nu_v f,
\ea
where in the sum we have $|\mu|\leq |\ga|-1$, $|\mu|+|\nu|\leq|\ga|$, and each $|\eta_i|\leq |\ga|-1$.
Using \eqref{vpceA011} in \eqref{vpceA008} and evaluating at $t=0$, we get
\begin{eqnarray}\label{vpceA012}
\del^\ga_{x,t}E(x,0)
&=&
(-1)^{\ga_t}\int\int
K(y)(\xi\cdot\na_x)^{\ga_t}\del^{\ga_x}_x f_0(x-y,\xi)dyd\xi\\
& &+
\sum\textit{coeff}
\int\int
K(y)\xi^\ka\left(
\prod^r_{i=1}\del^{\eta_i}_{x,t} E(x-y,0)
\right)
\del^\mu_x\del^\nu_v f_0(x-y,\xi)dyd\xi.\nn
\end{eqnarray}
For each $\ga$ with $|\ga|=s$, we find for each of the integrals in the sum in the second line of \eqref{vpceA012}, since $|\eta_i|\leq|\ga|-1=s-1$ for each $i$, the induction hypothesis gives us bounds of the form
\ba\label{vpceA013}
\left|
\prod^r_{i=1}\del^{\eta_i}_{x,t} E(x-y,0)
\right|\leq C.
\ea
Using this together with the fact that $f_0(x,v)$ is Schwartz in \eqref{vpceA012}, we conclude that
\ba\label{vpceA014}
|\del^\ga_{x,t}E(x,0)|\leq C\quad\textrm{for }\ga\textrm{ with }|\ga|=s.
\ea
Using \eqref{vpceA014} in \eqref{vpceA006}, we find that \eqref{vpceA001} follows for arbitrary $\al$ and $n$ with $|\al|+n= s+1$.

In addition, we know that \eqref{vpceA002} holds for any multi-index $\ga$ with $|\ga|= s$, directly from \eqref{vpceA014}. This finishes the induction step. Thus, \eqref{vpceA001} and \eqref{vpceA002} hold for all $\al$, $n$, and $\ga$.
\end{proof}

For our construction, we will also need to carefully consider the effect of making perturbations on the initial data $f_0(x,v)$ on the derivatives of the trajectory map. Now we proceed to define a special kind of ``additional bump'' which we will use to perturb a gaussian initial data $f_0(x,v)$ for the Vlasov-Poisson equation. The aim is to produce a perturbation whose derivatives up to a certain order are bounded, but with the property that beyond that order, the derivatives can be large.

\begin{defn}\label{vpceD02}
In the following, we define for $0<\delta< 1$ and integer $N\geq 2$ the function
\ba\label{vpce002}
F_{\delta,N}(x,v):=
\delta^{N-2}
\left(\frac{|x|^2}{\delta^2}-1
\right)
e^{-\frac{|x|^2}{\delta^2}}
g_N(v),
\ea
where $g_N(v)$ is given by
\ba\label{vpce0023}
g_N(v):=A_N e^{-c_N|v-v_N|^2},
\ea
for constants $A_N>0$, $c_N>1$, and $v_N\in\R^2$ chosen as follows.

Consider an arbitrary smooth function $h(x,t):\R^2\times\R\to\R$, and a smooth vector valued function $a(t):\R\to\R^2$. Consider the following formula:
\ba\label{vpce0021}
\frac{d^N}{dt^N}\left(
h(a(t),t)
\right)
=
\sum_{\substack{0\leq|\ga|\leq N\\ \vec k,\vec j}}
b_N(\ga,\vec k,\vec j)
\left(
\prod^{|\ga_x|}_{i=1}\frac{d^{k_i}a_{j_i}(t)}{dt^{k_i}}
\right)
\del^\ga_{x,t}h(a(t),t).
\ea
In the sum we have multi-indices $\ga=(\ga_x,\ga_t)$, vectors with integer components $\vec k=(k_1,\ldots,k_{|\ga_x|})$ and $\vec j=(j_1,\ldots,j_{|\ga_x|})$, and nonnegative integer coefficients $b_N(\ga,\vec k,\vec j)$ that do not depend on the choice of $h(x,t)$ and $a(t)$.
Consider the polynomial $p_N(\xi):\R^2\to\R$ defined by the following, where $\delta^i_j$ denotes the Kronecker delta:
\ba\label{vpce0022}
p_N(\xi):=\sum_{\substack{|\ga|= N\\ \vec k,\vec j}}
(-1)^{\ga_t}\delta^1_{j_1}\ldots\delta^1_{j_{|\ga_x|}}b_N(\ga,\vec k,\vec j)w_\ga(\xi),
\ea
where
\ba\label{wgadef}
w_\ga(\xi):=(\xi\cdot\na_x)^{\ga_t}\del_{x_1}\del^{\ga_x}_x\left(e^{-|x|^2}\right)\Big|_{x=0}.
\ea
First we comment that for any odd $N$, $p_N(\xi)$ is not the constant zero polynomial. This can be seen by observing that the coefficient of $\xi^N_1$ is then nonzero. Indeed, one finds that coefficient is then
\ba
\sum_{\substack{(\ga_t,\ga_x)=(N,0,\ldots,0)\\ \vec k,\vec j}}
(-1)^N\delta^1_{j_1}\ldots\delta^1_{j_{|\ga_x|}}b_N(\ga,\vec k, \vec j)
\Bigg(\del^{N+1}_{x_1}\left(e^{-|x|^2}\right)\Big|_{x=0}\Bigg) \neq 0.
\ea
If $N$ is odd, we choose the constants $A_N>0$, $c_N>1$, and $v_N\in\R^2$ appearing in \eqref{vpce0023} so that we have
\ba
\int p_N(\xi)g_N(\xi)d\xi\neq 0.
\ea
Such choices of $A_N$, $c_N$, and $v_N$ are clearly possible since for odd $N$, $p_N(\xi)$ is not the constant zero polynomial. If $N$ is even, we choose $A_N=1$, $c_N=2$, and $v_N=0$.
\end{defn}
\begin{lem}\label{vpceL07}
For any integers $N\geq 2$ and $M\geq 0$, and $F_{\delta,N}(x,v)$ as defined in Definition \ref{vpceD02}, we have the following bound for some constant $C'_{N,M}>0$:
\ba\label{vpce0461}
\left(1+|x|^2+|v|^2\right)^{M/2}\left|
\del^\al_{x,v}F_{\delta,N}(x,v)
\right|\leq C'_{N,M}\quad\mbox{for any multi-index }\al\mbox{ with }|\al|\leq N-2\mbox{ and all }\delta\in(0,1).
\ea
\end{lem}
\begin{proof}
This is easily verified from the definition of the function $F_{\delta,N}(x,v)$.
\end{proof}

We are now ready to consider perturbations of gaussian initial data for the system. The main strategy is to very precisely find the effect that perturbing the initial data has on the time derivatives of the Lagrangian trajectory map at $t=0$. It will prove useful to track the effect on the various mixed partial derivatives of $V(\zeta,t)$ and $E(x,t)$ at $t=0$ in particular. By the following lemma we find that if we perturb a kind of gaussian initial data by a function of the type given in Definition \ref{vpceD02}, we can ensure the derivatives at $t=0$ up to a certain order of the resulting Lagrangian trajectory map and electric field are close to the corresponding values for the unperturbed solution.
\begin{lem}\label{vpceL02}
Consider a fixed, positive function $f_0(x,v)$ that can be written in the following form, for some integer $K\geq 0$, and $a_i>1$, $b_i>1$, $v^0_i\in\R^2$, and polynomials $p^i(x)$, for each $i=1,\ldots,K$:
\ba\label{vpcef10}
f_0(x,v)=e^{-|x|^2-|v|^2}+\sum^K_{i=1}p^i(x)e^{-a_i |x|^2-b_i|v-v^0_i|^2},
\ea
It follows that we have the existence and uniqueness of a global classical solution to the Vlasov-Poisson equation with initial data $f_0(x,v)$. We denote the solution to the Vlasov-Poisson equation with initial data $f_0(x,v)$ by $f(x,v,t)$, with the corresponding electric field denoted by $E(x,t)$, and Lagrangian trajectory map given by $(X(\zeta,t),V(\zeta,t))$.

Consider integer $N\geq 2$, $\delta\in(0,1)$, and $\eps\in[0,\eps_{0,N}]$, for some $\eps_{0,N}>0$ to be specified momentarily. Let us take $F_{\delta,N}(x,v)$ as in Definition \ref{vpceD02}, and define
\ba\label{vpce003}
f^0_{\eps,\delta,N}(x,v):=
f_0(x,v)+\eps F_{\delta,N}(x,v).
\ea
We are able to pick an $\eps_{0,N}>0$ small enough that for all $\eps\in[0,\eps_{0,N}]$ and $\delta\in(0,1)$ we have $f^0_{\eps,\delta,N}(x,v)>0$ everywhere, and so we select such an $\eps_{0,N}$.

For initial data $f^0_{\eps,\delta,N}(x,v)$, we denote the corresponding solution to the Vlasov-Poisson system by $f_{\eps,\delta,N}(x,v,t)$, with electric field $E_{\eps,\delta,N}(x,t)$, and trajectory map $\left(X_{\eps,\delta,N}(\zeta,t), V_{\eps,\delta,N}(\zeta,t)\right)$.

For any multi-index $\al$ and integer $m\geq 0$ with $|\al|+m\leq N$, all $\eps\in[0,\eps_{0,N}]$ and $\delta\in(0,1)$, and where $\zeta=(q,p)$, we have
\ba\label{vpce004}
\left|\del^\al_{\zeta}\del^m_t V_{\eps,\delta,N} (\zeta,0)
-\del^\al_{\zeta}\del^m_t V(\zeta,0)\right|
\leq C_N(1+|p|)^{m+1}\eps.
\ea
Additionally, for any multi-index $\ga$ with $|\ga|\leq N-1$, all $\eps\in[0,\eps_{0,N}]$ and $\delta\in(0,1)$,
\ba\label{vpce005}
\left|\del^\ga_{x,t} E_{\eps,\delta,N} (x,0)
-\del^\ga_{x,t} E(x,0)\right|
\leq C_N\eps.
\ea
The constant $C_N$ appearing in the right hand sides of \eqref{vpce004} and \eqref{vpce005} depends only on $N$ and the unperturbed initial data $f_0(x,v)$.
\end{lem}
\begin{proof}
First we comment that uniqueness and global existence of solutions to the $2D$ Vlasov-Poisson equation, for initial data of the form \eqref{vpcef10}, is a consequence of the classical result due to Ukai and Okabe in \cite{uo} for initial data with sufficiently rapid decay.

Now let us fix such initial data $f_0(x,v)$ and an integer $N\geq2$. First we check that $f^0_{\eps,\delta,N}(x,v)$ is positive for all $\eps\in[0,\eps_{0,N}]$ and $\delta\in(0,1)$, for small enough $\eps_{0,N}$. To verify this, one considers separately the behavior of the function $f^0_{\eps,\delta,N}(x,v)$ outside of a large closed ball in $\R^4$, say $B_R$, and inside $B_R$, separately. Outside of the ball, as long as $\eps\leq 1$ and $R$ is large enough, $f^0_{\eps,\delta,N}(x,v)$ is positive since it is the sum of $e^{-|x|^2-|v|^2}$ and a more rapidly decaying function. Note that in addition to the fact that for each $i$, $a_i>1$ and $b_i>1$, we have used that $\delta<1$ and $c_N>1$, which determines the decay of $F_{\delta,N}(x,v)$, as made apparent by \eqref{vpce002} and \eqref{vpce0023}. Inside the ball, we can ensure $f^0_{\eps,\delta,N}(x,v)$ is positive for sufficiently small $\eps$ since it is the sum of $f_0(x,v)$, which is uniformly bounded below in $B_R$ by a positive number, and a perturbation, $\eps F_{\delta,N}(x,v)$, whose maximum magnitude is $O(\eps)$. Thus we may choose an $\eps_{0,N}>0$ small enough to guarantee that $f^0_{\eps,\delta,N}(x,v)$ is positive everywhere for $\eps\in[0,\eps_{0,N}]$ and $\delta\in(0,1)$.

Consider $\eps\in[0,\eps_{0,N}]$ and $\delta\in(0,1)$.
For ease of reading, let us replace our notation for the perturbed initial data $ f^0_{\eps,\delta,N}(x,v)$ by $ \tilde f_0(x,v)$, and replace $f_{\eps,\delta,N}(x,v,t)$ by $\tilde f(x,v,t)$, $E_{\eps,\delta,N}(x,t)$ by $\tilde E(x,t)$, and $X_{\eps,\delta,N}(\zeta,t)$ and $V_{\eps,\delta,N}(\zeta,t)$ by $\tilde X(\zeta,t)$ and $\tilde V(\zeta,t)$, respectively. We prove the remaining assertions of the lemma by induction on integer $s=0,\ldots,N$, proving that we have a bound of the form \eqref{vpce004} whenever $|\al|+m\leq s$, and that \eqref{vpce005} holds whenever $|\ga|\leq s-1$, for each such $s$.

For the base case, $s=0$, we get \eqref{vpce004} from noting $\tilde V(\zeta,0)=V(\zeta,0)=p$. The condition for \eqref{vpce005} in the base case is vacuously true since there is no multi-index $\ga$ with $|\ga|=-1$. Now we proceed with the inductive step. Consider a fixed integer $s$ with $0\leq s \leq N-1$, and suppose that \eqref{vpce004} holds for all $\al$, $m$ with $|\al|+m\leq s$, and for all $\ga$ with $|\ga|\leq s-1$ that \eqref{vpce005} holds.

First we verify that \eqref{vpce004} holds for $\al$ and $m$ with $|\al|+m=s+1$. Suppose $m=0$. Then we get \eqref{vpce004} from the fact that $\tilde V(\zeta,0)=V(\zeta,0)=p$. Now suppose $m\geq 1$. We use \eqref{vpcef02} for $\frac{dV}{dt}$ and $\frac{d\tilde V}{dt}$ to find
\begin{eqnarray}\label{vpce006}\\
\del^\al_\zeta\del^m_t \tilde V(\zeta,0)
-\del^\al_\zeta\del^m_t V(\zeta,0)
&=&
\left(
\del^\al_\zeta \del^{m-1}_t \left[\tilde E(\tilde X (\zeta,t),t)\right]
-\del^\al_\zeta \del^{m-1}_t [E(X(\zeta,t),t)]
\right)\Big|_{t=0},\nn\\
&=&
\sum_{
\substack{
0\leq|\ga|\leq s
\\
\bm{\be},\vec{j}
}
}
\textit{coeff}\;\Bigg\{
\Bigg(
\prod^{|\ga_x|}_{i=1}
\del^{\be_i}_{\zeta,t}\tilde X_{j_i} (\zeta,0)
\Bigg)
\del^\ga_{x,t}\tilde E(q,0)\label{vpce006c}\\
& &\hspace{2.5cm}
-\Bigg(
\prod^{|\ga_x|}_{i=1}
\del^{\be_i}_{\zeta,t}X_{j_i}(\zeta,0)
\Bigg)
\del^\ga_{x,t}E(q,0)
\Bigg\},\nn\\
&=&
\sum_{\substack{
0\leq|\ga|\leq s\\
\bm{\be},\vec{j}}
}\textit{coeff}\;
\Bigg\{
\Bigg(
\prod^{|\ga_x|}_{i=1}\del^{\be_i}_{\zeta,t}\tilde X_{j_i}(\zeta,0)
-\prod^{|\ga_x|}_{i=1}\del^{\be_i}_{\zeta,t}X_{j_i}(\zeta,0)
\Bigg)\del^\ga_{x,t}\tilde E(q,0)\label{vpce006b}\\
& &\hspace{2.5cm}
+\Bigg(\prod^{|\ga_x|}_{i=1}\del^{\be_i}_{\zeta,t}X_{j_i}(\zeta,0)\Bigg)
\left(\del^\ga_{x,t}\tilde E(q,0)- \del^\ga_{x,t}E(q,0)
\right)
\Bigg\}\nn.
\end{eqnarray}
First, in any of the summands above in \eqref{vpce006b}, we may write
\begin{eqnarray}\label{vpce007}\\
\prod^{|\ga_x|}_{i=1}\del^{\be_i}_{\zeta,t}\tilde X_{j_i}(\zeta,0)
-\prod^{|\ga_x|}_{i=1}\del^{\be_i}_{\zeta,t}X_{j_i}(\zeta,0)
&=&
\left[
\del^{\be_1}_{\zeta,t}\tilde X_{j_1}(\zeta,0)-\del^{\be_1}_{\zeta,t} X_{j_1}(\zeta,0)\right]\left(\prod^{|\ga_x|}_{i=2}\del^{\be_i}_{\zeta,t}\tilde X_{j_i}(\zeta,0)\right)+\ldots\nn\\
& &+
\left(
\prod^{|\ga_x|-1}_{i=1}\del^{\be_i}_{\zeta,t} X_{j_i}(\zeta,0)\right)
\left[
\del^{\be_{|\ga_x|}}_{\zeta,t}\tilde X_{j_{|\ga_x|}}(\zeta,0)-\del^{\be_{|\ga_x|}}_{\zeta,t} X_{j_{|\ga_x|}}(\zeta,0)\right].\nn
\end{eqnarray}
Above, we have a sum of terms, each of which has for some $l=1,\ldots,|\ga_x|$ a factor of the form
\ba\label{vpce008}
\left[
\del^{\be_l}_{\zeta,t}\tilde X_{j_l}(\zeta,0)-\del^{\be_l}_{\zeta,t} X_{j_l}(\zeta,0)
\right],
\ea
multiplied by the two products
\ba\label{vpce009}
\left(
\prod^{l-1}_{i=1}\del^{\be_i}_{\zeta,t} X_{j_i}(\zeta,0)\right)
\left(
\prod^{|\ga_x|}_{i=l+1}\del^{\be_i}_{\zeta,t} \tilde X_{j_i}(\zeta,0)\right).
\ea
First, note that the induction hypothesis gives us the following, observing that $|\be_l|\leq s$ for each $l$.
\ba\label{vpce010}
\left|
\del^{\be_l}_{\zeta,t}\tilde X_{j_l}(\zeta,0)-\del^{\be_l}_{\zeta,t} X_{j_l}(\zeta,0)
\right|
\leq C(1+|p|)^{(\be_l)_t}\eps.
\ea
Meanwhile, Lemma \ref{vpceL01} gives us the bound for each $i$
\ba\label{vpce011}
\left|\del^{\be_i}_{\zeta,t} X_{j_i}(\zeta,0)\right|
\leq C(1+|p|)^{(\be_i)_t}.
\ea
Using this together with \eqref{vpce010} we get for each $i$
\ba\label{vpce012}
\left|\del^{\be_i}_{\zeta,t} \tilde X_{j_i}(\zeta,0)\right|
\leq C'(1+|p|)^{(\be_i)_t}.
\ea
Thus we find for each $l$ that for the product of \eqref{vpce008} and \eqref{vpce009}, we have the bound
\begin{eqnarray}\label{vpce013}\\
\left|
\left[
\del^{\be_l}_{\zeta,t}\tilde X_{j_l}(\zeta,0)-\del^{\be_l}_{\zeta,t} X_{j_l}(\zeta,0)
\right]\left(
\prod^{l-1}_{i=1}\del^{\be_i}_{\zeta,t} X_{j_i}(\zeta,0)\right)
\left(
\prod^{|\ga_x|}_{i=l+1}\del^{\be_i}_{\zeta,t} \tilde X_{j_i}(\zeta,0)\right)
\right|
&\leq&
C\eps \left(\prod^{|\ga_x|}_{i=1}(1+|p|)^{(\be_i)_t}\right)\nn\\
&\leq& C(1+|p|)^{m-1}\eps.\nn
\end{eqnarray}
Here we have also used the observation that $\sum_i(\be_i)_t\leq m-1$, which is clear from \eqref{vpce006}-\eqref{vpce006c}. Using \eqref{vpce013} in \eqref{vpce007}, we obtain
\ba\label{vpce014}
\left|
\prod^{|\ga_x|}_{i=1}\del^{\be_i}_{\zeta,t}\tilde X_{j_i}(\zeta,0)
-\prod^{|\ga_x|}_{i=1}\del^{\be_i}_{\zeta,t}X_{j_i}(\zeta,0)
\right|
\leq C(1+|p|)^{m-1}\eps.
\ea

Now, we claim that for any multi-index $\ga$ with $0\leq|\ga|\leq s$, we have the bound
\ba\label{vpce015}
\left|\del^\ga_{x,t} \tilde E(q,0)-\del^\ga_{x,t}E(q,0)\right|
\leq C\eps.
\ea
Note that this is immediate as a consequence of the induction hypothesis if $0\leq|\ga|\leq s-1$. It just remains to handle the case $|\ga|=s$. In the following, we use the same formula \eqref{vpceA012} for calculating $\del^\ga_{x,t} E(x,0)$ as well as $\del^\ga_{x,t}\tilde E(x,0)$. This gives us
\begin{eqnarray}\label{vpce016}\\
\del^\ga_{x,t}\tilde E(x,t)-\del^\ga_{x,t}E(x,t)
&=&
(-1)^{\ga_t}\int\int
K(y)(\xi\cdot\na_x)^{\ga_t}\left(\del^{\ga_x}_x \tilde f_0(x-y,\xi)-\del^{\ga_x}_x f_0(x-y,\xi)\right)dyd\xi\nn\\
& &+
\sum\textit{coeff}
\int\int
K(y)\xi^\ka\Bigg\{\Bigg(
\prod^r_{i=1}\del^{\eta_i}_{x,t} \tilde E(x-y,0)
\Bigg)
\del^\mu_x\del^\nu_v \tilde f_0(x-y,\xi)\nn\\
& &\hspace{4cm}
-
\Bigg(
\prod^r_{i=1}\del^{\eta_i}_{x,t} E(x-y,0)
\Bigg)
\del^\mu_x\del^\nu_v f_0(x-y,\xi)
\Bigg\}
dyd\xi,\nn\\
&=&
I_1+I_2,\nn
\end{eqnarray}
where in the summation defining $I_2$, in each summand we have $|\mu|\leq |\ga|-1=s-1$, $|\mu|+|\nu|\leq s$, and each $|\eta_i|\leq s-1$. We also remark that the second sum $I_2$ is zero in the case that $|\ga|=0$. Let us consider $I_2$ in the case that $|\ga|\neq 0$. We have
\begin{eqnarray}\label{vpce017}\\
I_2
&=&
\sum\textit{coeff}
\int\int
K(y)\xi^\ka\Bigg(
\prod^r_{i=1}\del^{\eta_i}_{x,t}
\tilde E(x-y,0)
-
\prod^r_{i=1}\del^{\eta_i}_{x,t}
E(x-y,0)
\Bigg)
\del^\mu_x\del^\nu_v \tilde f_0(x-y,\xi)dyd\xi\nn\\
& &
+
\sum\textit{coeff}
\int\int
K(y)\xi^\ka
\Bigg(
\prod^r_{i=1}\del^{\eta_i}_{x,t} E(x-y,0)
\Bigg)
\left(
\del^\mu_x\del^\nu_v\tilde f_0(x-y,\xi)
-\del^\mu_x\del^\nu_v f_0(x-y,\xi)
\right)
dyd\xi,\nn\\
&=&I^a_2+I^b_2.\nn
\end{eqnarray}
First we bound the difference of products appearing in the term $I^a_2$. We note that we may write for any $x$
\begin{eqnarray}\label{vpce018}
\prod^r_{i=1}\del^{\eta_i}_{x,t}
\tilde E(x,0)
-
\prod^r_{i=1}\del^{\eta_i}_{x,t}
E(x,0)
&=&
\left[
\del^{\eta_1}_{x,t} \tilde E(x,0)
-\del^{\eta_1}_{x,t} E(x,0)
\right]
\left(\prod^{r}_{i=2}\del^{\eta_i}_{x,t} \tilde E(x,0)\right)+\ldots\\
& &
+
\left(\prod^{r-1}_{i=1}\del^{\eta_i}_{x,t} E(x,0)\right)
\left[
\del^{\eta_r}_{x,t} \tilde E(x,0)
-\del^{\eta_r}_{x,t} E(x,0)
\right].\nn
\end{eqnarray}
To bound this, first we note that since $|\eta_i|\leq s-1$ for each $i=1,\ldots,r$, we have from the induction hypothesis that
\ba\label{vpce019}
\left|
\del^{\eta_i}_{x,t} \tilde E(x,0)
-\del^{\eta_i}_{x,t} E(x,0)
\right|\leq C\eps.
\ea
Furthermore, using Lemma \ref{vpceL01} in addition to this, we find for each $i$
\begin{eqnarray}\label{vpce020}
\left|\del^{\eta_i}_{x,t} E(x,0)\right|
&\leq& C,\\
\left|\del^{\eta_i}_{x,t} \tilde E(x,0)\right|
&\leq& C'.
\end{eqnarray}
Using these bounds in \eqref{vpce018} we find
\ba\label{vpce021}
\left|
\prod^r_{i=1}\del^{\eta_i}_{x,t}
\tilde E(x,0)
-
\prod^r_{i=1}\del^{\eta_i}_{x,t}
E(x,0)
\right|
\leq C\eps
\ea
Now we claim that for each $\mu$, $\nu$ occurring in the sums in \eqref{vpce017} we have
\ba\label{vpce022}
\del^\mu_x\del^\nu_v\tilde f_0(x,v)
-\del^\mu_x\del^\nu_v f_0(x,v)=\eps h_\delta(x,v)
\ea
for a function $h_\delta(x,v)$ satisfying for any $M\geq0$
\ba\label{vpce023}
\int\int|\xi|^M |h_\delta(y,\xi)|dyd\xi\leq C_M,
\ea
where, in particular, $C_M$ is independent of $\delta\in(0,1)$.
To see this, recalling the definition of $F_{\delta,N}(x,v)$ in \eqref{vpce002}, we note
\begin{eqnarray}\label{vpce024}
\del^\mu_x\del^\nu_v \tilde f_0(x,v)-\del^\mu_x\del^\nu_v f_0(x,v)
&=&
\eps\del^\mu_x\del^\nu_v\left(F_{\delta,N}(x,v)\right),\\
&=&
\eps \delta^{N-2}\del^\mu_x\del^\nu_v
\left(
\left(\frac{|x|^2}{\delta^2}-1\right)e^{-\frac{|x|^2}{\delta^2}}g_N(v)
\right),\nn\\
&=&
\eps\delta^{N-2}\del^\mu_x
\left[
\left(\frac{|x|^2}{\delta^2}-1\right)e^{-\frac{|x|^2}{\delta^2}}
\right]
\del^\nu_v g_N(v).\nn
\end{eqnarray}
Since $|\mu|\leq s-1\leq N-2$, one then deduces \eqref{vpce022} holds, where we have $h_\delta(x,v)$ given by the following for some polynomial $p(x)$.
\ba
h_\delta(x,v)=\delta^{N-2-|\mu|}p(x/\delta)e^{-\frac{|x|^2}{\delta^2}}\del^\nu_v g_N(v),
\ea
and indeed one finds $h_\delta(x,v)$ then satisfies \eqref{vpce023}.

Now we use \eqref{vpce021} along with \eqref{vpce022} in \eqref{vpce017} to find
\ba\label{vpce025}
I^a_2=\sum\textit{coeff}\int\int
K(y)\xi^\kappa O(\eps)\left(
\del^\mu_x\del^\nu_v f_0(x-y,\xi)+\eps h_\delta(x-y,\xi)
\right)dyd\xi=O(\eps),
\ea
since $f_0(x,v)$ is Schwarz and $h_\delta(x,v)$ satisfies \eqref{vpce023} for any $M\geq0$. We remark that in our use of $O(\cdot)$ notation, we mean to indicate that the constants in the corresponding bounds are independent of both $\eps$ and $\delta$. Meanwhile, using Lemma \ref{vpceL01} together with \eqref{vpce022} and \eqref{vpce023}, recalling the definition of $I^b_2$ from \eqref{vpce017} we find
\ba\label{vpce026}
I^b_2
=\sum\textit{coeff}
\int\int
K(y)\xi^\kappa O(1)\eps h_\delta(x-y,\xi)dyd\xi=O(\eps).
\ea
So with this we have found that $I_2=I^a_2+I^b_2=O(\eps)$. Now we handle $I_1$. Recall from \eqref{vpce016} we have
\begin{eqnarray}\label{vpce027}
I_1
&=&
(-1)^{\ga_t}\int\int
K(y)(\xi\cdot\na_x)^{\ga_t}\left(\del^{\ga_x}_x \tilde f_0(x-y,\xi)-\del^{\ga_x}_x f_0(x-y,\xi)\right)dyd\xi,\\
&=&
(-1)^{\ga_t}\eps
\int\int
K(y)(\xi\cdot\na_x)^{\ga_t}\del^{\ga_x}_x F_{\delta,N}(x-y,\xi)dyd\xi,\nn
\end{eqnarray}
where for some polynomial $p_\ga(x,v)$
\begin{eqnarray}\label{vpce028}
(v\cdot\na_x)^{\ga_t}\del^{\ga_x}_x F_{\delta,N}(x,v)
&=&
\delta^{N-2}(v\cdot\na_x)^{\ga_t}\del^{\ga_x}_x
\left[
\left(
\frac{|x|^2}{\delta^2}-1
\right)e^{-\frac{|x|^2}{\delta^2}}
\right]g_N(v),\\
&=&
\delta^{N-2-s}p_\ga(x/\delta,v)e^{-\frac{|x|^2}{\delta^2}}g_N(v).\nn
\end{eqnarray}
Above, we have used the fact that we are handling the case in which $|\ga|=s$ (which, we recall, was the only remaining case, as we noted below \eqref{vpce015}).
We use \eqref{vpce028} in \eqref{vpce027} and find
\begin{eqnarray}\label{vpce029}
I_1
&=&
(-1)^{\ga_t}\eps \delta^{N-2-s}\int\int K(x-y)p_\ga(y/\delta,\xi)
e^{-\frac{|y|^2}{\delta^2}}g_N(\xi)dyd\xi,\\
&=&
(-1)^{\ga_t}\eps\delta^{N-s}\int\int K(x-\delta y)p_\ga(y,\xi)e^{-|y|^2}g_N(\xi)dyd\xi,\nn\\
&=&
(-1)^{\ga_t}\eps\delta^{N-s-1}\int\int K(x/\delta- y)p_\ga(y,\xi)e^{-|y|^2}g_N(\xi)dyd\xi,\nn\\
&=&\eps\delta^{N-s-1}O(1).\nn
\end{eqnarray}
Since $s\leq N-1$, we then find $I_1=O(\eps)$. Using this along with our deduction that $I_2=O(\eps)$, and recalling \eqref{vpce016}, we conclude that \eqref{vpce015} holds for any $\ga$ with $0\leq |\ga|\leq s$. Moreover, by Lemma \ref{vpceL01}, for any $\ga$ we find
\ba\label{vpce030}
\left|\del^\ga_{x,t} E(q,0)\right|\leq C.
\ea
Using this with \eqref{vpce015}, which we have now shown for all $\ga$ in the sum in \eqref{vpce006b}, we also find
\ba\label{vpce030b}
\left|\del^\ga_{x,t} \tilde E(q,0)\right|\leq C'.
\ea
Now we use \eqref{vpce011} for each $i$ and the fact that $\sum_i(\be_i)_t\leq m-1$ for all $\ga$ in the sum in \eqref{vpce006b} to verify
\ba\label{vpce03001}
\left|
\prod^{|\ga_\zeta|}_{i=1}\del^{\be_i}_{\zeta,t} X_{j_i}(\zeta,0)
\right|\leq C(1+|p|)^{m-1}.
\ea
Returning to \eqref{vpce006b}, we use \eqref{vpce014}, \eqref{vpce015},  \eqref{vpce030b}, and \eqref{vpce03001}, and conclude that
\ba\label{vpce03002}
\left|
\del^\al_\zeta\del^m_t \tilde V(\zeta,0)
-\del^\al_\zeta\del^m_t V(\zeta,0)
\right|
\leq C(1+|p|)^{m-1}\eps
\leq C(1+|p|)^{m+1}\eps.
\ea
Thus we have \eqref{vpce004} for arbitrary $\al$ and $m$ with $|\al|+m = s+1$. Moreover, recall that above we established \eqref{vpce015} for arbitrary multi-index $\ga$ with $|\ga|=s$. Thus, we have also shown \eqref{vpce005} for the desired multi-indices $\ga$, which concludes the induction step. So it follows that \eqref{vpce004} and \eqref{vpce005} hold for all $\al$, $m$, and $\ga$ with $|\al|+m\leq N$ and $|\ga|\leq N-1$.

\end{proof}

With the above lemma, we found that we can ensure that the effects of a certain perturbation of the initial data on the corresponding $V(\zeta,t)$ are very small, at least for the derivatives up to some finite order at $t=0$. There is essentially one other main feature of our strategy, which is to show that we can create perturbations which have solutions whose trajectory maps actually have large time derivatives of some order at $t=0$.

Now we show that for a derivative of a high enough order, we can make the effect of the perturbation quite large. This will be key in constructing initial data with a corresponding $V(\zeta,t)$ whose higher order derivatives at $t=0$ grow too quickly for the function to be analytic in time.
\begin{lem}\label{vpceL03}
Consider a fixed positive function $f_0(x,v)$ as in the statement of Lemma \ref{vpceL02}, which we take as initial data for the Vlasov-Poisson system. We denote the corresponding solution by $f(x,v,t)$, with electric field $E(x,t)$, and Lagrangian trajectory map $(X(\zeta,t),V(\zeta,t))$. For each integer $N\geq2$, we define the perturbed initial data $f^0_{\eps,\delta,N}(x,v)$ for $\eps$ and $\delta$ with $0\leq\eps\leq \eps_{0,N}$ and $0<\delta<1$ exactly as in the statement of Lemma \ref{vpceL02}, with corresponding solution to the Vlasov-Poisson system $f_{\eps,\delta,N}(x,v,t)$, with electric field $E_{\eps,\delta,N}(x,t)$, and Lagrangian trajectory map $(X_{\eps,\delta,N}(\zeta,t),V_{\eps,\delta,N}(\zeta,t))$.

Let us consider a fixed odd integer $N\geq 3$ and $\eps\in(0,\eps_{0,N}]$. We have for sufficiently small $\delta>0$ the inequality
\ba\label{vpce031}
\left|
\frac{d^{N+1} V_{\eps,\delta,N}}{dt^{N+1}}(0,e_1,0)
\right|
\geq C_{N,\eps}\delta^{-1},
\ea
where $e_1=(1,0)$, for some constant $C_{N,\eps}>0$, dependent only on $N$, $\eps$, and the unperturbed initial data $f_0(x,v)$.
\end{lem}
\begin{proof}
Let us fix an odd $N\geq 3$ and $\eps\in(0,\eps_{0,N}]$, and consider varying $\delta\in(0,1)$. For ease of reading, we replace our notation for the function $f^0_{\eps,\delta,N}$ with $\tilde f_0$, and similarly we replace the notation for the corresponding solution $f_{\eps,\delta,N}$ by $\tilde f$, $E_{\eps,\delta,N}$ by $\tilde E$, $X_{\eps,\delta,N}$ by $\tilde X$, and $V_{\eps,\delta,N}$ by $\tilde V$. We use \eqref{vpcef02} to find
\begin{eqnarray}\label{vpce032}
\frac{d^{N+1}\tilde V}{dt^{N+1}}(0,e_1,0)
&=&
\frac{d^N}{dt^N}\left(
\tilde E(\tilde X(0,e_1,t),t)
\right)\bigg|_{t=0},\\
&=&
\sum_{\substack{1\leq |\ga|\leq N\\\vec{k},\vec{j}}}
\textit{coeff}\;
\left(
\prod^{|\ga_x|}_{i=1}
\del^{k_i}_t \tilde X_{j_i}(0,e_1,0)
\right)
\del^\ga_{x,t} \tilde E(0,0).\nn
\end{eqnarray}
In the above sum, we find for each of the products, we can write each multiplicand as
\ba\label{vpcef11}
\del^{k_i}_t \tilde X_{j_i}(0,e_1,0)=\del^{k_i}_t  X_{j_i}(0,e_1,0)+\left(\del^{k_i}_t \tilde X_{j_i}(0,e_1,0)-\del^{k_i}_t  X_{j_i}(0,e_1,0)\right),
\ea
and since each of the $k_i$ satisfy $1\leq |k_i|\leq N$, and we have $\tilde V=\frac{d \tilde X}{dt}$ and $V=\frac{dX}{dt}$, by applying Lemmas \ref{vpceL01} and \ref{vpceL02}, we find the quantity in \eqref{vpcef11} is $O(1)$. That is, it is bounded by a constant indepedent of $\eps$ and $\delta$ (though only dependence on $\delta$ is relevant). Thus each of the products of these terms that we find in the sum in \eqref{vpce032} is $O(1)$.

Consider $\ga$ with $|\ga|\leq N-1$. We then find in each of the corresponding summands $\del^\ga_{x,t}\tilde E(0,0)=O(1)$, by using Lemmas \ref{vpceL01} and \ref{vpceL02}. So in fact, recalling the definition of the coefficients $b_N(\ga,\vec k,\vec j)$ from Definition \ref{vpceD02}, we find
\ba\label{vpce033}
\frac{d^{N+1}\tilde V}{dt^{N+1}}(0,e_1,0)
=
\sum_{\substack{1\leq|\ga|\leq N-1\\ \vec k, \vec j}}\big(O(1)\mbox{ terms}\big)+
\sum_{\substack{|\ga|=N\\\vec k, \vec j}}\delta^1_{j_1}\ldots\delta^1_{j_{|\ga_x|}}b_N(\ga,\vec k,\vec j)\del^\ga_{x,t} \tilde E(0,0).
\ea
Here we have used the fact that in the sum in \eqref{vpce032}, whenever $|\ga|=N$, for each $i$ we must have $k_i=1$, and then $\del^{k_i}_{t} \tilde X_{j_i}(0,e_1,0)=\tilde V_{j_i}(0,e_1,0)=\delta^1_{j_i}$.

Consider $\ga$ with $|\ga|=N$. With the use of the formula \eqref{vpceA012} obtained in the proof of Lemma \ref{vpceL01}, we calculate
\begin{eqnarray}\label{vpce034}
\del^\ga_{x,t}\tilde E(0,0)
&=&
\int\int
K(y)(-1)^{\ga_t}(\xi\cdot\na_x)^{\ga_t}\del^{\ga_x}_x \tilde f_0(-y,\xi)dyd\xi,\\
& &
+\sum\textit{coeff}\;
\int\int
K(y)\xi^\kappa
\left(
\prod^r_{i=1}
\del^{\eta_i}_{x,t} \tilde E (-y,t)
\right)
\del^\mu_x\del^\nu_v \tilde f_0(-y,\xi)dyd\xi,\nn\\
&=&
I_1+I_2.\nn
\end{eqnarray}
where in each term in the sum we have $|\eta_i|\leq |\ga|-1=N-1$ for each $i$, $|\mu|\leq N-1$, and $|\mu|+|\nu|\leq N$.

First we claim $I_2=O(1)$. To see this, first note since each $|\eta_i|\leq N-1$ that Lemmas \ref{vpceL01} and \ref{vpceL02} imply for each summand a bound of the form
\ba\label{vpce035}
\left|
\prod^r_{i=1}
\del^{\eta_i}_{x,t} \tilde E (x,t)
\right|
\leq C.
\ea
Now consider the following, recalling the definition of $\tilde f_0(x,v)=f^0_{\eps,\delta,N}(x,v)$ in \eqref{vpce003} and of $F_{\delta,N}(x,v)$ in Definition \ref{vpceD02}.
\begin{eqnarray}\label{vpce036}
\del^\mu_x\del^\nu_v \tilde f_0(x,v)
&=&
\del^\mu_x\del^\nu_v f_0(x,v)
+
\eps \delta^{N-2}\del^\mu_x\left[\left(\frac{|x|^2}{\delta^2}-1\right)e^{-\frac{|x|^2}{\delta^2}}\right]
\del^\nu_v g_N(v),\\
&=&
\del^\mu_x\del^\nu_v f_0(x,v)
+
\eps
\delta^{N-2-|\mu|}p_\mu(x/\delta)e^{-\frac{|x|^2}{\delta^2}}\del^\nu_v g_N(v),\nn
\end{eqnarray}
for a polynomial $p_\mu(x)$. Let us write for each term in the sum $\Gamma=(\kappa,\bm\eta,\mu,\nu)$ in \eqref{vpce034}. Then for each $\Gamma$ in the summation we define
\ba\label{vpce037}
\tilde h_\Gamma(x,t):=
\prod^r_{i=1}
\del^{\eta_i}_{x,t} \tilde E (x,t).
\ea
We then find from \eqref{vpce034} and \eqref{vpce036}
\begin{eqnarray}\label{vpce038}
I_2
&=&
\sum_\Ga \textit{coeff}\,
\int\int
K(y)\xi^\kappa \tilde h_{\Gamma}(-y,t)
\del^\mu_x\del^\nu_v \tilde f_0(-y,\xi)dyd\xi,\\
&=&
\sum_\Ga \textit{coeff}\,
\int\int
K(y)\xi^\kappa \tilde h_{\Gamma}(-y,t)
\del^\mu_x\del^\nu_v f_0(-y,\xi)dyd\xi,\nn\\
& &+\sum_\Ga \textit{coeff}\;\eps\delta^{N-2-|\mu|}
\int\int
K(y)\xi^\kappa \tilde h_{\Gamma}(-y,t)
p_\mu(-y/\delta)e^{-\frac{|y|^2}{\delta^2}}
\del^\nu_v g_N(\xi)
dyd\xi,\nn\\
&=&O(1)+
\sum_\Ga \textit{coeff}\;
\eps\delta^{N-2-|\mu|}\int\int
K(y)\xi^\kappa \tilde h_{\Gamma}(-y,t)
p_\mu(-y/\delta)e^{-\frac{|y|^2}{\delta^2}}
\del^\nu_v g_N(\xi)
dyd\xi.\label{vpce038b}
\end{eqnarray}
In the above, we have used that $f_0(x,v)$ is Schwarz together with \eqref{vpce035}, which gives us the bound $\left|\tilde h_\Ga(x,t)\right|\leq C$, to find that the first sum in the right hand side of \eqref{vpce038b} is indeed $O(1)$.
Now we note that for each term in the remaining sum in \eqref{vpce038b} we have
\begin{eqnarray}\label{vpce039}
& &
\int\int
K(y)\xi^\kappa \tilde h_\Gamma(-y,t)
p_\mu(-y/\delta)e^{-\frac{|y|^2}{\delta^2}}
\del^\nu_v g_N(\xi)
dyd\xi\\
& &\hspace{1cm}=
\int \xi^\kappa \del^\nu_v g_N(\xi)d\xi
\int K(y) \tilde h_\Gamma(-y,t)
p_\mu(-y/\delta)e^{-\frac{|y|^2}{\delta^2}}dy,\nn\\
& &\hspace{1cm}=C_{\kappa,\nu,N}\delta^2
\int K(\delta y) \tilde h_\Gamma(-\delta y,t)
p_\mu(-y)e^{-|y|^2}dy,\nn\\
& &\hspace{1cm}=C_{\kappa,\nu,N}\delta
\int K( y) \tilde h_\Gamma(-\delta y,t)
p_\mu(-y)e^{-|y|^2}dy,\nn\\
& &\hspace{1cm}=C_{\kappa,\nu,N}\delta
\int K( y) O(1)
p_\mu(-y)e^{-|y|^2}dy,\nn\\
& &\hspace{1cm}=O(\delta).\nn
\end{eqnarray}
Above, we have again used \eqref{vpce035} to bound $\tilde h_\Ga(x,t)$. Incorporating \eqref{vpce039} into \eqref{vpce038b}, we find
\ba\label{vpce040}
I_2=O(1)+\sum_\Ga O(\delta^{N-1-|\mu|})=O(1),
\ea
where we have used the fact that for each summand we have $|\mu|\leq N-1$.

Now we consider the contribution of $I_1$. Recalling \eqref{vpce034} and the definition of $\tilde f_0(x,v)$, we get
\begin{eqnarray}\label{vpce041}
I_1
&=&
(-1)^{\ga_t}
\int \int
K(y)(\xi\cdot\na_x)^{\ga_t}
\del^{\ga_x}_x\tilde f_0(-y,\xi)dyd\xi,\\
&=&
(-1)^{\ga_t}
\int \int
K(y)(\xi\cdot\na_x)^{\ga_t}
\del^{\ga_x}_x f_0(-y,\xi)dyd\xi\nn\\
& &+(-1)^{\ga_t}
\eps \delta^{N-2}
\int \int
K(y)(\xi\cdot\na_y)^{\ga_t}
\del^{\ga_x}_y\left[\left(\frac{|y|^2}{\delta^2}-1\right)e^{\frac{-|y|^2}{\delta^2}}\right]g_N(\xi)dyd\xi,\nn\\
&=&
O(1)+(-1)^{\ga_t}\eps \delta^{N-2}\int\int K(y)(\xi\cdot\na_y)^{\ga_t}
\del^{\ga_x}_y\left[\left(\frac{|y|^2}{\delta^2}-1\right)e^{\frac{-|y|^2}{\delta^2}}\right]g_N(\xi)dyd\xi.\nn
\end{eqnarray}
Above, we have used that $f_0$ is Schwartz to get the $O(1)$ bound on the first term.

Using \eqref{vpce041} along with \eqref{vpce040} in \eqref{vpce034}, we find for $\ga$ with $|\ga|=N$ that
\begin{eqnarray}\label{vpce042}
\del^\ga_{x,t}\tilde E(0,0)=
O(1)
+
(-1)^{\ga_t}\eps \delta^{N-2}\int\int K(y)(\xi\cdot\na_y)^{\ga_t}
\del^{\ga_x}_y\left[\left(\frac{|y|^2}{\delta^2}-1\right)e^{\frac{-|y|^2}{\delta^2}}\right]g_N(\xi)dyd\xi.
\end{eqnarray}
Using \eqref{vpce042} in \eqref{vpce033}, we get the following for the first component.
\begin{eqnarray}\label{vpce043}\\
& &\frac{d^{N+1}\tilde V^1}{dt^{N+1}}(0,e_1,0)=
O(1)
+
\eps \delta^{N-2}\sum_{\substack{|\ga|=N\\\vec k, \vec j}}\Bigg(
(-1)^{\ga_t} \delta^1_{j_1}\ldots\delta^1_{j_{|\ga_x|}}b_N(\ga,\vec k,\vec j)\nn\\
& &\hspace{6cm}\times
\int\int K^1(y)(\xi\cdot\na_y)^{\ga_t}
\del^{\ga_x}_y\left[\left(\frac{|y|^2}{\delta^2}-1\right)e^{\frac{-|y|^2}{\delta^2}}\right]dyg_N(\xi)d\xi\Bigg),\nn
\end{eqnarray}
We now find a more explicit expression for the integral in the $y$ variable. Recall $K(x)$ is the integral kernel for the $\na\Delta^{-1}$ operator. Meanwhile, one calculates that
\ba
\Delta\left(
e^{-\frac{|x|^2}{\delta^2}}
\right)
=4\delta^{-2}\left(\frac{|x|^2}{\delta^2}-1\right)e^{-\frac{|x|^2}{\delta^2}}.
\ea
We then find that for the multi-indices $\ga$ with $|\ga|=N$, for each $\xi$ in $\R^2$,
\begin{eqnarray}\label{vpce044}
\int
K^1(y)(\xi\cdot\na_y)^{\ga_t}\del^{\ga_x}_y\left[\left(\frac{|y|^2}{\delta^2}-1\right)e^{\frac{-|y|^2}{\delta^2}}\right]dy
&=&
\int
K^1(y)(\xi\cdot\na_y)^{\ga_t}\del^{\ga_x}_y\left[\frac{\delta^2}{4}\Delta_y\left(e^{-\frac{|y|^2}{\delta^2}}\right)\right]dy,\\
&=&
\frac{1}{4}\delta^2
\del_{y_1} \Delta^{-1}_y \left( 
(\xi\cdot\na_y)^{\ga_t}\del^{\ga_x}_y
\Delta_y\left(e^{-\frac{|y|^2}{\delta^2}}\right)
\right)\bigg|_{y=0},\nn\\
&=&
\frac{1}{4}\delta^{-N+1} (\xi\cdot\na_y)^{\ga_t}\del_{y_1}\del^{\ga_x}_y\left(e^{-|y|^2}\right)\bigg|_{y=0},\nn\\
&=&
\delta^{-N+1}w_\ga(\xi),\nn
\end{eqnarray}
where we have $w_\ga(\xi)$ as defined in \eqref{wgadef}. We substitute this in for the corresponding terms appearing in \eqref{vpce043} and find that we have
\begin{eqnarray}\label{vpce045}
\frac{d^{N+1}\tilde V^1}{dt^{N+1}}(0,e_1,0)
&=&O(1)+\eps\delta^{-1}
\int\Bigg(\sum_{\substack{|\ga|=N\\\vec k, \vec j}}(-1)^{\ga_t} \delta^1_{j_1}\ldots\delta^1_{j_{|\ga_x|}}b_N(\ga,\vec k,\vec j)w_\ga(\xi)\Bigg) g_N(\xi) d\xi,\\
&=&O(1)+\eps\delta^{-1}\int p_N(\xi)g_N(\xi)d\xi.\nn
\end{eqnarray}
Here, we have implemented our definition of the polynomial $p_N(\xi)$ given in \eqref{vpce0022}, from Definition \ref{vpceD02}. Further, we recall from Definition \ref{vpceD02} that $g_N(\xi)$ is a gaussian with the property that $\lambda_N:=\int p_N(\xi)g_N(\xi)d\xi$ is nonzero. We then have
\ba\label{vpce046}
\frac{d^{N+1}\tilde V^1}{dt^{N+1}}(0,e_1,0)
=O(1)+\eps\delta^{-1} \lambda_N,
\ea
where $\lambda_N$ is nonzero. It follows that for small $\delta$ we have \eqref{vpce031}, as claimed in the statement of the lemma.
\end{proof}

\end{subsection}
\begin{subsection}{Construction of the counterexample for the Vlasov-Poisson system}

Now that we have considered the effects of perturbations of the initial data on the derivatives of the trajectory map at $t=0$, we are ready to begin with the construction of the initial data to the Vlasov-Poisson system which we can show must have some resulting non-analytic Lagrangian trajectory.

We will define a sequence of functions by making a series of such perturbations. Here we provide a rough sketch of the process. At each stage, say the $N$th stage, we will start with an iterate for our initial data which has a Lagrangian trajectory with many of its initial lower order time derivatives, say order $n$ with $n\leq N$, right where we want them, each of magnitude at least $n^n$. We will perturb off of this iterate to create successive iterates.
Using Lemma \ref{vpceL02}, we ensure that the lower order derivatives stay approximately the same. Using Lemma \ref{vpceL03}, we ensure that some next higher order derivative $n'>N$ has magnitude larger than $n'^{n'}$. We continue the process to develop a sequence of functions $f^n_0(x,v)$ which ultimately converges to smooth initial data which results in a Lagrangian trajectory that has initial $n$th order time derivative of size $n^n$, for infinitely many $n$. This Lagrangian trajectory cannot then be analytic at $t=0$.

Now we proceed with the rigorous construction of the sequence of functions $f^n_0(x,v)$.

\begin{defn}\label{vpceD01}
Here we iteratively define sequences of real numbers $\{\delta_n\}_{n\geq 2}$ and $\{\eps_n\}_{n\geq 2}$ and a sequence of functions $\{f^n_0(x,v)\}_{n\geq 1}$, where each $f^n_0(x,v)$ satisfies the hypotheses required of the function $f_0(x,v)$ in the statement of Lemma \ref{vpceL02}.

Let us define
\ba
f^1_0(x,v):=e^{-|x|^2-|v|^2}.
\ea
Observe that $f^1_0(x,v)$ trivially satisfies the hypotheses required of $f_0(x,v)$ in the statement of Lemma \ref{vpceL02}.

Let us consider a fixed $n\geq 2$, assuming we have already defined $f^{n-1}_0(x,v)$, and that it satisfies the hypotheses required of $f_0(x,v)$ in the statement of Lemma \ref{vpceL02}. Let us then take $f_0(x,v)$ to be $f^{n-1}_0(x,v)$ and $N$ to be $n$ in the statement of the lemma. For $0\leq\eps\leq \eps_{0,n}$ and $0<\delta<1$, for $\eps_{0,n}>0$ as given by Lemma \ref{vpceL02}, the corresponding perturbation $f^0_{\eps,\delta,n}(x,v)$ defined in the statement of the lemma is then
\ba\label{vpce4602}
f^0_{\eps,\delta,n}(x,v)=f^{n-1}_0(x,v)+\eps F_{\delta,n}(x,v).
\ea
We choose a positive $\eps_n\leq \eps_{0,n}$ so that we have

(i)
\ba\label{vpce460201}
C_n\eps_n\leq 2^{-n},
\ea
where $C_n$ is the constant appearing in the right hand sides of \eqref{vpce004} and \eqref{vpce005}, and so that 

(ii)
\ba\label{vpce460202}
C'_{n,n}\eps_n \leq 2^{-n},
\ea
in which $C'_{n,n}$ is the constant appearing in the right hand side of the bound \eqref{vpce0461} in the case $N=M=n$.

Now we proceed with the definition of $\delta_n$ for the given $n\geq 2$.

Suppose $n$ is odd. In this case, consider the statement of Lemma \ref{vpceL03}, in which we take $\eps$ to be $\eps_n$, $f_0(x,v)$ to be $f^{n-1}_0(x,v)$, and $N$ to be $n$. Now, we choose a positive $\delta_n <1$ small enough to ensure that the following holds:
\ba\label{vpce46021}
C_{n,\eps_n}\delta^{-1}_n
\geq(n+1)^{n+1},
\ea
in which $C_{n,\eps_n}$ is precisely the constant $C_{N,\eps}$ appearing in the right hand side of \eqref{vpce031}. If $n$ is even, we simply take $\delta_n$ to be some positive number less than $1$.

We then define $f^n_0(x,v)$ to be the perturbed initial datum $f^0_{\eps_n,\delta_n,n}(x,v)$. In other words, we define
\ba\label{vpce4603}
f^n_0(x,v):=f^{n-1}_0(x,v)+\eps_n F_{\delta_n,n}(x,v),
\ea
for $F_{\delta_n,n}(x,v)$ as defined in Definition \ref{vpceD02}. Note that $f^n_0(x,v)$ then satisfies the hypotheses required of the function $f_0(x,v)$ in the statement of Lemma \ref{vpceL02}.

By iteratively applying the above process for each $n=2,3,\ldots$, we thus obtain the definition of  $\eps_n$, $\delta_n$, and $f^n_0(x,v)$ for all $n\geq 2$. For each $n\geq 1$, we denote the solution to the Vlasov-Poisson system with initial data $f^n_0(x,v)$ by $f^n(x,v,t)$, with electric field $E^n(x,t)$ and Lagrangian trajectory map $\left(X^n(\zeta,t),V^n(\zeta,t)\right)$.
\end{defn}

Now we check that for certain integers $n$, all the successive iterates $V^N(\zeta,t)$ maintain the property that their $n$th time derivatives have magnitude at least of the order $n^n$ at $\zeta=(0,e_1)$, $t=0$.

\begin{lem}\label{vpceL05}
Consider even $n\geq4$. For all $N\geq n-1$, with $V^N(\zeta,t)$ as defined by Definition \ref{vpceD01}, we have the following inequality.
\ba\label{vpce4604}
\left|
\frac{d^n}{dt^n}V^N(0,e_1,0)
\right|
\geq
\frac{1}{2}n^n.
\ea
\end{lem}
\begin{proof}
For an arbitrary even $n\geq 4$, let us consider $N\geq n-1$. Observe
\ba\label{vpce4605}
\frac{d^n}{dt^n}V^N(0,e_1,0)=
\frac{d^n}{dt^n}V^{n-1}(0,e_1,0)+\left(
\frac{d^n}{dt^n}V^N(0,e_1,0)
-\frac{d^n}{dt^n}V^{n-1}(0,e_1,0)
\right)
\ea
Note
\ba\label{vpce4606}
\left|
\frac{d^n}{dt^n}V^N(0,e_1,0)
-\frac{d^n}{dt^n}V^{n-1}(0,e_1,0)
\right|
\leq
\sum^N_{i=n}\left|
\frac{d^n}{dt^n}V^i(0,e_1,0)
-\frac{d^n}{dt^n}V^{i-1}(0,e_1,0)
\right|,
\ea
Meanwhile, for any pair of positive integers $l$ and $k$ with $l\leq k$, in view of Lemma \ref{vpceL02} and the definition of the iterates $V^n$ in Definition \ref{vpceD01}, we find
\ba\label{vpce4607}
\left|
\frac{d^l}{dt^l}V^k(0,e_1,0)
-\frac{d^l}{dt^l}V^{k-1}(0,e_1,0)
\right|
\leq 2 C_k \eps_k,
\ea
where $C_k$ is the constant $C_N$ in the statement of Lemma \ref{vpceL02}, in the case $N=k$. Using this in \eqref{vpce4606}, together with \eqref{vpce460201} we find for any $N\geq n$ that
\begin{eqnarray}\label{vpce4608}
\left|
\frac{d^n}{dt^n}V^N(0,e_1,0)
-\frac{d^n}{dt^n}V^{n-1}(0,e_1,0)
\right|
&\leq&
\sum^N_{i=n}2C_i\eps_i,\\
&\leq&
\sum^N_{i=n} 2^{-i+1},\nn\\
&\leq&1.\nn
\end{eqnarray}
Now, recalling Definition \ref{vpceD01}, in view of \eqref{vpce46021} together with Lemma \ref{vpceL03}, we find
\ba
\left|\frac{d^n}{dt^n}V^{n-1}(0,e_1,0)\right|
\geq C_{n-1,\eps_{n-1}}\delta^{-1}_{n-1}\geq n^n.
\ea

Using this bound together with \eqref{vpce4608} in \eqref{vpce4605}, we find that
\ba\label{vpce4609}
\left|
\frac{d^n}{dt^n}V^N(0,e_1,0)
\right|
\geq n^n-1.
\ea
The bound \eqref{vpce4604} then easily follows for even $n\geq 4$.
\end{proof}

With the following lemma, we verify that the sequence of $f^n_0(x,v)$ indeed converges to a smooth function and we record bounds that will be useful in showing that the limiting function satisfies the properties we will need it to satisfy.

\begin{lem}\label{vpceL06}
We have the existence of the limit $\lim_{n\to\infty}f^n_0(x,v)$
for every $(x,v)$ in $\R^2\times\R^2$. For each $(x,v)$ let us denote the limit by $f^\sharp_0(x,v)$. We then have that the function $f^\sharp_0(x,v)$ is Schwartz and nonnegative, and for any $m\geq 0$
\ba\label{vpce4610}
\lim_{n\to\infty}
\left\lVert
f^\sharp_0-f^n_0
\right\rVert_{C^m}=0.
\ea

Moreover, for any multi-index $\al$ and integer $l\geq 0$ we have the uniform bound for $n\geq 1$,
\ba\label{vpcef06}
(1+|x|^2+|v|^2)^{l/2}
\left|\del^\al_{x,v}f^n_0(x,v)\right|\leq C_{\al,l},
\ea
as well as the following bound for any multi-index $\ga$, uniform in $n\geq 1$,
\ba\label{vpcef07}
\left|\del^\ga_{x,t}E^n(x,0)\right|\leq C_\ga.
\ea
\end{lem}
\begin{proof}
First, we prove $f^n_0(x,v)$ converges to a continuous function as $n$ tends to infinity. Notice from the definition of the $f^n_0(x,v)$ in Definition \ref{vpceD01} we may write for $n\geq 1$
\ba\label{vpce4611}
f^n_0(x,v)=f^1_0(x,v)+\sum^n_{i=2}\eps_i F_{\delta_i,i}(x,v).
\ea
Using Lemma \ref{vpceL07} and \eqref{vpce460202} we also have the bound for all $i\geq 2$
\ba\label{vpce4612}
\eps_i\left| F_{\delta_i,i}(x,v)\right|
\leq  C'_{i,i}\eps_i\leq 2^{-i},
\ea
where $C'_{i,i}$ is the constant in the right hand side of \eqref{vpce0461} in the case $N=M=i$. Since this upper bound sums to a finite quantity, an application of the Weierstrass M-test verifies that we have uniform convergence of the sum
\ba\label{vpce4613}
\sum^\infty_{i=2}\eps_i F_{\delta_i,i}(x,v).
\ea
Since each summand is continuous and the convergence is uniform, it follows that the sum itself is continuous. Thus, we have convergence of the following limit
\ba\label{vpce4614}
\lim_{n\to\infty}f^n_0(x,v)=f^1_0(x,v)+\sum^\infty_{i=2}\eps_i F_{\delta_i,i}(x,v),
\ea
to a continuous function. We define $f^\sharp_0(x,v)$ to be the resulting function. Note that $f^\sharp_0(x,v)$ is nonnegative since each of the $f^n_0(x,v)$ is positive.

Now we show that $f^\sharp_0(x,v)$ is smooth. Let us consider an arbitrary multi-index $\al=(\al_x,\al_v)$ and define $k:=|\al|$. For $n\geq k+1$ we write
\ba\label{vpce4615}
f^n_0(x,v)=f^{k+1}_0(x,v)+\sum^n_{i=k+2}\eps_i F_{\delta_i,i}(x,v).
\ea
We have
\ba\label{vpce4616}
\del^\al_{x,v}f^n_0(x,v)=\del^\al_{x,v}f^{k+1}_0(x,v)+\sum^n_{i=k+2}\eps_i \del^\al_{x,v}F_{\delta_i,i}(x,v).
\ea
For any $i\geq k+2$, by applying Lemma \ref{vpceL07} and \eqref{vpce460202}, we get the bound
\ba\label{vpce4617}
\eps_i\left| \del^\al_{x,v}F_{\delta_i,i}(x,v)\right|
\leq  C'_{i,i}\eps_i\leq 2^{-i}.
\ea
Thus by using the Weierstrass M-test we get that the following sum converges uniformly to a continuous function:
\ba\label{vpce4618}
\sum^{\infty}_{i=k+2}\eps_i \del^\al_{x,v}F_{\delta_i,i}(x,v).
\ea
From this, together with the fact that $\del^\al_{x,v}f^{k+1}_0(x,v)$ is continuous, we find that we have uniform convergence of the following limit to a continuous function:
\ba\label{vpce4619}
\lim_{n\to\infty}\del^\al_{x,v}f^n_0(x,v).
\ea
Since we have uniform convergence for the various partial derivatives of $f^n_0(x,v)$, it is not difficult to verify that the corresponding partial derivatives of $f^\sharp_0(x,v)$ then exist, and that for each multi-index $\al$
\ba\label{vpce4620}
\lim_{n\to\infty}\del^\al_{x,v}f^n_0(x,v)=\del^\al_{x,v}f^\sharp_0(x,v).
\ea
This implies the partials of $f^\sharp_0(x,v)$ are all continuous, so we conclude $f^\sharp_0(x,v)$ is smooth. This also gives us \eqref{vpce4610}.

Now we check that we have the bound \eqref{vpcef06}. Consider such $\al$ and $l$ and define $m:=\max(l,|\al|+2)$. Let us consider an $n$ larger than $m$. We write
\ba\label{vpce4622}
f^n_0(x,v)
=
f^{m-1}_0(x,v)+\sum^n_{i=m}\eps_i F_{\delta_i,i}(x,v).
\ea
Then we observe
\begin{eqnarray}\label{vpce4623}
\left(1+|x|^2+|v|^2\right)^{l/2}\del^\al_{x,v}
f^n_0(x,v)
&=&
\left(1+|x|^2+|v|^2\right)^{l/2}\del^\al_{x,v}
f^{m-1}_0(x,v)\\
& &+\sum^n_{i=m}\eps_i 
\left(1+|x|^2+|v|^2\right)^{l/2}\del^\al_{x,v}F_{\delta_i,i}(x,v).\nn
\end{eqnarray}
It is easy to see that $f^{m-1}_0(x,v)$ is Schwartz, so that we have for some constant $C_m$
\ba\label{vpce46231}
\left(1+|x|^2+|v|^2\right)^{l/2}\left|\del^\al_{x,v}
f^{m-1}_0(x,v)\right|
\leq C_m.
\ea
Now we observe that for each $i\geq m$, noting that we then also have both $i\geq l$ and $i\geq|\al|+2$, an application of Lemma \ref{vpceL07} gives us that
\ba\label{vpce4624}
\left(1+|x|^2+|v|^2\right)^{l/2}
\left|\del^\al_{x,v}F_{\delta_i,i}(x,v)\right|
\leq
\left(1+|x|^2+|v|^2\right)^{i/2}\left|\del^\al_{x,v}F_{\delta_i,i}(x,v)\right|
\leq
C'_{i,i}.
\ea
Using this in combination with \eqref{vpce460202}, we then find
\begin{eqnarray}\label{vpce4625}
\sum^n_{i=m}\eps_i 
\left(1+|x|^2+|v|^2\right)^{l/2}\left|\del^\al_{x,v}F_{\delta_i,i}(x,v)
\right|
&\leq&
\sum^n_{i=m}C'_{i,i}\eps_i,\\
&\leq&\sum^n_{i=m}2^{-i},\nn\\
&\leq& 1.\nn
\end{eqnarray}
Using this bound together with \eqref{vpce46231} in \eqref{vpce4623}, we get a bound of the form \eqref{vpcef06} for all $n>m$. To produce a satisfactory bound for all the $n\leq m$ is trivial, as we only need to provide bounds for a finite number of functions, each of which is obviously Schwartz. By taking the limit as $n$ tends to infinity in \eqref{vpcef06}, and using \eqref{vpce4620}, we also get the bounds which show that $f^\sharp_0(x,v)$ is Schwartz.

Now we check \eqref{vpcef07}. Consider a multi-index $\ga$ as in the claimed bound. In view of Definition \ref{vpceD01}, in particular \eqref{vpce460201}, together with Lemma \ref{vpceL02}, it is not difficult to verify for any multi-index $\ga$, and any $n\geq |\ga|+1$,
\begin{eqnarray}\label{vpcef13}
\left|
\del^\ga_{x,t} E^n(x,0)
\right|
&\leq&
\left|
\del^\ga_{x,t} E^{|\ga|+1}(x,0)
\right|
+
\sum^n_{i=|\ga|+2}
\left|
\del^\ga_{x,t} E^i(x,0)-\del^\ga_{x,t}E^{i-1}(x,0)
\right|,\\
&\leq&
C'_\ga+\sum^n_{i=|\ga|+2}C_i\eps_i,\nn\\
&\leq& C'_\ga+\sum^n_{i=|\ga|+2}2^{-i},\nn\\
&\leq& C'_\ga+1.\nn
\end{eqnarray}
Above, the constants $C_i$ refer to the constant $C_N$ that appears in the statement of Lemma \ref{vpceL02} in the case $N=i$. For the bound on the first term in the right hand side of \eqref{vpcef13} we used Lemma \ref{vpceL01}. This gives a uniform bound on the desired quantity for $n>|\ga|$. It remains to bound the $\del^\ga_{x,t} E^n(x,0)$ for $n\leq |\ga|$, but there are only finitely many for each given $\ga$, all of which are easily handled by applying Lemma \ref{vpceL01}. This concludes the proof of the bound \eqref{vpcef07} and that of the lemma.
\end{proof}
\begin{defn}\label{vpceD03}
For $f^\sharp_0(x,v)$ as in the statement of Lemma \ref{vpceL06}, let us define the corresponding solution to the Vlasov-Poisson equation with initial data $f^\sharp_0(x,v)$ to be $f^\sharp(x,v,t)$. For the corresponding electric field we denote $E^\sharp(x,t)$ and for the corresponding Lagrangian trajectory map we denote $(X^\sharp(\zeta,t),V^\sharp(\zeta,t))$.
\end{defn}
\begin{rem}\label{globalrmk}
In the above definition, we again use the result due to Ukai and Okabe in \cite{uo} for global existence of solutions to the $2$D Vlasov-Poisson system for initial data with sufficiently rapid decay.
\end{rem}
Finally, we verify that the time derivatives of $V^\sharp(0,e_1,t)$ are approximated by the time derivatives of the $V^n(0,e_1,t)$ at $t=0$. This will allow us to conclude that we preserve the behavior we care about, namely the growth of the time derivatives of the Lagrangian trajectory.
\begin{lem}\label{vpceL04}
For $V^n(\zeta,t)$ as given in Definition \ref{vpceD01} for $n\geq 1$ and $V^\sharp(\zeta,t)$, defined in Definition \ref{vpceD03}, for any $k\geq 0$,
\ba\label{vpce047}
\lim_{n\to \infty}\frac{d^k V^n}{dt^k}(0,e_1,0)=\frac{d^k V^\sharp}{dt^k}(0,e_1,0).
\ea
\end{lem}
\begin{proof}
First we prove that for any multi-index $\ga$ we have
\ba\label{vpce048}
\lim_{n\to \infty}\del^\ga_{x,t} E^n(x,0)=\del^\ga_{x,t} E^\sharp(x,0).
\ea
We prove this by induction. Note for the base case that
\ba\label{vpce049}
E^n(x,0)=\int\int K(y) f^n_0(x-y,\xi)dyd\xi.
\ea
Note for each fixed $x$ the function of $(y,\xi)$ given by $K(y)f^n_0(x-y,\xi)$ tends to $K(y)f^\sharp_0(x-y,\xi)$ pointwise almost everywhere. By using the bound \eqref{vpcef06} from Lemma \ref{vpceL06}, we find that we can bound the $K(y)f^n_0(x-y,\xi)$ uniformly in $n$ by an integrable function of $(y,\xi)$. Thus we find from the Lebesgue dominated convergence theorem that
\ba\label{vpce050}
\lim_{n\to\infty}E^n(x,0)=\int\int K(y)f^\sharp_0(x-y,\xi)dyd\xi
= E^\sharp(x,0).
\ea
Assume that for some integer $m\geq 0$, \eqref{vpce048} holds for any multi-index $\ga$ with $|\ga|\leq m$. Now consider $\ga$ such that $|\ga|=m+1$. Using the formula \eqref{vpceA012}, we note that
\begin{eqnarray}\label{vpce051}
\del^\ga_{x,t} E^n(x,0)
&=&
\int\int
K(y)(-1)^{\ga_t}(\xi\cdot\na_x)^{\ga_t}\del^{\ga_x}_x f^n_0(x-y,\xi)dyd\xi\label{vpce051b}\\
& &+
\sum \textit{coeff}\,\int\int
K(y) \xi^\kappa
\left(
\prod^r_{i=1}\del^{\eta_i}_{x,t} E^n(x-y,0)
\right)
\del^\mu_x\del^\nu_v f^n_0(x-y,\xi)dyd\xi,\nn
\end{eqnarray}
where each $|\eta_i|\leq |\ga|-1=m$, $|\mu|\leq m$, and $|\mu|+|\nu|\leq m+1$. Note then by the induction hypothesis we have the following pointwise limit for any $x$
\ba\label{vpce052}
\lim_{n\to\infty}
\left(
\prod^r_{i=1}\del^{\eta_i}_{x,t} E^n(x,0)
\right)
=
\prod^r_{i=1}\del^{\eta_i}_{x,t} E^\sharp(x,0).
\ea
Meanwhile, from Lemma \ref{vpceL06} we know
\ba\label{vpce053}
\lim_{n\to\infty}
\left\lVert
f^\sharp_0-f^n_0
\right\rVert_{C^{m+1}}=0.
\ea
Moreover, for any fixed $x$, we may bound all the expressions appearing in the integrals in the right hand side of \eqref{vpce051b} by integrable functions, independently of $n$, by using the bounds \eqref{vpcef06} and \eqref{vpcef07}, from Lemma \ref{vpceL06}. So from the Lebesgue dominated convergence theorem we obtain
\begin{eqnarray}\label{vpce054}
\lim_{n\to\infty}\del^\ga_{x,t} E^n(x,0)
&=&
\lim_{n\to\infty}\Bigg(
\int\int
K(y)(-1)^{\ga_t}(\xi\cdot\na_x)^{\ga_t}\del^{\ga_x}_x f^n_0(x-y,\xi)dyd\xi\\
& &+
\sum \textit{coeff}\,\int\int
K(y) \xi^\kappa
\left(
\prod^r_{i=1}\del^{\eta_i}_{x,t} E^n(x-y,0)
\right)
\del^\mu_x\del^\nu_v f^n_0(x-y,\xi)dyd\xi\Bigg),\nn\\
&=&
\int\int
K(y)(-1)^{\ga_t}(\xi\cdot\na_x)^{\ga_t}\del^{\ga_x}_x f^\sharp_0(x-y,\xi)dyd\xi\nn\\
& &+
\sum \textit{coeff}\,\int\int
K(y) \xi^\kappa
\left(
\prod^r_{i=1}\del^{\eta_i}_{x,t} E^\sharp(x-y,0)
\right)
\del^\mu_x\del^\nu_v f^\sharp_0(x-y,\xi)dyd\xi,\nn\\
&=&
\del^\ga_{x,t} E^\sharp(x,0),\nn
\end{eqnarray}
where we have again used the formula \eqref{vpceA012} for solutions of the Vlasov-Poisson equation. This concludes the induction step and thus the proof that \eqref{vpce048} holds for all $\ga$. Now we prove \eqref{vpce047} by induction on $k$. Note the base case $k=0$ is true. Assume for some integer $m\geq 0$ that \eqref{vpce047} holds for all $k\leq m$. Then we compute
\begin{eqnarray}\label{vpce055}
\frac{d^{m+1}V^n}{dt^{m+1}}(0,e_1,0)
&=&
\frac{d^m}{dt^m}\left(
E^n\left(X^n(0,e_1,t),t\right)
\right)\big|_{t=0},\\
&=&
\sum_{\substack{0\leq|\ga|\leq m\\\vec{k},\vec{j}}}\textit{coeff}\left(
\prod^{|\ga|}_{i=1}\del^{k_i}_t X^n_{j_i}(0,e_1,0)
\right)
\del^\ga_{x,t} E^n(0,0).\label{vpce055b}
\end{eqnarray}
It follows from \eqref{vpce048} together with the fact that $\frac{d X^n}{dt}=V^n$ for each $n$ and the induction hypothesis that as $n$ tends to infinity, the expression in \eqref{vpce055b} tends to the analogous expression with $\sharp$ in place of $n$. From this, we find
\ba\label{vpce056}
\lim_{n\to\infty}
\frac{d^{m+1}V^n}{dt^{m+1}}(0,e_1,0)
=
\sum_{\substack{0\leq|\ga|\leq m\\\vec{k},\vec{j}}}\textit{coeff}\left(
\prod^{|\ga|}_{i=1}\del^{k_i}_t X^\sharp_{j_i}(0,e_1,0)
\right)
\del^\ga_{x,t} E^\sharp(0,0)
=
\frac{d^{m+1}V^\sharp}{dt^{m+1}}(0,e_1,0),
\ea
using the equivalent of \eqref{vpcef02} for $\frac{dV^\sharp}{dt}$. This concludes the proof of the lemma.
\end{proof}
\begin{cor}\label{finalcorVP}
The Lagrangian trajectory $(X^\sharp(0,e_1,t),V^\sharp(0,e_1,t))$ is not analytic at $t=0$.
\end{cor}
\begin{proof}
First, consider any even $n\geq4$. Then, by Lemma \ref{vpceL04}, there exists $N$ such that for all $m\geq N$, \ba\label{vpce057}
\left|
\frac{d^n V^\sharp}{dt^n}(0,e_1,0)-
\frac{d^n V^m}{dt^n}(0,e_1,0)
\right|\leq1.
\ea
Now observe that with Lemma \ref{vpceL05}, if we take $m\geq n-1$, we have in addition
\ba\label{vpce058}
\left|\frac{d^n V^m}{dt^n}(0,e_1,0)
\right|\geq\frac{1}{2}n^n.
\ea
Thus for some constant $C>0$, we have the following for arbitrary even $n\geq4$:
\ba\label{vpce059}
\left|
\frac{d^n V^\sharp}{dt^n}(0,e_1,0)
\right|\geq \frac{1}{2}n^n-1\geq Cn^n.
\ea
Since this holds for an infinite sequence of $n$, it easily follows that $V^\sharp(0,e_1,t)$ cannot be analytic at $t=0$.
\end{proof}
We summarize the above results with the statement of the main theorem.
\begin{theo}\label{maintheoVP}
There exists $C^\infty$ initial data $f_0(x,v)$ for the system \eqref{000vp}-\eqref{001c} with a global $C^\infty$ solution $f(x,v,t)$ which has the following property. For the phase-space trajectory map $(X(\zeta,t),V(\zeta,t))$ solving the Hamiltonian formulation of the system, \eqref{002}, for $E(x,t)$ as in \eqref{003}, there is some trajectory $(X(\zeta_0,t),V(\zeta_0,t))$ which is not analytic in time at $t=0$.
\end{theo}
\begin{proof}
We simply take the initial data $f_0$ to be the $f^\sharp_0$ given by Lemma \ref{vpceL06}. In view of Definition \ref{vpceD03} and Remark \ref{globalrmk}, we note that we then have a global smooth solution, and from Corollary \ref{finalcorVP}, we find that a particular trajectory is not analytic at $t=0$.
\end{proof}
\end{subsection}
\end{section}

\bibliographystyle{plain}
\bibliography{mybib2}
\end{document}